\documentclass{article}%
\usepackage{amsmath}
\usepackage{amsfonts}
\usepackage{amssymb}
\usepackage{graphicx}%
\providecommand{\U}[1]{\protect\rule{.1in}{.1in}}
\newtheorem{theorem}{Theorem}

\newtheorem{corollary}[theorem]{Corollary}

\newtheorem{definition}[theorem]{Definition}
\newtheorem{example}[theorem]{Example}

\newtheorem{lemma}[theorem]{Lemma}

\newtheorem{proposition}[theorem]{Proposition}

\newenvironment{proof}[1][Proof]{\noindent\textbf{#1.} }{\ \rule{0.5em}{0.5em}}

\newcommand{\BA}{\mathbb{BA}}
\newcommand{\BI}{\mathbb{BI}}
\newcommand{\KL}{\mathbb{PKA}}
\newcommand{\BZ}{\mathbb{BZL}}
\newcommand{\PBZs}{\mathbb{PBZL}^{\ast }}
\newcommand{\OL}{\mathbb{OL}}
\newcommand{\OML}{\mathbb{OML}}
\newcommand{\AOL}{\mathbb{AOL}}

\newcommand{\PBZ}{PBZ$^{\ast }$}

\def\N{{\mathbb N}}

\def\I{{\mathbb I}}

\def\V{{\mathbb V}}
\def\C{{\mathbb C}}
\def\D{{\mathbb D}}
\def\H{{\mathrm H}}
\def\S{{\mathrm S}}
\def\P{{\mathrm P}}
\def\1{\textcircled{1}}
\def\2{\textcircled{2}}

\begin{document}

\title{\vspace*{-30pt}Ordinal and Horizontal Sums\\ Constructing PBZ$^{\ast}$--Lattices}
\author{Roberto GIUNTINI, Claudia MURE\c{S}AN, Francesco PAOLI\\ {\small University of Cagliari}\\{\small giuntini@unica.it; c.muresan@yahoo.com; paoli@unica.it}}
\date{\today }
\maketitle

\begin{abstract} \PBZ --lattices are algebraic structures related to quantum logics, which consist of bounded lattices endowed with two kinds of complements, named {\em Kleene} and {\em Brouwer}, such that the Kleene complement satisfies a weakening of the orthomodularity condition and the De Morgan laws, while the Brouwer complement only needs to satisfy the De Morgan laws for the pairs of elements with their Kleene complements. \PBZ --lattices form a variety $\PBZs $, which includes the variety $\OML $ of orthomodular lattices (considered with an extended signature, by letting their two complements coincide) and the variety $V(\AOL )$ generated by the class $\AOL $ of antiortholattices.

We investigate the congruences of antiortholattices, in particular of those obtained through certain ordinal sums and of those whose Brower complements satisfy the De Morgan laws, infer characterizations for their subdirect irreducibility and prove that even the lattice reducts of antiortholattices are directly irreducible. Since the two complements act the same on the lattice bounds in all \PBZ --lattices, we can define the horizontal sum of any nontrivial \PBZ --lattices, obtained by glueing them at their smallest and at their largest elements; a horizontal sum of two nontrivial \PBZ --lattices is a \PBZ --lattice exactly when at least one of its summands is an orthomodular lattice. We investigate the algebraic structures and the congruence lattices of these horizontal sums, then the varieties they generate.

We obtain a relative axiomatization of the variety $V(\OML \boxplus \AOL )$ generated by the horizontal sums of nontrivial orthomodular lattices with nontrivial antiortholattices w.r.t. $\PBZs $, as well as a relative axiomatization of the join of varieties $\OML \vee V(\AOL )$ w.r.t. $V(\OML \boxplus \AOL )$.


\end{abstract}

\section{Introduction}

Let $\mathbf{H}$ be complex separable Hilbert space, and let $\mathcal{E}\left(  \mathbf{H}\right)  $ be the set of all \emph{effects} of $\mathbf{H}$,
i.e., the set of all positive linear operators of $\mathbf{H}$ that are
bounded by the identity operator $\mathbb{I}$. Within the unsharp approach to
quantum logic, it has been argued at length (see e.g. \cite[Ch. 4]{RQT}) that
effects are a more adequate mathematical counterpart than projection operators
of the notion of quantum \emph{event}, in that the latter do not form the
largest set of operators that can be assigned a probability value according to
the Born rule. However, if we order the effects in $\mathcal{E}\left(
\mathbf{H}\right)  $ under the natural order determined by the set of all
density operators of $\mathbf{H}$ via the trace functional --- namely, if we
let, for all $E,F\in\mathcal{E}\left(  \mathbf{H}\right)  $,
\begin{align*}
E\leq F\text{ iff }  &  \text{for all density operators }\rho\text{ of
}\mathbf{H}\text{,}\\
&  Tr\left(  \rho E\right)  \leq Tr\left(  \rho F\right)  \text{,}%
\end{align*}
the poset $\left(  \mathcal{E}\left(  \mathbf{H}\right)  ,\leq\right)  $ has
the drawback of failing, in general, to be a lattice.

On the other hand, consider the structure$$
\mathbf{E}\left(  \mathbf{H}\right)  =\left(  \mathcal{E}\left(
\mathbf{H}\right)  ,\wedge_{s},\vee_{s},^{\prime},^{\sim},\mathbb{O}%
,\mathbb{I}\right)  \text{,}$$where:

\begin{itemize}
\item $\wedge_{s}$ and $\vee_{s}$ are the meet and the join, respectively, of
the \emph{spectral ordering} $\leq_{s}$ so defined for all $E,F\in
\mathcal{E}\left(  \mathbf{H}\right)  $:
\[
E\leq_{s}F\,\,\,\text{iff}\,\,\,\forall\lambda\in\mathbb{R}:\,\,M^{F}%
(\lambda)\leq M^{E}(\lambda),
\]
where for any effect $E$, $M^{E}$ is the unique spectral family \cite[Ch.
7]{Kr} such that $E=\int_{-\infty}^{\infty}\lambda\,dM^{E}(\lambda)$ (the
integral is here meant in the sense of norm-converging Riemann-Stieltjes sums
\cite[Ch. 1]{Strocco});

\item $\mathbb{O}$ and $\mathbb{I}$ are the null and identity operators, respectively;

\item $E^{\prime}=\mathbb{I}-E$ and $E^{\sim}=P_{\ker\left(  E\right)  }$ (the
projection onto the kernel of $E$).
\end{itemize}

The operations in $\mathbf{E}\left(  \mathbf{H}\right)  $ are well-defined. The spectral ordering is indeed a lattice ordering \cite{Ols, deG} that coincides with the natural order when both orderings are restricted to the set of projection operators of the same Hilbert space.

The papers \cite{GLP}, \cite{PBZ2} and \cite{rgcmfp} contain the beginnings of
an algebraic investigation of a variety of lattices with additional structure,
the variety {$\mathbb{PBZL}^{\ast}$ of }\emph{PBZ*-lattices}. A PBZ*-lattice{
can be viewed as an abstraction from this }concrete physical model, much in
the same way as an orthomodular lattice can be viewed as an abstraction from
its substructure consisting of projection operators only. The faithfulness of
PBZ*-lattices to the physical model whence they stem is further underscored by
the fact that they reproduce at an abstract level the \textquotedblleft
collapse\textquotedblright\ of several notions of \emph{sharp physical
property} that can be observed in $\mathbf{E}\left(  \mathbf{H}\right)  $.

Further motivation for the study of {$\mathbb{PBZL}^{\ast}$ comes from its
universal algebraic properties. For a start, PBZ*-lattices can be seen as a
common generalisation of orthomodular lattices \cite{Beran} and of Kleene
algebras \cite{Kalman} with an additional unary operation. In the lattice of
subvarieties of $\mathbb{PBZL}^{\ast}$, moreover, we happen to encounter many
situations of intrinsic interest in universal algebra: to name a few,
subtractive varieties with equationally definable principal ideals that fail
to be point-regular \cite{BS}; binary discriminator varieties \cite{CHR, BS};
ternary discriminator varieties generated by a single finite non-primal
algebra. }

Regarding their similarity type, \PBZ --lattices have, in addition to their bounded lattice structure, two unary operations, out of which one is a lattice involution, called {\em Kleene complement}, and the second is called {\em Brouwer complement}; the bounded involution lattice reduct of a \PBZ --lattice has to satisfy a weakening of the orthomodularity condition, which is called {\em paraorthomodularity}; and, while the Brouwer complement does not satisfy the De Morgan laws, as the involution does, it is required to satisfy them for all pairs of elements with their Kleene complements; this latter property is called {\em condition $(\ast )$}.

This paper is concerned with the study of {\em ordinal} and \emph{horizontal sums} producing PBZ*-lattices. Informally, the ordinal sum of a lattice ${\bf A}$ with a largest element $1^{\bf A}$ and a lattice ${\bf B}$ with a smallest element $0^{\bf B}$ is a lattice ${\bf A}\oplus {\bf B}$ obtained by glueing ${\bf A}$ and ${\bf B}$ at the $1^{\bf A}$ and $0^{\bf B}$, while the horizontal sum of two non--trivial bounded lattices $\mathbf{L}$ and $\mathbf{M}$ is the non--trivial bounded lattice $\mathbf{L}\boxplus {\bf M}$ obtained by glueing $\mathbf{L}$ and $\mathbf{M}$ at their smallest elements, as well as at their largest elements. If ${\bf H}$ is a non--trivial bounded lattice, ${\bf H}^d$ is the dual of ${\bf H}$ and ${\bf K}$ is a pseudo--Kleene algebra (that is a bounded involution lattice in which any meet of an element and its involution is smaller than any join of an element and its involution), then the ordinal sum ${\bf H}\oplus {\bf K}\oplus {\bf H}^d$ can be organized as an antiortholattice, that is a \PBZ --lattice with no other sharp elements beside $0$ and $1$, with the clear definition for the involution and the trivial Brouwer complement, which takes $0$ to $1$ and all other elements to $0$. Since both complements take $0$ to $1$ and $1$ to $0$, we can define the horizontal sum of two non--trivial \PBZ --lattices $\mathbf{L}$ and $\mathbf{M}$, obtained by defining the Kleene and Brouwer complements on the horizontal sum of the bounded lattice reducts of $\mathbf{L}$ and $\mathbf{M}$ by restriction, that is such that $\mathbf{L}$ and $\mathbf{M}$ become subalgebras of $\mathbf{L}\boxplus {\bf M}$; however, while, in this way, $\mathbf{L}\boxplus {\bf M}$ becomes an algebra of the same similarity type as \PBZ --lattices, it does not become a \PBZ --lattice unless at least one of $\mathbf{L}$ and $\mathbf{M}$ is an orthomodular lattice (organized as a \PBZ --lattice by letting its Brouwer complement equal its involution). We study the algebraic structures and congruences of these glued sums producing \PBZ --lattices. There is a well-developped theory of horizontal sums in the context of orthomodular lattices (see e.g. \cite{Beran, Chajdor, Greechie}), but the case of \PBZ --lattices differs substantially from this particular one, as we learn by examining the congruences, the singleton--generated subalgebras and the sets of sharp and of dense elements of horizontal sums of \PBZ --lattices.

In the final section of this paper, we study the subvarieties generated by horizontal sums of the variety $\PBZs $ of the \PBZ --lattices. We axiomatize the variety $V(\OML \boxplus \AOL )$ generated by the horizontal sums of orthomodular lattices with antiortholattices with respect to $\PBZs $, as well as the varietal join $\OML \vee V(\AOL )$ of the variety of orthomodular lattices with the variety generated by the class of antiortholattices with respect to $V(\OML \boxplus \AOL )$. These results yield an alternate proof for the axiomatization of $\OML \vee V(\AOL )$ relative to $\PBZs $ that we have obtained in \cite{rgcmfp}.

\section{Preliminaries\label{preliminaries}}

We will often use the results in this section without referencing them.

\subsection{Notations for Lattices and Universal Algebras}
\label{ualglat}

We refer the reader to \cite{bur,gralgu} for the following universal algebra notions and to \cite{gratzer} for the lattice--theoretical ones.

We will denote by ${\mathbb{N}}$ the set of the natural numbers and by ${\mathbb{N}}^{\ast}={\mathbb{N}}\setminus\{0\}$. For any class ${\mathbb{C}}$ of algebras of the same type, $V({\mathbb{C}})$ will denote the variety generated by ${\mathbb{C}}$; so $V({\mathbb{C}})=\H \S \P ({\mathbb{C}})$, where $\H $, $\S $ and $\P $ denote the usual class operators. For any algebra $\mathbf{A}$ and any class operator $O$, $O(\mathbf{A})$ will be shorthand for $O(\{\mathbf{A}\})$ and, whenever we need to specify the type, if $\C $ is a subclass of a variety $\V $ of algebras of the same type and $\mathbf{A}\in \V $, we will denote $O_{\V }(\C )$ and $O_{\V }(\mathbf{A})$ instead of $O(\C )$ and $O(\mathbf{A})$, respectively. The join of varieties will be denoted by $\vee $. We will consider only algebras with a nonempty universe; by \emph{trivial
algebra} we mean a one--element algebra. For brevity, we denote by
$\mathbf{A}\cong\mathbf{B}$ the fact that two algebras $\mathbf{A}$ and
$\mathbf{B}$ of the same type are isomorphic. We make the following
convention: if $\mathbf{A}$ is an algebra, then $A$ will denote the universe
of $\mathbf{A}$, with the exception of partition, equivalence and congruence lattices, that will be designated by their universes; other such exceptions
will be specified later; sometimes, for brevity, algebras will be designated by their set reducts without being specified as such. If $\mathbf{A}$ is a member of a variety $\V $ and $S\subseteq A$, then $\langle S\rangle _{\V }$ or $\langle S\rangle _{\V ,\mathbf{A}}$ will denote the subalgebra of $\mathbf{A}$ generated by $S$, as well as its universe; if $a\in A$, then $\langle \{a\}\rangle _{\V ,\mathbf{A}}$ will simply be denoted by $\langle a\rangle _{\V ,\mathbf{A}}$ or $\langle a\rangle _{\V }$. For any subalgebras $\mathbf{B}$ and $\mathbf{C}$ of an algebra $\mathbf{A}$, $\mathbf{B}\cap\mathbf{C}$ will denote the subalgebra of $\mathbf{A}$ with universe $B\cap C$.

The dual of any (bounded) lattice $\mathbf{L}$ will be denoted by $\mathbf{L}^{d}$. For any bounded lattice $\mathbf{L}$, $\mathrm{At}(\mathbf{L})$ and $\mathrm{CoAt}(\mathbf{L})$ will
denote the set of the atoms and that of the coatoms of $\mathbf{L}$, respectively. For any lattice $\mathbf{L}$ and
any $a,b\in L$, $[a)$ and $(a]$ will be the principal filter, respectively
principal ideal of $\mathbf{L}$ generated by $a$, and $[a,b]=[a)\cap(b]$ will
be the interval of $L$ bounded by $a$ and $b$. For any
$n\in{\mathbb{N}}^{\ast}$, $\mathbf{D}_{n}$ denotes the $n$--element chain, regardless of the algebraic structure with a bounded lattice reduct we consider on it, and $D_{n}$ denotes its universe. We will use Gr\"{a}tzer`s notation for lattices \cite{gratzer}.

$\amalg $ will denote the disjoint union of sets and, for any set $S$, $|S|$
will be the cardinality of $S$, $\mathcal{P}(S)$ will be the set of the
subsets of $S$, $(\mathrm{Part}(S),\wedge,\vee,\{\{x\}:x\in S\},\{S\})$ and
$(\mathrm{Eq}(S),\cap,\vee,\Delta_{S},\nabla_{S})$ will be the bounded
lattices of the partitions and the equivalences of $S$, respectively, and
$eq:\mathrm{Part}(S)\rightarrow\mathrm{Eq}(S)$ will be the canonical lattice
isomorphism; for any $n\in{\mathbb{N}}^{\ast}$ and any $\{S_{1},\ldots
,S_{n}\}\in\mathrm{Part}(S)$, $eq(\{S_{1},\ldots,S_{n}\})$ will simply be
denoted by $eq(S_{1},\ldots,S_{n})$. Also, for any $U\subseteq S^{2}$ and any
$\sigma\in\mathrm{Eq}(S)$, we denote by $U/\sigma=\{(x/\sigma,y/\sigma
):(x,y)\in U\}$.

Let ${\mathbb{V}}$ be a variety of algebras of a similarity type $\tau$ and
$\mathbf{A}\in{\mathbb{V}}$. Then, for any $n\in{\mathbb{N}}^{\ast}$, any
terms $t(x_{1},\ldots,x_{n})$ and $u(x_{1},\ldots,x_{n})$ over $\tau$ with at
most the variables $x_{1},\ldots,x_{n}$ and any $M_{1},\ldots,M_{n}%
\in\mathcal{P}(A)$, we will use the following notation: $\mathbf{A}%
\vDash_{M_{1},\ldots,M_{n}}t(x_{1},\ldots,x_{n})\approx u(x_{1},\ldots,x_{n})$
iff, for all $a_{1}\in M_{1},\ldots,a_{n}\in M_{n}$, $t^{\mathbf{A}}%
(a_{1},\ldots,a_{n})=u^{\mathbf{A}}(a_{1},\ldots,a_{n})$. If the variables in
an equation are not numbered, then, by convention, we consider the set of
these variables ordered by their order of appearance in that equation in its
current writing, from left to right. As usual, if $k\in{\mathbb{N}}^{\ast}$
and, for all $i\in[1,k]$, $\gamma_{i}=(t_{i}(x_{1},\ldots
,x_{n})\approx u_{i}(x_{1},\ldots,x_{n}))$ for some terms $t_{i}(x_{1}%
,\ldots,x_{n})$ and $u_{i}(x_{1},\ldots,x_{n})$ over $\tau$ with at most the
variables $x_{1},\ldots,x_{n}$, then $\mathbf{A}\vDash_{M_{1},\ldots,M_{n}%
}\{\gamma_{1},\ldots,\gamma_{k}\}$ will be short for $\mathbf{A}\vDash
_{M_{1},\ldots,M_{n}}\gamma_{i}$ for all $i\in[1,k]$.

We will denote by $\mathrm{Con}_{{\mathbb{V}}}(\mathbf{A})$ the congruence
lattice of $\mathbf{A}$ (with respect to $\tau$); if, for some $n\in
{\mathbb{N}}^{\ast}$, $\tau$ contains constants $\kappa_{1},\ldots,\kappa_{n}%
$, then we will denote by $\mathrm{Con}_{{\mathbb{V}}\kappa_{1}\ldots
\kappa_{n}}(\mathbf{A})=\{\theta\in\mathrm{Con}_{{\mathbb{V}}}(\mathbf{A}%
):(\forall\,i\in[1,n])\,(\kappa_{i}^{\mathbf{A}}/\theta=\{\kappa
_{i}^{\mathbf{A}}\})\}$. Recall that ${\bf A}$ is subdirectly irreducible in $\V $ iff ${\bf A}$ is trivial or $\Delta _A$ is strictly meet--irreducible in the lattice $\mathrm{Con}_{{\mathbb{V}}}(\mathbf{A})$. For any $U\subseteq A^{2}$, we will denote by
$Cg_{{\mathbb{V}}}(U)$ the congruence of $\mathbf{A}$ (with respect to $\tau$)
generated by $U$; for any $a,b\in A$, the principal congruence
$Cg_{{\mathbb{V}}}(\{\left(  a,b\right)  \})$ will simply be denoted by
$Cg_{{\mathbb{V}}}(a,b)$. For any $S\subseteq A$, the $\tau$--subalgebra of
$\mathbf{A}$ generated by $S$ will be denoted by $\langle S\rangle
_{{\mathbb{V}}}$ and so will its universe. If ${\mathbb{V}}$ is the variety of
bounded lattices, then the index ${\mathbb{V}}$ in the previous notations will
be omitted.

Let $\mathbf{L}=(L,\leq^{\mathbf{L}})$ be a lattice with greatest element
$1^{\mathbf{L}}$, and $\mathbf{M}=(M,\leq^{\mathbf{M}})$ be a lattice with
least element $0^{\mathbf{M}}$. Also, let $\varepsilon$ be the equivalence on
$L\amalg M$ defined by:\vspace*{-7pt}$$\varepsilon =eq(\{\{1^{\mathbf{L}},0^{\mathbf{M}}\}\}\cup\{\{x\}:x\in(L\amalg
M)\setminus\{1^{\mathbf{L}},0^{\mathbf{M}}\}\}),$$and let $L\oplus M=(L\amalg M)/\varepsilon $. Note that $\varepsilon\cap L^{2}=\Delta_{L}\in\mathrm{Con}(\mathbf{L})$, thus
$\mathbf{L}\cong\mathbf{L}/\Delta_{L}=\mathbf{L}/(\varepsilon\cap L^{2})$, and
$\varepsilon\cap M^{2}=\Delta_{M}\in\mathrm{Con}(\mathbf{M})$, thus
$\mathbf{M}\cong\mathbf{M}/\Delta_{M}=\mathbf{M}/(\varepsilon\cap M^{2})$, so
we can identify $\mathbf{L}$ with $\mathbf{L}/(\varepsilon\cap L^{2})=(L/\varepsilon,\leq^{\mathbf{L}}\!\!/\varepsilon)$ and $\mathbf{M}$ with
$\mathbf{M}/(\varepsilon\cap M^{2})=(M/\varepsilon,\leq^{\mathbf{M}%
}\!\!/\varepsilon)$, by identifying $x$ with $x/\varepsilon$ for each $x\in L$
and each $x\in M$; with this identification, we get $1^{\mathbf{L}}=0^{\mathbf{M}}$ and $L\cap M=\{1^{\mathbf{L}}\}=\{0^{\mathbf{M}}\}$. Then the \emph{ordinal sum} of
$\mathbf{L}$ and $\mathbf{M}$ is the lattice $\mathbf{L}\oplus\mathbf{M}=(L\oplus M,\leq^{\mathbf{L}\oplus\mathbf{M}})$, where:\[
\leq^{\mathbf{L}\oplus\mathbf{M}}=\leq^{\mathbf{L}}\cup\leq^{\mathbf{M}}\cup\{(x,y):x\in L,y\in M\}.
\]
For any $\alpha\in\mathrm{Con}(L)$ and $\beta\in\mathrm{Con}(M)$, we let:$$\alpha\oplus\beta=eq((L/\alpha\setminus\{1^{\mathbf{L}}/\alpha\})\cup(M/\beta
\setminus\{0^{\mathbf{M}}/\beta\})\cup\{1^{\mathbf{L}}/\alpha\cup 0^{\mathbf{M}}/\beta\})\in\mathrm{Con}(L\oplus M).$$Clearly, the ordinal sum of bounded lattices and the attendant operation on congruences are both associative operations, and the map $(\alpha,\beta)\mapsto\alpha\oplus\beta$ is a lattice isomorphism from $\mathrm{Con}(\mathbf{L})\times\mathrm{Con}(\mathbf{M})$ to $\mathrm{Con}(\mathbf{L}\oplus\mathbf{M})$.

Let $(\mathbf{L}_i)_{i\in I}$ be a non--empty family of nontrivial bounded lattices, with $\mathbf{L}_i=(L_i,\leq^{\mathbf{L}_i},0^{\mathbf{L}_i},1^{\mathbf{L}_i})$ for all $i\in I$. Also, let $\varepsilon$ be the equivalence on
$\amalg _{i\in I}L_i$ defined by:$$\varepsilon =eq(\{\{0^{\mathbf{L}_i}:i\in I\},\{1^{\mathbf{L}_i}:i\in I\}\}\cup\{\{x\}:x\in\amalg _{i\in I}(L_i\setminus\{0^{\mathbf{L}_i},1^{\mathbf{L}_i}\})\}),$$and let $\boxplus_{i\in I}L_i=(\amalg _{i\in I}L_i)/\varepsilon $. Note that, for all $i\in I$, $\varepsilon\cap L_i^{2}=\Delta_{L_i}\in\mathrm{Con}(\mathbf{L}_i)$, thus $\mathbf{L}_i\cong\mathbf{L}_i/\Delta_{L_i}=\mathbf{L}_i/(\varepsilon\cap L_i^{2})$. For each $i\in I$, we identify $\mathbf{L}_i$ with $\mathbf{L}_i/(\varepsilon\cap L_i^{2})=(L_i/\varepsilon ,\leq ^{\mathbf{L}_i}\!\!/\varepsilon,0^{\mathbf{L}_i}/\varepsilon,1^{\mathbf{L}_i}/\varepsilon)$, by identifying $x$ with $x/\varepsilon$ for each $x\in L_i$. The \emph{horizontal sum} of the family $(\mathbf{L}_i)_{i\in I}$ is the bounded lattice:$$\boxplus_{i\in I}\mathbf{L}_i=(\boxplus_{i\in I}L_i,\leq ^{\boxplus_{i\in I}\mathbf{L}_i},0^{\boxplus_{i\in I}\mathbf{L}_i},1^{\boxplus_{i\in I}\mathbf{L}_i}),$$where $0^{\boxplus_{i\in I}\mathbf{L}_i}=0^{\mathbf{L}_j}$ and $1^{\boxplus_{i\in I}\mathbf{L}_i}=1^{\mathbf{L}_j}$ for each $j\in I$, and $\leq^{\boxplus_{i\in I}\mathbf{L}_i}=\bigcup_{i\in I}\leq^{\mathbf{L}_i}$. If $\alpha _i\in\mathrm{Eq}(L_i)\setminus\left\{\nabla_{L_i}\right\}$ for all $i\in I$, then we denote by $\boxplus _{i\in I}\alpha _i$ the equivalence on $\boxplus _{i\in I}L_i$ defined by:$$\boxplus _{i\in I}\alpha _i=eq(\bigcup_{i\in I}(L_i/\alpha_i\setminus\{0^{\mathbf{L}_i}/\alpha_i,1^{\mathbf{L}_i}/\alpha_i\})\cup \{\bigcup_{i\in I}0^{\mathbf{L}_i}/\alpha_i,\bigcup_{i\in I}1^{\mathbf{L}_i}/\alpha_i\}).$$Note that, for any nontrivial bounded lattice $\mathbf{L}$, $\mathbf{D}_{2}\boxplus\mathbf{L}=\mathbf{L}$ and $\Delta _{D_2}\boxplus \alpha =\alpha $ for any $\alpha \in {\rm Eq}(L)\setminus \{\nabla _L\}$. Clearly, the binary operation $\boxplus$ on nontrivial bounded lattices is associative and commutative, and so is the attendant operation on proper equivalences of the universes of those lattices.

\subsection{Congruences with Singleton Classes and Generated Subalgebras}

\begin{theorem}
\textup{\cite[Corollary 2, p. 51]{gralgu}} The congruence lattice of any algebra is a complete sublattice of the equivalence lattice of its set reduct.\label{cgeq}
\end{theorem}

\begin{corollary}
The congruence lattice of any algebra is a complete sublattice of the congruence lattice of any of its reducts.\label{cgred}
\end{corollary}

\begin{lemma}
\begin{enumerate}
\item \label{maxcg1} If $M$ is a set, $\emptyset\neq S\subseteq M$ and
$\sigma\in\mathrm{Part}(S)$, then $P=\{\pi\in\mathrm{Part}(M):\sigma
\subseteq\pi\}$ and $E=\{\varepsilon\in\mathrm{Eq}(M):\sigma\subseteq
M/\varepsilon\}$ are complete sublattices of $\mathrm{Part}(M)$ and
$\mathrm{Eq}(M)$, respectively, in particular they are bounded lattices.

\item \label{maxcg2} If $\mathbf{A}$ is an algebra from a variety
${\mathbb{V}}$, $\emptyset\neq S\subseteq A$ and $\sigma\in\mathrm{Part}(S)$
is such that the set $C=\{\theta\in\mathrm{Con}_{{\mathbb{V}}}(\mathbf{A}%
):\sigma\subseteq A/\theta\}$ is non--empty, then $C$ is a complete sublattice
of $\mathrm{Con}_{{\mathbb{V}}}(\mathbf{A})$, in particular it is a bounded lattice.

\item \label{maxcg3} Let ${\mathbb{V}} $ be a variety of algebras of a similarity type $\tau$,
$n\in{\mathbb{N}} ^{*}$ and $\kappa_{1},\ldots,\kappa_{n}$ be constants in
$\tau$. If $\mathbf{A}$ is a member of ${\mathbb{V}} $ such that
$\mathrm{Con}_{{\mathbb{V}} \kappa_{1}\ldots\kappa_{n}}(\mathbf{A})$ is
non--empty, then $\mathrm{Con}_{{\mathbb{V}} \kappa_{1}\ldots\kappa_{n}%
}(\mathbf{A})$ is a complete sublattice of $\mathrm{Con}_{{\mathbb{V}}
}(\mathbf{A})$, in particular it is a bounded lattice.\end{enumerate}\label{maxcg}\end{lemma}

\begin{proof}(\ref{maxcg1}) Note that, in the statement of the lemma, by $\subseteq$ we
mean set inclusion, not the partitions ordering, so that, for any $\pi
\in\mathrm{Part}(M)$, $\sigma\subseteq\pi$ means that, for each $x\in S$,
$x/eq(\sigma)=x/eq(\pi)$. $S/\sigma\cup(\{M\setminus S\}\setminus
\{\emptyset\})\in P$, thus $P\neq\emptyset$, hence $E=eq(P)\neq\emptyset$.

If $S=M$, then, for any $\pi\in\mathrm{Part}(M)$, $\sigma\subseteq\pi$ is
equivalent to $\sigma=\pi$, thus, in this case, $P=\{\sigma\}$ and
$E=\{eq(\sigma)\}$, therefore $P$ and $E$ are trivial, thus complete
sublattices of $\mathrm{Part}(M)$ and $\mathrm{Eq}(M)$, respectively.

If $S\subsetneq M$, then, for any $\emptyset\neq\{\pi_{i}:i\in I\}\subseteq P$
and any $j\in I$, the fact that $\sigma\subseteq\pi_{j}$ shows that $\pi
_{j}\setminus\sigma\in\mathrm{Part}(M\setminus S)$, and hence $\bigwedge_{i\in
I}\pi_{i}=\sigma\cup\bigwedge(\pi_{i}\setminus\sigma)\supseteq\sigma$ and
$\bigvee_{i\in I}\pi_{i}=\sigma\cup\bigvee(\pi_{i}\setminus\sigma
)\supseteq\sigma$, therefore $\bigwedge_{i\in I}\pi_{i},\bigvee_{i\in I}%
\pi_{i}\in P$, hence $P$ is a complete sublattice of $\mathrm{Part}(M)$, thus
$E=eq(P)$ is a complete sublattice of $\mathrm{Eq}(M)$.

\noindent(\ref{maxcg2}) $C=E\cap\mathrm{Con}_{{\mathbb{V}}}(\mathbf{A}%
)\subseteq E$, so that, if $C\neq\emptyset$, then, by (\ref{maxcg1}) and
Theorem \ref{cgeq}, for any $\emptyset\neq\{\gamma_{i}:i\in I\}\subseteq
C\subseteq E$, we have $\bigwedge_{i\in I}\gamma_{i},\bigvee_{i\in I}%
\gamma_{i}\in E\cap\mathrm{Con}_{{\mathbb{V}}}(\mathbf{A})=C$, hence $C$ is a
complete sublattice of $\mathrm{Con}_{{\mathbb{V}}}(\mathbf{A})$.

\noindent(\ref{maxcg3}) This is a particular case of (\ref{maxcg2}).\end{proof}

\begin{lemma}\begin{enumerate}
\item\label{cgsgltcls1} If $\V $ is a variety and $(\mathbf{A}_{i})_{i\in I}\subseteq{\mathbb{V}}$ is a non--empty family such that
$\prod_{i\in I}\mathbf{A}_{i}$ has no skew congruences, then $\mathrm{Con}
_{{\mathbb{V}}\kappa_{1}\ldots\kappa_{n}}(\prod_{i\in I}\mathbf{A}%
_{i})=\{\prod_{i\in I}\alpha_{i}:((\forall\,i\in I)\,\left(  \alpha_{i}%
\in\mathrm{Con}_{{\mathbb{V}}\kappa_{1}\ldots\kappa_{n}}(\mathbf{A}%
_{i})\right)  \}\cong\prod_{i\in I}\mathrm{Con}_{{\mathbb{V}}\kappa_{1}\ldots\kappa_{n}}(\mathbf{A}_{i})$.
\item\label{cgsgltcls2} If ${\mathbb{V}}$ is a variety of bounded lattice--ordered structures, then, for all $\mathbf{A},\mathbf{B}\in{\mathbb{V}}$, $\mathrm{Con}_{{\mathbb{V}}01}(\mathbf{A}\times\mathbf{B})=\{\alpha\times\beta:\alpha\in\mathrm{Con}_{{\mathbb{V}}01}(\mathbf{A}),\beta\in\mathrm{Con}_{{\mathbb{V}}01}(\mathbf{B})\}\cong\mathrm{Con}_{{\mathbb{V}}01}(\mathbf{A})\times\mathrm{Con}_{{\mathbb{V}}01}(\mathbf{B})$.\end{enumerate}\label{cgsgltcls}\end{lemma}

\begin{proof} (\ref{cgsgltcls1}) Routine.

\noindent (\ref{cgsgltcls2}) By (\ref{cgsgltcls1}) and the fact that lattice--ordered structures are congruence--dis\-tri\-bu\-tive, thus $\mathbf{A}\times\mathbf{B}$ has no skew congruences \cite{fremck,bj}; see also \cite{eucardbi}.\end{proof}

Let $\V $ be a variety of similar algebras, ${\bf A}$ a member of $\V $, $\theta \in {\rm Con}_{\V }({\bf A})$, $S\subseteq A$, ${\bf B}\in \S _{\V }({\bf A})$, $({\bf A}_i)_{i\in I}$ a non--empty family of members of $\V $ and, for all $i\in I$, $S_i\subseteq A_i$. Then:\begin{enumerate}
\item\label{prodsubalg} $\langle \prod _{i\in I}S_i\rangle _{\V ,\prod _{i\in I}{\bf A}_i}=\prod _{i\in I}\langle S_i\rangle _{\V ,{\bf A}_i}$;
\item\label{subsubalg} $\langle S\cap B\rangle _{\V ,{\bf B}}=\langle S\cap B\rangle _{\V ,{\bf A}}\cap {\bf B}\subseteq \langle S\rangle _{\V ,{\bf A}}\cap {\bf B}$, where the converse of the inclusion doesn`t always hold;
\item\label{quosubalg} $\langle S\rangle _{\V ,{\bf A}}/\theta =\langle S/\theta \rangle _{\V ,{\bf A}/\theta }$.\end{enumerate}

Indeed, (\ref{prodsubalg}) is clear and so is the inclusion in (\ref{subsubalg}), while, if we replace $\V $ by the variety of lattices, ${\bf A}$ by the five--element modular non--distributive lattice ${\bf M}_3$, ${\bf B}$ by the sublattice of ${\bf M}_3$ with universe $B=\{0,a,b,1\}$ and $S$ by the set $\{b,c\}$, where $a,b,c$ are the three atoms of ${\bf M}_3$, then we get a counter--example for the converse of the inclusion in (\ref{subsubalg}). The fact that $S\cap B\subseteq \langle S\cap B\rangle _{\V ,{\bf A}}\cap B$ and $\langle S\cap B\rangle _{\V ,{\bf A}}\cap {\bf B}\in \S _{\V }({\bf B})$ proves that $\langle S\cap B\rangle _{\V ,{\bf B}}\subseteq \langle S\cap B\rangle _{\V ,{\bf A}}\cap {\bf B}$ and, if we consider an element $b\in \langle S\cap B\rangle _{\V ,{\bf A}}\cap B$, then, for some $n\in \N $, $b_1,\ldots ,b_n\in S\cap B$ and some term $t$ in the language of $\V $, we have $b=t^{\bf A}(b_1,\ldots ,b_n)=t^{\bf B}(b_1,\ldots ,b_n)\in \langle S\cap B\rangle _{\V ,{\bf B}}$, thus $\langle S\cap B\rangle _{\V ,{\bf A}}\cap {\bf B}\subseteq \langle S\cap B\rangle _{\V ,{\bf B}}$, so (\ref{subsubalg}) holds.

Finally, $S/\theta \subseteq \langle S\rangle _{\V ,{\bf A}}/\theta \in \S _{\V }({\bf A}/\theta )$ and thus $\langle S/\theta \rangle _{\V ,{\bf A}/\theta }\subseteq \langle S\rangle _{\V ,{\bf A}}/\theta $, while, if we consider an $a\in \langle S\rangle _{\V ,{\bf A}}$, then, for some $n\in \N $, $a_1,\ldots ,a_n\in S$ and some term $t$ in the language of $\V $, we have $a=t^{\bf A}(a_1,\ldots ,a_n)$, thus $a/\theta =t^{{\bf A}/\theta }(a_1/\theta ,\ldots ,a_n/\theta )\in \langle S/\theta \rangle _{\V ,{\bf A}/\theta }$, hence $\langle S\rangle _{\V ,{\bf A}}/\theta \subseteq \langle S/\theta \rangle _{\V ,{\bf A}/\theta }$, which concludes the proof of (\ref{quosubalg}).

\subsection{\PBZ --Lattices: Definitions, Notations and Previously Established Properties}\label{pbzl}

We recall some preliminary notions on PBZ*-lattices and related structures only to such an extent as is necessary for the purposes of the present paper. For additional information on bounded involution lattices and pseudo-Kleene algebras, see \cite{RQT, PSK}; for Kleene lattices, a locus classicus is
\cite{Kalman}; for BZ-lattices, see \cite{RQT, CN}; finally, for
PBZ*-lattices, see \cite{GLP, PBZ2, rgcmfp}.

\begin{definition}
\label{defbi}A \emph{bounded involution lattice} (in brief, \emph{BI-lattice})
is an algebra $\mathbf{L}=(L,\wedge,\vee,^{\prime},0,1)$ of type $(2,2,1,0,0)$
such that $(L,\wedge,\vee,0,1)$ is a bounded lattice with induced partial
order $\leq$, $a^{\prime\prime}=a$ for all $a\in L$, and $a\leq b$ implies
$b^{\prime}\leq a^{\prime}$ for all $a,b\in L$.

A \emph{pseudo--Kleene algebra} is a BI-lattice $\mathbf{L}$ satisfying, for
all $a,b\in L$: $a\wedge a^{\prime}\leq b\vee b^{\prime}$. Distributive
pseudo--Kleene algebras are called \emph{Kleene algebras} or \emph{Kleene
lattices}.
\end{definition}

Note that, for any BI-lattice $\mathbf{L}$, $^{\prime}:L\rightarrow L$ is a dual lattice automorphism of $\mathbf{L}_{l}$, called \emph{involution}. The
involution of a pseudo--Kleene algebra is called \emph{Kleene complement}. If $\mathbf{L}$ is a BI-lattice, then for any $U\subseteq L$ and any $X\subseteq L^2$ we set $U^{\prime}=\{u^{\prime}:u\in U\}$ and $X^{\prime}=\{(x^{\prime},y^{\prime}):(x,y)\in X\}$.

For every algebra $\mathbf{A}$, if $\mathbf{A}$ has a (bounded) lattice reduct, then this reduct will be denoted by $\mathbf{A}_{l}$, and, if $\mathbf{A}$ has
a BI-lattice reduct, then such a reduct will be denoted by $\mathbf{A}_{bi}$. Let $\mathbf{L}$ be an algebra having a BI-lattice reduct. We say that an
element $a\in L$ is \emph{Kleene--sharp} or, simply, \emph{sharp}\footnote{See
however the introduction to the present paper or \cite{GLP} for the several
distinct notions of sharp element that collapse in the context of
PBZ*-lattices.} iff $a\wedge a^{\prime}=0$, or, equivalently, iff $a\vee
a^{\prime}=1$. We will denote the set of the sharp elements of $\mathbf{L}$ by
$S(\mathbf{L})$.

\begin{definition}
Let $\mathbf{L}$ be a BI-lattice. Then:

\begin{itemize}
\item $\mathbf{L}$ is an \emph{ortholattice} iff $S(\mathbf{L})=L$;

\item $\mathbf{L}$ is a \emph{paraorthomodular BI-lattice} iff, for all
$a,b\in L$, if $a\leq b$ and $a^{\prime}\wedge b=0$, then $a=b$;

\item $\mathbf{L}$ is an \emph{orthomodular lattice} iff $\mathbf{L}$ is an
ortholattice and, for all $a,b\in L$, if $a\leq b$, then $b=(b\wedge
a^{\prime})\vee a$.
\end{itemize}

If an algebra $\mathbf{A}$ has a BI-lattice reduct and $\mathbf{A}_{bi}$ is
paraorthomodular, then $\mathbf{A}$ is said to be paraorthomodular, as well.
\end{definition}

Clearly, any ortholattice is a pseudo--Kleene algebra. Note that, if a
BI-lattice $\mathbf{L}$ is orthomodular, then it is paraorthomodular; however,
if $\mathbf{L}$ is an ortholattice, then $\mathbf{L}$ is orthomodular iff it
is pa\-ra\-or\-tho\-mo\-du\-lar \cite[Prop. 2.1]{BH}.

We denote ${\bf MO}_0={\bf D}_2$ and, for any non--empty set $I$, ${\bf MO}_{|I|}=\boxplus _{i\in I}{\bf D}_2^2$. Clearly, for any cardinal number $\kappa $, ${\bf MO}_{\kappa }$ is an orthomodular lattice (and a Boolean algebra iff $\kappa \in \{0,1\}$).

\begin{definition}
\label{defbz} A \emph{Brouwer--Zadeh lattice} (in brief, \emph{BZ--lattice})
is an algebra $\mathbf{L}=\left(  {L,\,\wedge,\vee\,,\,}^{\prime}{,\,^{\sim
},0\,,1}\right)  $ of type $\left(  2,2,1,1,0,0\right)  $ such that $\left(
{L,\,\wedge,\vee\,,\,}^{\prime}{,0\,,1}\right)  $ is a pseudo--Kleene algebra
and, for all $a,b\in L$:%

\begin{flushleft}\begin{tabular}
[c]{clcl}%
(1) & $a\wedge a^{\sim}=0$; & (2) & $a\leq a^{\sim\sim}$;\\
(3) & $a\leq b$ implies $b^{\sim}\leq a^{\sim}$; & (4) & $a^{\sim\prime
}=a^{\sim\sim}$.
\end{tabular}\end{flushleft}

A \emph{BZ$^{\ast}$--lattice} is a BZ--lattice $\mathbf{L}$ satisfying the condition:%

\begin{flushleft}\begin{tabular}
[c]{cl}%
$(\ast)$ & for all $a\in L$, $(a\wedge a^{\prime})^{\sim}\leq a^{\sim}\vee
a^{\prime\sim}$.
\end{tabular}\end{flushleft}

A \emph{PBZ$^{\ast}$ --lattice} is a paraorthomodular BZ$^{\ast}$--lattice.

An \emph{antiortholattice} is a PBZ$^{\ast}$ --lattice $\mathbf{L}$ with $S(\mathbf{L})=\{0,1\}$.\label{thepbzl}\end{definition}

The operation $^{\sim}$ of a BZ--lattice is called \emph{Brouwer complement}. If $\mathbf{L}$ is a BZ-lattice, then for
any $U\subseteq L$ we set $U^{\sim}=\{u^{\sim}:u\in U\}$.

\begin{lemma}
\textup{\cite{RQT, GLP}}\label{basics} If $\mathbf{L}$ is a BZ--lattice, then,
for all $a,b\in L$:

\begin{flushleft}\begin{tabular}
[c]{clcl}
(i) & $a^{\sim}\leq a^{\prime}$; & (iv) & $(a\wedge b)^{\sim}\geq a^{\sim}\vee
b^{\sim}$;\\
(ii) & $a^{\sim\sim\sim}=a^{\sim}$; & (v) & $a^{\prime}\leq a$ implies
$a^{\sim}=0$.\\
(iii) & $(a\vee b)^{\sim}=a^{\sim}\wedge b^{\sim}$; &  &
\end{tabular}\end{flushleft}\end{lemma}

Lemma \ref{basics}.(i) shows that, in any BZ--lattice $\mathbf{L}$, for any $a\in L$, $a^{\sim}=1$ iff $a=0$. By Lemma \ref{basics}.(iv), in any BZ--lattice $\mathbf{L}$, condition $(\ast)$ is equivalent to $(a\wedge a^{\prime})^{\sim}=a^{\sim}\vee a^{\prime\sim}$ for all $a\in L$.

In any BZ--lattice $\mathbf{L}$, we set $\Diamond a=a^{\sim\sim}$ and $\square
a=a^{\prime\sim}$ for all $a\in L$. Note that, if $\mathbf{L}$ is a BZ$^{\ast
}$--lattice, then $\mathbf{L}$ is paraorthomodular iff it satisfies the
following equational condition: for all $a,b\in L$, $(a^{\sim}\vee(\Diamond
a\wedge\Diamond b))\wedge\Diamond a\leq\Diamond b$, therefore PBZ$^{\ast}$--lattices form a variety.

If $\mathbf{L}$ is a PBZ$^{\ast}$ --lattice, then $S(\mathbf{L})=\{a^{\sim
}:a\in L\}=\{a\in L:a=a^{\sim\sim}\}=\{a\in L:a\vee a^{\sim}=1\}=\{a\in
L:a^{\sim}=a^{\prime}\}$ \cite{GLP}. For every PBZ$^{\ast}$
--lattice ${\mathbf{L}}$, $S(\mathbf{L})$ is the universe of the largest subalgebra of $\mathbf{L}$ which is an orthomodular lattice, denoted by $\mathbf{S}(\mathbf{L})$.

\begin{lemma}\label{uniquesim}\textup{\cite{GLP}} Let $\mathbf{L}$ be a PBZ$^{\ast}$
--lattice. Then:

\begin{itemize}
\item $\mathbf{L}_{bi}$ is an ortholattice iff $S(\mathbf{L})=L$ iff
$\mathbf{L}_{bi}$ is an orthomodular lattice;

\item $\mathbf{L}$ is an antiortholattice iff $x^{\sim}=0$ for all $x\in
L\setminus\{0\}$.
\end{itemize}\end{lemma}

The Brouwer complement of an antiortholattice $L$, given by Lemma \ref{uniquesim}, is called \emph{the trivial Brouwer complement}: $0^{\sim}=1$ and $x^{\sim}=0$ for all $x\in L\setminus\{0\}$.

We denote by $\mathbb{BA}$, $\mathbb{BI}$, $\mathbb{PKA}$, $\mathbb{OL}$,
$\mathbb{OML}$, $\mathbb{BZL}$ and $\mathbb{PBZL}^{\ast}$ the varieties of
Boolean algebras, BI-lattices, pseudo--Kleene algebras, ortholattices,
orthomodular lattices, BZ--lattices and PBZ$^{\ast}$ --lattices, respectively. As shown by Lemma \ref{uniquesim}, $\OL $ and $\OML $ can be viewed as classes of algebras of type $\left(  2,2,1,1,0,0\right)  $, with a repeat occurrence of the unary operation symbol. The proper universal class of
antiortholattices will be denoted by $\mathbb{AOL}$. By \cite{rgcmfp}, $V\left(  \mathbb{AOL}\right)  $ is axiomatized
relative to $\mathbb{PBZL}^{\ast}$ by the equation:

\begin{description}
\item[J0] $x\approx\left(  x\wedge y^{\sim}\right)  \vee\left(  x\wedge
\Diamond y\right)  $.
\end{description}

Clearly, $\AOL $ is closed w.r.t. subalgebras and quotients, but not w.r.t. direct products, since Definition \ref{thepbzl} ensures us that every antiortholattice is directly indecomposable.

Note that antiortholattices are exactly the paraorthomodular pseudo--Kleene
algebras endowed with the trivial Brouwer complement which satisfy condition
$(\ast)$, or, equivalently, exactly the pseudo--Kleene algebras $\mathbf{L}$
with $S(\mathbf{L})=\{0,1\}$, endowed with the trivial Brouwer complement.
Thus any pseudo--Kleene algebra where $0$ is meet--irreducible becomes an
antiortholattice once endowed with the trivial Brouwer complement. Hence any
BZ--lattice with a meet--irreducible bottom element is an antiortholattice. In
particular, any BZ--chain is an antiortholattice and, of course, any
self--dual bounded chain becomes an antiortholattice if endowed with its dual
lattice automorphism as Kleene complement, and with the trivial Brouwer complement.

Let $\mathbf{M}$ be a bounded lattice, $f:\mathbf{M}\rightarrow\mathbf{M}^{d}$
be a dual lattice isomorphism and $\mathbf{K}$ be a BI-lattice, with
involution $^{\prime\mathbf{K}}$. In the ordinal sum $\mathbf{M}%
\oplus\mathbf{K}_{l}\oplus\mathbf{M}^{d}$, we will denote by $M^{d}$ the
universe of the sublattice $\mathbf{M}^{d}$. Then the bounded lattice
$\mathbf{M}\oplus\mathbf{K}_{l}\oplus\mathbf{M}^{d}$ can be made into a
BI-lattice $\mathbf{M}\oplus\mathbf{K}\oplus\mathbf{M}^{d}$ by defining its
involution as follows:\vspace*{-7pt}\[
x^{\prime}=%
\begin{cases}
f(x), & \mbox{for all }x\in M;\\
x^{\prime\mathbf{K}}, & \mbox{for all }x\in K;\\
f^{-1}(x), & \mbox{for all }x\in M^{d}.
\end{cases}
\]
In this BI-lattice, $M^{\prime}=M^{d}$. Clearly, $\mathbf{M}\oplus\mathbf{K}\oplus\mathbf{M}^{d}$
is a pseudo--Kleene algebra iff $\mathbf{K}$ is a
pseudo--Kleene algebra. The pseudo--Kleene algebra $\mathbf{M}\oplus\mathbf{D}_{1}\oplus\mathbf{M}^{d}$ will be denoted by $\mathbf{M}\oplus\mathbf{M}^{d}$, as its underlying bounded lattice.

Let $\mathbf{A}$ and $\mathbf{B}$ be nontrivial BI-lattices, with involutions
$^{\prime\mathbf{A}}$ and $^{\prime\mathbf{B}}$, respectively. Then
$\mathbf{A}_{l}\boxplus\mathbf{B}_{l}$ can be organized as a BI-lattice
$\mathbf{A}\boxplus\mathbf{B}$ by letting the involution of $\mathbf{A}%
\boxplus\mathbf{B}$ be $^{\prime}:A\boxplus B\rightarrow A\boxplus B$,
$a^{\prime}=a^{\prime\mathbf{A}}$ for all $a\in A$ and $b^{\prime}%
=b^{\prime\mathbf{B}}$ for all $b\in B$. This makes $\mathbf{A}$ and
$\mathbf{B}$ subalgebras of the BI-lattice $\mathbf{A}\boxplus\mathbf{B}$.

Now let $\mathbf{A}$ and $\mathbf{B}$ be BZ--lattices, with Brouwer
complements $^{\sim\mathbf{A}}$ and $^{\sim\mathbf{B}}$, respectively. Then
$\mathbf{A}_{bi}\boxplus\mathbf{B}_{bi}$ can be organized as an algebra
$\mathbf{A}\boxplus\mathbf{B}$ of type $(2,2,1,1,0,0)$ by defining $^{\sim
}:A\boxplus B\rightarrow A\boxplus B$, $a^{\sim}=a^{\sim\mathbf{A}}$ for all
$a\in A$ and $b^{\sim}=b^{\sim\mathbf{B}}$ for all $b\in B$. This makes
$\mathbf{A}$ and $\mathbf{B}$ subalgebras of $\mathbf{A}\boxplus\mathbf{B}$,
which is an algebra of type $(2,2,1,1,0,0)$, but not necessarily a BZ--lattice.

If ${\mathbb{C}}$ and ${\mathbb{D}}$ are subclasses of the variety of bounded
lattices or of one of the varieties $\mathbb{BI}$ and $\mathbb{BZL}$, then we
let:%
\[
{\mathbb{C}}\boxplus{\mathbb{D}}=\{\mathbf{D}_{1}\}\cup\{\mathbf{A}%
\boxplus\mathbf{B}:\mathbf{A}\in{\mathbb{C}}\setminus\{\mathbf{D}%
_{1}\},\mathbf{B}\in{\mathbb{D}}\setminus\{\mathbf{D}_{1}\}\}.
\]

Clearly, the operation $\boxplus$ on classes is associative and, if
$\mathbf{A}$ is a nontrivial bounded lattice or BI-lattice or BZ--lattice,
then $\mathbf{D}_{2}\boxplus\mathbf{A}=\mathbf{A}$, hence, in the notation above:

\begin{itemize}
\item if $\mathbf{D}_{2}\in{\mathbb{C}} $, then ${\mathbb{D}} \subseteq
{\mathbb{C}} \boxplus{\mathbb{D}} $;

\item thus, if $\mathbf{D}_{2}\in{\mathbb{C}} \cap{\mathbb{D}} $, then
${\mathbb{C}} \cup{\mathbb{D}} \subseteq{\mathbb{C}} \boxplus{\mathbb{D}} $.
\end{itemize}

For future reference, we consider the following identities in the language of BZ--lattices:

\begin{description}
\item[SDM] (the \emph{Strong de Morgan law}) $\left(  x\wedge y\right)
^{\sim}\approx x^{\sim}\vee y^{\sim}$;

\item[WSDM] (\emph{weak SDM}) $\left(  x\wedge y^{\sim}\right)  ^{\sim}\approx
x^{\sim}\vee\Diamond y$;

\item[S1] $\left(  x\wedge\left(  x\wedge y\right)  ^{\sim}\right)  ^{\sim
}\approx x^{\sim}\vee\Diamond\left(  x\wedge y\right)  $;

\item[S2] $\left(  x\wedge\left(  y\wedge y^{\prime}\right)  ^{\sim}\right)
^{\sim}\approx x^{\sim}\vee\Diamond\left(  y\wedge y^{\prime}\right)  $;

\item[S3] $\left(  x\wedge\Diamond\left(  y\wedge y^{\prime}\right)  \right)
^{\sim}\approx x^{\sim}\vee\left(  y\wedge y^{\prime}\right)  ^{\sim}$;

\item[J1] $x\approx\left(  x\wedge\left(  x\wedge y\right)  ^{\sim}\right)
\vee\left(  x\wedge\Diamond\left(  x\wedge y\right)  \right)  $;

\item[J2] $x\approx\left(  x\wedge\left(  y\wedge y^{\prime}\right)  ^{\sim
}\right)  \vee\left(  x\wedge\Diamond\left(  y\wedge y^{\prime}\right)
\right)  $.\end{description}

Remark that $J0$ above implies $J1$ and $J2$ and that SDM
implies WSDM, which in turn implies $S1,S2,S3$. Also note that, for any BZ--lattice $\mathbf{L}$, $\mathbf{L}\vDash J1$ iff, for all $x,y\in L$, $y\leq x$ implies $x=\left(  x\wedge y^{\sim}\right)  \vee\left( x\wedge\Diamond y\right)  $. Hence an equivalent form of $J1$ is as follows and, similarly, $S1$ can be written equivalently in the following form:

\begin{description}
\item[J1'] $x\vee y\approx\left(  (x\vee y)\wedge y^{\sim}\right)  \vee\left(
(x\vee y)\wedge\Diamond y\right)  $;

\item[S1'] $((x\vee y)\wedge y^{\sim})^{\sim}\approx(x\vee y)^{\sim}%
\vee\Diamond y$.
\end{description}

Clearly, $\mathbb{OML}\vDash SDM$, thus $\mathbb{OML}\vDash\{WSDM,S1,S2,S3\}$,
and, in $\{\mathbf{L}\in\mathbb{BZL}:\mathbf{L}\vDash x^{\sim}\approx
x^{\prime}\}$, $J1$ is equivalent to the orthomodularity condition, thus
$\mathbb{OML}\vDash J1$. Also, clearly, $\mathbb{OML}\vDash J2$.

Trivially, $\mathbb{AOL}\vDash\{WSDM,S1,S2,S3\}$ and $\AOL \vDash J0$, hence $\AOL \vDash \{J1,J2\}$. The fact that $J0$ axiomatizes $V(\AOL )$ over $\PBZs $ and $V(\AOL )\vDash WSDM$ shows that $J0$ implies WSDM. Clearly, an
antiortholattice $\mathbf{L}$ satisfies $SDM$ iff $0$ is meet--irreducible in $\mathbf{L}_{l}$; for instance, $\mathbf{D}_{2}^{2}\oplus\mathbf{D}_{2}^{2}$ can be organized as an antiortholattice (see Section \ref{ordsum}) that fails SDM. Therefore $\mathbb{AOL}\nvDash SDM$.

Note from the above that $\mathbb{OML}\vee V(\mathbb{AOL})\vDash \{WSDM,J1,J2,S1,S2,S3\}$.

\section{Dense elements in \PBZ --lattices}
\label{dense}

Whenever a bounded lattice $\mathbf{L}$ is endowed with a closure operator
$C$, important information on the structure of $\mathbf{L}$ is encoded not
only in its set $\{x\in L:C\left(  x\right)  =$\linebreak $x\}$ of
\emph{closed} elements, but also in its set $\left\{  x\in L:C\left(
x\right)  =1\right\}  $ of \emph{dense} elements. Under optimal circumstances,
like for Stone algebras, knowledge of both sets --- plus some information
concerning their distribution in the lattice ordering --- is sufficient to
fully reconstruct $\mathbf{L}$. This is the idea behind the
\emph{representation by triples} of Stone algebras and other related
structures (see e.g. \cite{CG}). The case of PBZ*-lattices falls somewhat
short of such an ideal situation --- still, the study of dense elements
provides useful insights into their structure.

For any PBZ$^{\ast}$ --lattice $\mathbf{L}$, we call \emph{dense} any $a\in L$
such that $a^{\sim}=0$. The set of all dense elements of $\mathbf{L}$ will be
denoted by $D(\mathbf{L})$; we also set $T(\mathbf{L})=\{0\}\cup
D(\mathbf{L})$. Clearly, $S(\mathbf{L})\cap T(\mathbf{L})=\{0,1\}$ and a
subalgebra of $\mathbf{L}$ is included in $T(\mathbf{L})$ iff it is an
antiortholattice; in particular, if $T(\mathbf{L})$ is the universe of a
subalgebra of $\mathbf{L}$, then this subalgebra, that we will denote by
$\mathbf{T}(\mathbf{L})$, is the largest subalgebra of $\mathbf{L}$ which is an antiortholattice.

Since $S(\mathbf{L})\cap T(\mathbf{L})=\{0,1\}$ and $S(\mathbf{L})^{\prime
}=S(\mathbf{L})$, we also have $S(\mathbf{L})\cap T(\mathbf{L})^{\prime
}=S(\mathbf{L})\cap(T(\mathbf{L})\cup T(\mathbf{L})^{\prime})=\{0,1\}$.

\begin{lemma}
\label{olsaolt}Let $\mathbf{L}$ be a PBZ$^{\ast}$ --lattice. Then:

\begin{enumerate}
\item $\mathbf{L}$ is an orthomodular lattice iff $S(\mathbf{L})=L$ iff
$T(\mathbf{L})=\{0,1\}$;

\item $\mathbf{L}$ is an antiortholattice iff $S(\mathbf{L})=\{0,1\}$ iff
$T(\mathbf{L})=L$.
\end{enumerate}
\end{lemma}

\begin{proof}
The only nontrivial item is the right-to-left direction of the second
equivalence in (i). Let $\mathbf{L}\in\mathbb{PBZL}^{\ast}\backslash
\mathbb{OML}$. If $\mathbf{L}$ is isomorphic to $\mathbf{D}_{3}$, then
$T(\mathbf{L})\neq\left\{  0,1\right\}  $. Otherwise, there is an $a\in L$ s.t.$$0<a\wedge a^{\prime}=\left(  a\vee a^{\prime}\right)  ^{\prime}\leq a\vee
a^{\prime}<1\text{,}$$so, by Lemma \ref{basics}.(v), $a\vee a^{\prime}\in T(\mathbf{L})\backslash
\left\{  0,1\right\}  $.
\end{proof}

For any PBZ$^{\ast}$--lattice $\mathbf{L}$, $\{0,1\}$ is the universe of the smallest subalgebra of $\mathbf{L}$, which belongs to $\mathbb{OML}\cap\mathbb{AOL}$, so note from the previous proposition that, if $\mathbf{L}\in\mathbb{OML}\cup\mathbb{AOL}$, then $S(\mathbf{L})$ and $T(\mathbf{L})$ are subuniverses of $\mathbf{L}$.

On the other hand, let $\mathbf{L}$ be a generic PBZ$^{\ast}$ --lattice. Observe that:

\begin{itemize}
\item $L\setminus S(\mathbf{L})$ is closed w.r.t. to the Kleene complement;

\item $T(\mathbf{L})$ is closed w.r.t. to the Brouwer complement, as well as joins, hence $T(\mathbf{L})^{\prime }$ is is closed w.r.t. meets and: if $T(\mathbf{L})$ is closed under the Kleene complements, then it is also closed under meets and $T(\mathbf{L})=T(\mathbf{L})^{\prime }$;

\item if $D(\mathbf{L})$ is closed w.r.t. meets, then $D(\mathbf{L})$ is a lattice filter of $\mathbf{L}_{l}$ and $T(\mathbf{L})$ is closed w.r.t. meets;

\item $\left[  u\right)  \subseteq D(\mathbf{L})$ for all $u\in D(\mathbf{L})$, hence $(u]\subseteq D(\mathbf{L})^{\prime}$ for all $u\in D(\mathbf{L})^{\prime}$.\end{itemize}

Clearly, if $\mathbf{L}$ satisfies the SDM, then $D(\mathbf{L})$ is closed
w.r.t. meets, thus so is $T(\mathbf{L})$, hence $T(\mathbf{L})$ is the
universe of a bounded sublattice of $\mathbf{L}_{l}$ in which $0$ is meet--irreducible, thus $T(\mathbf{L})^{\prime}$ is the universe of a bounded sublattice of $\mathbf{L}_{l}$ in which $1$ is join--irreducible. Recall from Subsection \ref{pbzl} that, in any antiortholattice which satisfies the SDM, $0$ is meet--irreducible, and that any BZ--lattice with $0$ meet--irreducible is an antiortholattice which satisfies the SDM. Hence, for any BZ--lattice $\mathbf{L}$, the following are equivalent:

\begin{itemize}
\item $\mathbf{L}$ satisfies the SDM and $0$ is meet--irreducible in $\mathbf{L}_{l}$;

\item $\mathbf{L}$ is an antiortholattice and it satisfies the SDM;

\item $\mathbf{L}$ is an antiortholattice and $0$ is meet--irreducible in $\mathbf{L}_{l}$.\end{itemize}

Let us also retain, from the above:

\begin{lemma}
\label{tfilter}For any PBZ$^{\ast}$ --lattice $\mathbf{L}$:

\begin{itemize}
\item $T(\mathbf{L})$ is closed w.r.t. meets iff $T(\mathbf{L})$ is the
universe of a bounded sublattice of $\mathbf{L}_{l}$ iff $T(\mathbf{L}%
)^{\prime}$ is closed w.r.t. joins iff $T(\mathbf{L})^{\prime}$ is the
universe of a bounded sublattice of $\mathbf{L}_{l}$ iff $\left\langle
T(\mathbf{L})\right\rangle _{\mathbb{BI}}=T(\mathbf{L})\cup T(\mathbf{L}%
)^{\prime}$;

\item If $\mathbf{L}$ satisfies the SDM, then $T(\mathbf{L})$ is the universe
of a bounded sublattice of $\mathbf{L}_{l}$ in which $0$ is meet--irreducible,
so $D(\mathbf{L})$ is a lattice filter of $\mathbf{L}_{l}$;

\item $T(\mathbf{L})$ is closed w.r.t. the Kleene complement iff
$T(\mathbf{L})$ is a subuniverse of $\mathbf{L}$ iff $T(\mathbf{L}%
)=T(\mathbf{L})^{\prime}$.
\end{itemize}
\end{lemma}

Observe that, if $\mathbf{L}$ is a \PBZ --lattice, then all sublattices of $\mathbf{L}_{l}$ which are closed w.r.t. the Brouwer complement are bounded sublattices of $\mathbf{L}_{l}$. Moreover, for any subsemilattice $\mathbf{S}$ of the underlying join--semilattice of $\mathbf{L}$ or the underlying meet--semilattice of $\mathbf{L}$, if $S$ is closed w.r.t. the Brouwer complement of $\mathbf{L}$, then $\{0,1\}\subseteq S$. Thus, the only interval of $\mathbf{L}_{l}$ which is closed w.r.t. the Brouwer complement is $[0,1]=L$, so $\mathbf{L}$ has no proper convex subalgebras.

The next lemma will be useful in what follows.

\begin{lemma}
\label{2joinirred}Let $\mathbf{L}$ be a nontrivial PBZ$^{\ast}$ --lattice
that satisfies $J1$ and $u,v\in L\setminus\{0\}$ such that $u\leq v$ and $v$
is join--irreducible in $\mathbf{L}_{l}$. Then:

\begin{enumerate}
\item \label{2joinirred1} either $v\vee v^{\sim}\leq u^{\sim}$ or $u^{\sim
}=v^{\sim}$;

\item \label{2joinirred2} if $v\in S(\mathbf{L})$, then $v=\Diamond u$;

\item \label{2joinirred3} if $u,v\in S(\mathbf{L})$, then $u=v$;

\item \label{2joinirred5} if $v\in T(\mathbf{L})$, then $u\in T(\mathbf{L})$;

\item \label{2joinirred6} if $v\in T(\mathbf{L})$ and $u\in S(\mathbf{L})$, then $u=v=1$ and $\mathbf{L}$ is an antiortholattice that satisfies the SDM.
\end{enumerate}
\end{lemma}

\begin{proof}(\ref{2joinirred1}) $u\leq v$ implies $v^{\sim}\leq u^{\sim}$. Since
$\mathbf{L}$ satisfies $J1$, we have $v=(v\wedge u^{\sim})\vee(v\wedge\Diamond
u)$, so that, by the join--irreducibility of $v$, either $v=v\wedge u^{\sim}$
or $v=v\wedge\Diamond u$, so either $v\leq u^{\sim}$ or $v\leq\Diamond u$,
hence either $v\vee v^{\sim}\leq u^{\sim}$ or $u^{\sim}=u^{\sim\sim\sim}\leq
v^{\sim}$, the latter of which implies $u^{\sim}=v^{\sim}$.

\noindent (\ref{2joinirred2}) If $v\in S(\mathbf{L})$, then $v\vee v^{\sim}=1\nleq
u^{\sim}$, so, by (\ref{2joinirred1}), we have $u^{\sim}=v^{\sim}$, thus
$v=\Diamond v=\Diamond u$.

\noindent (\ref{2joinirred3}) By (\ref{2joinirred2})
and the fact that $\Diamond u=u$.

\noindent (\ref{2joinirred5}) By
(\ref{2joinirred1}), we have either $v\leq u^{\sim}$ or $u^{\sim}=v^{\sim}=0$,
so that either $u^{\sim}\in T(\mathbf{L})$ or $u\in D(\mathbf{L})$, that is
either $u^{\sim}=1$, which would contradict $u\neq0$, or $u\in D(\mathbf{L})$,
thus $u\in D(\mathbf{L})$.

\noindent (\ref{2joinirred6}) By (\ref{2joinirred5}), we obtain $u\in(S(\mathbf{L})\cap T(\mathbf{L}))\setminus\{0\}$, so
$v=u=1$, thus $1$ is join--irreducible in $\mathbf{L}_{l}$, hence $\mathbf{L}$ is an antiortholattice and it satisfies the SDM.\end{proof}

\begin{proposition}
\label{jirred}Let $\mathbf{L}$ be a nontrivial PBZ$^{\ast}$ --lattice that
satisfies $J1$ and $v\in D(\mathbf{L})$ such that $v$ is join--irreducible in
$\mathbf{L}_{l}$. Then:

\begin{itemize}
\item all elements of $\mathbf{L}$ which are comparable with $v$ belong to
$T(\mathbf{L})$;

\item for any $x\in L\setminus T(\mathbf{L})$, we have $x\wedge v=0$, $x\vee
x^{\prime}\nleq v$ and $x$ and $x\wedge x^{\prime}$ are incomparable to $v$.
\end{itemize}
\end{proposition}

\begin{proof}
By Lemma \ref{2joinirred}.(\ref{2joinirred5}) and the fact that $D(\mathbf{L}%
)$ is closed w.r.t. upper bounds.
\end{proof}

\begin{proposition}
Let $\mathbf{L}$ be a nontrivial orthomodular lattice. Then the only
join--irreducible elements of $\mathbf{L}_{l}$ are its atoms and, dually, its
only meet--irreducible elements are its coatoms.
\end{proposition}

\begin{proof}
Since $\mathbf{L}\in\mathbb{OML} $, we have $\mathbf{L}\vDash J1$, so, by
Lemma \ref{2joinirred}.(\ref{2joinirred3}), for every $v\in L\setminus\{0\}$
such that $v$ is join--irreducible in $\mathbf{L}_{l}$, there exists no $u\in
L$ with $0<u<v$, hence $v$ is an atom of $\mathbf{L}_{l}$.
\end{proof}

\begin{corollary}
Let $\mathbf{L}$ be a nontrivial PBZ$^{\ast}$ --lattice, $u\in
S(\mathbf{L})\setminus\{1\}$ and $v\in S(\mathbf{L})\setminus\{0\}$. Then:

\begin{itemize}
\item if $v$ is join--irreducible in $\mathbf{L}_{l}$, then $v$ is an atom of
$\mathbf{S}(\mathbf{L})_{l}$;

\item if $u$ is meet--irreducible in $\mathbf{L}_{l}$, then $u$ is a co--atom
of $\mathbf{S}(\mathbf{L})_{l}$.
\end{itemize}\label{sharpirred}\end{corollary}

\section{Ordinal Sums and Congruences of Antiortholattices}
\label{ordsum}

Note, from Corollary \ref{cgred} and the characterization of subdirect irreducibility in Subsection \ref{ualglat}, that, if a reduct of an algebra $\mathbf{A}$ is subdirectly irreducible, then so is $\mathbf{A}$. We will often use the following lemmas and propositions without referencing them.

\begin{lemma}
If $\mathbf{M}$ is a nontrivial bounded lattice and $\mathbf{K}$ is a
pseudo-Kleene algebra, then the canonical pseudo-Kleene algebra $\mathbf{L}%
=\mathbf{M}\oplus\mathbf{K}\oplus\mathbf{M}^{d}$, endowed with the trivial
Brouwer complement, becomes an antiortholattice.
\end{lemma}

\begin{proof}
Clearly, for any $x,y\in L$, $x\wedge y=0$ implies $x=0$ or $y=0$ or $x,y\in
M$. Thus, for any $a,b\in L$ such that $a\leq b$ and $a^{\prime}\wedge b=0$,
we have one of the following situations:

\begin{itemize}
\item $a^{\prime}=0$, so that $a=1$ and thus $b=1=a$;

\item $b=0$, so that $a=0=b$;

\item $a^{\prime},b\in M$, so that $a\in M^{d}$, thus $\mathbf{K}%
\cong\mathbf{D}_{1}$, with $K=\{a\}=\{b\}$, so $a=b$.
\end{itemize}

Therefore $\mathbf{L}$ is paraorthomodular.

Now let $^{\sim}:L\rightarrow L$ be the trivial Brouwer complement and let
$a\in L$. If $a\in K$, then $a^{\prime}\in K$, so $a\wedge a^{\prime}\in K$,
thus $0\notin\{a,a^{\prime},a\wedge a^{\prime}\}$, hence $(a\wedge a^{\prime
})^{\sim}=0=a^{\sim}\vee a^{\prime\sim}$. If $a\in M$, then $a^{\prime}\in
M^{d}$, thus $a\leq a^{\prime}$, so $a^{\prime\sim}\leq a^{\sim}$, hence
$(a\wedge a^{\prime})^{\sim}=a^{\sim}=a^{\sim}\vee a^{\prime\sim}$. If $a\in
M^{d}$, then $a^{\prime}\in M$, so $(a\wedge a^{\prime})^{\sim}=a^{\prime\sim
}=a^{\sim}\vee a^{\prime\sim}$, by duality from the previous case. Thus the pseudo--Kleene algebra $\mathbf{L}$, endowed with the trivial $^{\sim}$, fulfills condition $(\ast)$, hence it becomes an antiortholattice.\end{proof}

We call the antiortholattice $\mathbf{M}\oplus\mathbf{K}\oplus\mathbf{M}^{d}$ in the previous lemma the {\em canonical antiortholattice} with lattice reduct $\mathbf{M}\oplus\mathbf{K}_l\oplus\mathbf{M}^{d}$.

Let $\mathbf{L}$ be a BI--lattice, $\mathbf{S}=(S,\wedge,\vee)$ a sublattice of $\mathbf{L}_{l}$ and $\mathbf{S}^{\prime }=(S^{\prime },\wedge,\vee)$. Then $\mathbf{S}^{\prime }$ is also a sublattice of $\mathbf{L}_{l}$ and, in the particular case when $S=L$, we have $\mathbf{S}=\mathbf{S}^{\prime }=\mathbf{L}_{l}$. The map $^{\prime }\mid _{S^{\prime }}:S^{\prime }\rightarrow S$ is a dual lattice isomorphism between $\mathbf{S}^{\prime }$ and $\mathbf{S}$, hence $\mathbf{S}^{\prime }\cong \mathbf{S}^d$ and the map $\theta \mapsto \theta ^{\prime }$ from $\mathrm{Con}(\mathbf{S}^{\prime })$ to $\mathrm{Con}(\mathbf{S}^d)=\mathrm{Con}(\mathbf{S})$ is a lattice isomorphism, thus $(\theta \cap \zeta )^{\prime }=\theta ^{\prime }\cap \zeta ^{\prime }$ and $(\theta \vee \zeta )^{\prime }=\theta ^{\prime }\vee \zeta ^{\prime }$ for all $\theta ,\zeta \in\mathrm{Con}(\mathbf{S}^{\prime })$. Note, also, that $\theta ^{\prime \prime }=\theta $ for all $\theta \in\mathrm{Con}(\mathbf{S}^{\prime })$.

If $\mathbf{L}=\mathbf{M}\oplus \mathbf{K}\oplus \mathbf{M}^d$ for some bounded lattice $\mathbf{M}$ and some BI--lattice $\mathbf{K}$, so that, with the notation above, $\mathbf{M}^d=\mathbf{M}^{\prime d}$, then, for any $\theta\in\mathrm{Con}(\mathbf{M})$, $\alpha\in\mathrm{Con}(\mathbf{K})$ and $\zeta \in\mathrm{Con}(\mathbf{M}^{d})$, we have: $\theta ^{\prime }\in\mathrm{Con}(\mathbf{M}^{d})$, $\zeta ^{\prime }\in\mathrm{Con}(\mathbf{M})$ and $\theta \oplus \alpha \oplus \zeta \in\mathrm{Con}(\mathbf{L})$, in particular $\theta \oplus \alpha \oplus \theta ^{\prime }\in\mathrm{Con}(\mathbf{L})$. In particular, if $\mathbf{K}$ is trivial, so that $\mathrm{Con}(\mathbf{K})=\{\Delta_{K}\}=\{\nabla_{K}\}\cong \mathbf{D}_{1}$ and $\mathbf{M}\oplus\mathbf{K}\oplus\mathbf{M}^{d}=\mathbf{M}\oplus\mathbf{M}^{d}$, we have $\theta\oplus\Delta_{K}\oplus\theta^{\prime}=\theta\oplus\theta^{\prime}\in\mathrm{Con}(\mathbf{M}\oplus\mathbf{M}^{d})$ for all $\theta \in\mathrm{Con}(\mathbf{M})$.

\begin{lemma}
If $\mathbf{L}$ is a nontrivial BI--lattice and $\theta
\in\mathrm{Con}(\mathbf{L})_{\mathbb{BI}}\setminus\{\nabla_{L}\}$, then: $\theta$ preserves the trivial Brouwer complement on $L$ iff $0/\theta=\{0\}$ iff $1/\theta=\{1\}$.\label{bicg01}\end{lemma}

\begin{proof} Since $\theta $ preserves the involution, we have $0/\theta=\{0\}$ iff $1/\theta=\{1\}$. Now let $^{\sim }:L\rightarrow L$ be the trivial Brouwer complement. If $0/\theta=\{0\}$, then clearly $\theta $ preserves $^{\sim }$. Finally, assume that $\theta $ preserves $^{\sim }$, let $a\in 0/\theta $ and assume by absurdum that $a\neq 0$. Then $(0,1)=(a^{\sim },1)=(a^{\sim },0^{\sim })\in \theta $, which contradicts the fact that $\theta \neq \nabla_{L}$. Therefore $0/\theta =\{0\}$.\end{proof}

\begin{proposition}
\begin{enumerate}
\item \label{bibzl1} For any BI--lattice $\mathbf{L}$,
$\mathrm{Con}_{\mathbb{BI}}(\mathbf{L})=\{\theta\in\mathrm{Con}(\mathbf{L}):\theta=\theta^{\prime}\}$ and $\mathrm{Con}_{\mathbb{BI}01}(\mathbf{L})=\mathrm{Con}_{\mathbb{BI}0}(\mathbf{L})$.

\item \label{bibzl2} For any antiortholattice $\mathbf{L}$, $\mathrm{Con}_{\mathbb{BZL}01}(\mathbf{L})=\mathrm{Con}_{\mathbb{BI}01}(\mathbf{L})$ and
$\mathrm{Con}_{\mathbb{BZL}}(\mathbf{L})=\mathrm{Con}_{\mathbb{BI}01}(\mathbf{L})\cup\{\nabla_{L}\}$, which is a complete bounded sublattice of
$\mathrm{Con}_{\mathbb{BI}}(\mathbf{L})$. If $\mathbf{L}$ is nontrivial, then $\mathrm{Con}_{\mathbb{BZL}}(\mathbf{L})\cong\mathrm{Con}_{\mathbb{BI}01}(\mathbf{L})\oplus\mathbf{D}_{2}$ and has the top element of $\mathrm{Con}_{\mathbb{BI}01}(\mathbf{L})$ as a unique co--atom.\end{enumerate}

\label{bibzl}
\end{proposition}

\begin{proof}
(\ref{bibzl1}) $\mathrm{Con}_{\mathbb{BI}}(\mathbf{L})\subseteq\mathrm{Con}(\mathbf{L})$ and, for any $\theta\in\mathrm{Con}(\mathbf{L})$, we have
$\theta\in\mathrm{Con}_{\mathbb{BI}}(\mathbf{L})$ exactly when, for all
$a,b\in L$: $(a,b)\in\theta$ iff $(a^{\prime},b^{\prime})\in\theta$ iff
$(a,b)\in\theta^{\prime}$, that is exactly when $\theta=\theta^{\prime}$. Thus $\mathrm{Con}_{\mathbb{BI}01}(\mathbf{L})=\mathrm{Con}_{\mathbb{BI}0}(\mathbf{L})$.

\noindent(\ref{bibzl2}) By Lemma \ref{bicg01}, $\mathrm{Con}_{\mathbb{BZL}%
01}(\mathbf{L})=\mathrm{Con}_{\mathbb{BI}01}(\mathbf{L})$ and $\mathrm{Con}%
_{\mathbb{BZL}}(\mathbf{L})=\mathrm{Con}_{\mathbb{BI}01}(\mathbf{L}%
)\cup\{\nabla_{L}\}$. $\mathrm{Con}_{\mathbb{BI}01}(\mathbf{L})$ is a bounded
lattice, with smallest element $\Delta_{L}$, and a complete sublattice of
$\mathrm{Con}_{\mathbb{BI}}(\mathbf{L})$, therefore $\mathrm{Con}%
_{\mathbb{BI}01}(\mathbf{L})\cup\{\nabla_{L}\}$ is a complete bounded
sublattice of $\mathrm{Con}_{\mathbb{BI}}(\mathbf{L})$. If $\mathbf{L}$ is
nontrivial, then $\nabla_{L}\notin\mathrm{Con}_{\mathbb{BI}01}(\mathbf{L})$,
hence $\mathrm{Con}_{\mathbb{BI}01}(\mathbf{L})\cup\{\nabla_{L}\}\cong%
\mathrm{Con}_{\mathbb{BI}01}(\mathbf{L})\oplus\mathbf{D}_{2}$ and has
$\max(\mathrm{Con}_{\mathbb{BI}01}(\mathbf{L}))$ as a unique co--atom.
\end{proof}

\begin{theorem}
Let $\mathbf{M}$ be a bounded lattice, $\mathbf{K}$ be a bounded involution
lattice and $\mathbf{L}=\mathbf{M}\oplus\mathbf{K}\oplus\mathbf{M}^{d}$. Then:

\begin{enumerate}
\item \label{mainthaol1} $\mathrm{Con}_{\mathbb{BI} }(\mathbf{L}%
)=\{\alpha\oplus\beta\oplus\alpha^{\prime}:\alpha\in\mathrm{Con}%
(\mathbf{M}),\beta\in\mathrm{Con}_{\mathbb{BI} }(\mathbf{K})\}\cong%
\mathrm{Con}(\mathbf{M})\times\mathrm{Con}_{\mathbb{BI} }(\mathbf{K})$;

\item \label{mainthaol2} if $\mathbf{M}$ is nontrivial and $\mathbf{K}$ is a
pseudo-Kleene algebra, then $\mathrm{Con}_{\mathbb{BZL} }(\mathbf{L}%
)=\{\alpha\oplus\beta\oplus\alpha^{\prime}:\alpha\in\mathrm{Con}%
_{0}(\mathbf{M}),\beta\in\mathrm{Con}_{\mathbb{BI} }(\mathbf{K})\}\cup
\{\nabla_{L}\}\cong(\mathrm{Con}_{0}(\mathbf{M})\times\mathrm{Con}%
_{\mathbb{BI} }(\mathbf{K}))\oplus\mathbf{D}_{2}$;

\item \label{mainthaol3} $\mathbf{L}_{bi}$ is subdirectly irreducible iff one of the following holds:

\begin{itemize}
\item $\mathbf{M}$ is trivial and $\mathbf{K}$ is subdirectly irreducible;

\item $\mathbf{M}$ is subdirectly irreducible and $\mathbf{K}$ is trivial.
\end{itemize}

\item \label{mainthaol4} if $\mathbf{M}$ is nontrivial and $\mathbf{K}$ is a
pseudo-Kleene algebra, then the antiortholattice $\mathbf{L}$ is subdirectly
irreducible iff one of the following holds:

\begin{itemize}
\item $\mathrm{Con}_{0}(\mathbf{M})=\{\Delta_{M}\}$ and $\mathbf{K}$ is
subdirectly irreducible;

\item $\mathbf{K}$ is trivial and the set $\mathrm{Con}_{0}(\mathbf{M}%
)\setminus\{\Delta_{M}\}$ has a minimum.
\end{itemize}
\end{enumerate}

\label{mainthaol}
\end{theorem}

\begin{proof}
(\ref{mainthaol1}) For any $\alpha,\gamma\in\mathrm{Con}(\mathbf{M})$ and any
$\beta\in\mathrm{Con}(\mathbf{K})$, we have, according to the definition of
the involution of $\mathbf{L}$: $(\alpha\oplus\beta\oplus\gamma)^{\prime
}=\gamma^{\prime}\oplus\beta^{\prime}\oplus\alpha^{\prime}$, hence, by
Proposition \ref{bibzl}.(\ref{bibzl1}): $\alpha\oplus\beta\oplus\gamma
=(\alpha\oplus\beta\oplus\gamma)^{\prime}$ iff $\alpha=\gamma^{\prime}$,
$\beta=\beta^{\prime}$ and $\gamma=\alpha^{\prime}$ iff $\alpha=\gamma
^{\prime}$ and $\beta\in\mathrm{Con}_{\mathbb{BI}}(\mathbf{K})$, hence
$\mathrm{Con}_{\mathbb{BI}}(\mathbf{L})=\{\alpha\oplus\beta\oplus
\alpha^{\prime}:\alpha\in\mathrm{Con}(\mathbf{M}),\beta\in\mathrm{Con}%
_{\mathbb{BI}}(\mathbf{K})\}$, which is isomorphic to $\mathrm{Con}%
(\mathbf{M})\times\mathrm{Con}_{\mathbb{BI}}(\mathbf{K})$, because the map
$(\alpha,\beta)\mapsto\alpha\oplus\beta\oplus\alpha^{\prime}$ for all
$\alpha\in\mathrm{Con}(\mathbf{M})$ and $\beta\in\mathrm{Con}_{\mathbb{BI}%
}(\mathbf{K})$ sets a lattice isomorphism between these lattices, since it is clearly bijective and preserves the join and the intersection.

\noindent(\ref{mainthaol2}) By (\ref{mainthaol1}), Proposition \ref{bibzl}%
.(\ref{bibzl2}) and the clear fact that, if $\mathbf{M}$ is nontrivial,
then, for any $\alpha\in\mathrm{Con}(\mathbf{M})$ and any $\beta
\in\mathrm{Con}(\mathbf{K})$, we have: $0/(\alpha\oplus\beta\oplus
\alpha^{\prime})=\{0\}$ iff $0/\alpha=\{0\}$ (iff $1/\alpha^{\prime}=\{1\}$
iff $1/(\alpha\oplus\beta\oplus\alpha^{\prime})=\{1\}$).

\noindent(\ref{mainthaol3}),(\ref{mainthaol4}) By (\ref{mainthaol1}) and
(\ref{mainthaol2}), respectively.
\end{proof}

\begin{corollary}
Let $\mathbf{K}$ be a pseudo-Kleene algebra. Then:

\begin{enumerate}
\item \label{coraol1} for any $0$-regular, in particular any simple
nontrivial bounded lattice $\mathbf{M}$, if $\mathbf{L}=\mathbf{M}%
\oplus\mathbf{K}\oplus\mathbf{M}^{d}$, then $\mathrm{Con}_{\mathbb{BZL}%
}(\mathbf{L})=(\Delta_{M}\oplus\nabla_{K}\oplus\Delta
_{M}]\cup\{\nabla _{L}\}\cong\mathrm{Con}_{\mathbb{BI}}(\mathbf{K})\oplus\mathbf{D}_{2}$, and $\mathbf{L}$ is subdirectly irreducible as an antiortholattice iff $\mathbf{K}$ is subdirectly irreducible;

\item \label{coraol2} if $\mathbf{L}=\mathbf{D}_{2}\oplus\mathbf{K}\oplus\mathbf{D}_{2}$, then $\mathrm{Con}_{\mathbb{BZL}}(\mathbf{L})=(eq(\{0\},K,\{1\})]\cup\{\nabla_{L}\}\cong\mathrm{Con}_{\mathbb{BI}}(\mathbf{K})\oplus\mathbf{D}_{2}$, and $\mathbf{L}$ is subdirectly
irreducible as an antiortholattice iff $\mathbf{K}$ is subdirectly irreducible.
\end{enumerate}\label{coraol}\end{corollary}

\begin{proof}
(\ref{coraol1}) By Theorem \ref{mainthaol}.(\ref{mainthaol2}%
)-(\ref{mainthaol4}) and the fact that, in this case, $\mathrm{Con}%
_{0}(\mathbf{M})=\{\Delta_{M}\}\cong\mathbf{D}_{1}$, so $\{\alpha\oplus
\beta\oplus\alpha^{\prime}:\alpha\in\mathrm{Con}_{0}(\mathbf{M}),\beta
\in\mathrm{Con}_{\mathbb{BI}}(\mathbf{K})\}=(\Delta_{M}\oplus\nabla_{K}%
\oplus\Delta_{M}]$, as a principal filter of $\mathrm{Con}_{\mathbb{BZL}}(\mathbf{L})$.

\noindent(\ref{coraol2}) By (\ref{coraol1}), the equality $\Delta_{\mathbf{D}_{2}}\oplus\nabla_{K}\oplus\Delta_{\mathbf{D}_{2}}=eq(\{0\},K,\{1\})$ and the fact that $\mathbf{D}_{2}$.\end{proof}

We take advantage of this opportunity to correct a mistake in
\cite[Lm.3.3.(2)]{PBZ2}. There, it had been claimed that, if $\mathbf{L}$ is a subdirectly irreducible algebra in $V\left(  \mathbb{AOL}\right)  $, then every $a\in L$ is comparable with $a^{\prime}$. However, the canonical antiortholattice on $\mathbf{D}_{2}\oplus\mathbf{MO}_{2}\oplus\mathbf{D}_{2}$, where $\mathbf{MO}_{2}=\mathbf{D}_{2}^2\boxplus \mathbf{D}_{2}^2$ is the smallest orthomodular lattice which is not a Boolean algebra \cite{BH}, contains two pairs of incomparable elements $a,a^{\prime}$ and $b,b^{\prime}$, corresponding to the four atoms of $\mathbf{MO}_{2}$.

\begin{corollary} Let ${\bf K}$ be a BI--lattice. Then:\begin{enumerate}
\item\label{intermcoraol1} the BI--lattice ${\bf D}_2\oplus {\bf K}\oplus {\bf D}_2$ is subdirectly irreducible iff ${\bf K}$ is trivial;
\item\label{intermcoraol2} if ${\bf K}$ is a pseudo--Kleene algebra, then the antiortholattice ${\bf D}_3\oplus {\bf K}\oplus {\bf D}_3$ is subdirectly irreducible iff ${\bf K}$ is trivial.\end{enumerate}\label{intermcoraol}
\end{corollary}

\begin{proof} (\ref{intermcoraol1}) By Theorem \ref{mainthaol}.(\ref{mainthaol3}) and the fact that ${\bf D}_2$ is nontrivial and simple, thus subdirectly irreducible.

\noindent (\ref{intermcoraol2}) By (\ref{intermcoraol1}), Corollary \ref{coraol}.(\ref{coraol2}) and the fact that ${\bf D}_3\oplus {\bf K}\oplus {\bf D}_3\cong {\bf D}_2\oplus {\bf D}_2\oplus {\bf K}\oplus {\bf D}_2\oplus {\bf D}_2$.\end{proof}

\begin{corollary} The only simple antiortholattices that satisfy SDM are ${\bf D}_1$, ${\bf D}_2$ and ${\bf D}_3$.\label{splaolsdm}\end{corollary}

\begin{proof} Recall that an antiortholattice satisfies SDM iff it has the $0$ meet--irreducible, so any antiortholattice chain satisfies SDM. The antiortholattices ${\bf D}_1$ and ${\bf D}_2$ are simple and, by Corollary \ref{coraol}.(\ref{coraol2}), so is ${\bf D}_3={\bf D}_2\oplus {\bf D}_1\oplus {\bf D}_2$.

Now let ${\bf L}$ be a simple antiortholattice which satisfies SDM and assume ex absurdo that $|L|>3$. By Proposition \ref{bibzl}.(\ref{bibzl2}), ${\rm Con}_{\BZ }({\bf L})\cong {\rm Con}_{\BI 01}({\bf L})\oplus {\bf D}_2$, thus ${\rm Con}_{\BI 01}({\bf L})\cong {\bf D}_1$ since ${\bf L}$ is simple. But $0$ is meet--irreducible in ${\bf L}_l$, so $1$ is join--irreducible in ${\bf L}_l$, from which it easily follows that $\alpha =eq(\{0\},L\setminus \{0,1\},\{1\})\in {\rm Con}_{01}({\bf L})$ (see also \cite{eunoucard}); clearly, $\alpha $ preserves the Kleene complement of ${\bf L}$, hence $\alpha \in {\rm Con}_{\BI 01}({\bf L})$. Therefore $\Delta _L,\alpha \in {\rm Con}_{\BI 01}({\bf L})$, and, since $|L|>3$, it follows that $\Delta _L\neq \alpha $, which contradicts the fact that $|{\rm Con}_{\BI 01}({\bf L})|=1$.\end{proof}

Another proof of the previous corollary can be obtained from the results in \cite[Subsection 4.2]{rgcmfp}.

\begin{lemma} Any infinite chain $\mathbf{C}$ is subdirectly reducible. Moreover, $\Delta _C$ is meet--reducible in ${\rm Con}(\mathbf{C})$, as well as in ${\rm Con}_0(\mathbf{C})$ in the case when $\mathbf{C}$ has a bottom element.\label{infchains}\end{lemma}

\begin{proof} Let $\mathbf{C}$ be an infinite chain. Then there exist $a,b,c\in C$ such that $a<b<c$. If we denote by $\theta =eq(\{[a,b]\}\cup \{\{x\}:x\in C\setminus [a,b]\})$ and by $\zeta =eq(\{[b,c]\}\cup \{\{x\}:x\in C\setminus [b,c]\})$, then, clearly, $\theta ,\zeta \in {\rm Con}(\mathbf{C})\setminus \{\Delta _C\}$ and $\theta \cap \zeta =\Delta _C$. If $\mathbf{C}$ has a $0$, then we may take $a\neq 0$, and then $\theta ,\zeta \in {\rm Con}_0(\mathbf{C})$.\end{proof}

Note that, for any $n\in{\mathbb{N}}^{\ast}$, $\mathrm{Con}(\mathbf{D}_{n})\cong\mathbf{D}_{2}^{n-1}$. Indeed, $\mathrm{Con}(\mathbf{D}_{1})\cong\mathbf{D}_{1}\cong\mathbf{D}_{2}^0$, while, if $n\geq2$, then $\mathbf{D}_{n}=\bigoplus_{i=1}^{n-1}\mathbf{D}_{2}$, so that $\mathrm{Con}(\mathbf{D}_{n})\cong\mathrm{Con}(\mathbf{D}_{2})^{n-1}\cong\mathbf{D}_{2}^{n-1}$.

\begin{corollary}
\begin{enumerate}
\item \label{2ndcoraol1} For any $k\in{\mathbb{N}} ^{*}$ and any
$n\in\{2k,2k+1\}$, $\mathrm{Con}_{\mathbb{BI} }(\mathbf{D}_{n})\cong%
\mathbf{D}_{2}^{k}$ and $\mathrm{Con}_{\mathbb{BZL} }(\mathbf{D}_{n}%
)\cong\mathbf{D}_{2}^{k-1}\oplus\mathbf{D}_{2}$.

\item \label{2ndcoraol2} The only subdirectly irreducible (bounded) involution chains are $\mathbf{D}_{1}$, $\mathbf{D}_{2}$ and $\mathbf{D}_{3}$. The only
subdirectly irreducible antiortholattice chains are $\mathbf{D}_{1}$, $\mathbf{D}_{2}$, $\mathbf{D}_{3}$, $\mathbf{D}_{4}$ and $\mathbf{D}_{5}$.\end{enumerate}

\label{2ndcoraol}
\end{corollary}

\begin{proof}
(\ref{2ndcoraol1}) Note that $\mathbf{D}_{2k}\cong\mathbf{D}_{k}%
\oplus\mathbf{D}_{2}\oplus\mathbf{D}_{k}$ and $\mathbf{D}_{2k+1}%
\cong\mathbf{D}_{k+1}\oplus\mathbf{D}_{k+1}\cong\mathbf{D}_{k+1}%
\oplus\mathbf{D}_{1}\oplus\mathbf{D}_{k+1}$. By Theorem \ref{mainthaol}%
.(\ref{mainthaol1}), it follows that:

$\mathrm{Con}_{\mathbb{BI} }(\mathbf{D}_{2k})\cong\mathrm{Con}(\mathbf{D}%
_{k})\times\mathrm{Con}_{\mathbb{BI} }(\mathbf{D}_{2})\cong\mathbf{D}%
_{2}^{k-1}\times\mathbf{D}_{2}\cong\mathbf{D}_{2}^{k}$;

$\mathrm{Con}_{\mathbb{BI} }(\mathbf{D}_{2k+1})\cong\mathrm{Con}%
(\mathbf{D}_{k+1})\times\mathrm{Con}_{\mathbb{BI} }(\mathbf{D}_{1}%
)\cong\mathbf{D}_{2}^{k}\times\mathbf{D}_{1}\cong\mathbf{D}_{2}^{k}$.

Let us denote by $\lfloor r\rfloor$ the integer part of any real number $r$,
so that $k=\lfloor n/2\rfloor$. Of course, $\mathrm{Con}_{\mathbb{BI}
}(\mathbf{D}_{1})\cong\mathbf{D}_{1}\cong\mathbf{D}_{2}^{\lfloor1/2\rfloor}$
and $\mathrm{Con}_{\mathbb{BZL} }(\mathbf{D}_{2})\cong\mathbf{D}_{2}%
\cong\mathbf{D}_{2}^{\lfloor2/2\rfloor-1}\oplus\mathbf{D}_{2}$.

If $n\geq3$, then $\mathbf{D}_{n}\cong\mathbf{D}_{2}\oplus\mathbf{D}%
_{n-2}\oplus\mathbf{D}_{2}$ and $n-2\in\{2(k-1),2(k-1)+1\}$, hence, by
Corollary \ref{coraol}.(\ref{coraol2}), $\mathrm{Con}_{\mathbb{BZL}
}(\mathbf{D}_{n})\cong\mathrm{Con}_{\mathbb{BI} }(\mathbf{D}_{n-2}%
)\oplus\mathbf{D}_{2}\cong\mathbf{D}_{2}^{k-1}\oplus\mathbf{D}_{2}$.

\noindent(\ref{2ndcoraol2}) $\mathbf{D}_{1}$ is trivial, thus subdirectly irreducible both as a BI--lattice and as an antiortholattice, while (\ref{2ndcoraol1}) ensures us that, for any $n\in{\mathbb{N}}$ with $n\geq2$, if $k=\lfloor n/2\rfloor\in{\mathbb{N}}^{\ast}$, then:

\begin{itemize}
\item $\mathrm{Con}_{\mathbb{BI} }(\mathbf{D}_{n})\cong\mathbf{D}_{2}^{k}$,
which has exactly $k$ atoms, so that: $(\mathbf{D}_{n})_{bi}$ is subdirectly
irreducible iff $k=1$ iff $n\in\{2,3\}$;

\item $\mathrm{Con}_{\mathbb{BZL}}(\mathbf{D}_{n})\cong\mathbf{D}_{2}%
^{k-1}\oplus\mathbf{D}_{2}$, which has one atom if $k\leq2$ and $k-1$ atoms if
$k>2$, so that the antiortholattice $\mathbf{D}_{n}$ is subdirectly
irreducible iff $k-1\leq1$ iff $k\leq2$ iff $n\in[2,5]$.
\end{itemize}

Hence the only subdirectly irreducible finite bounded involution chains are $\mathbf{D}_{1}$, $\mathbf{D}_{2}$ and $\mathbf{D}_{3}$, while the only subdirectly irreducible finite antiotholattice chains are $\mathbf{D}_{1}$, $\mathbf{D}_{2}$, $\mathbf{D}_{3}$, $\mathbf{D}_{4}$ and $\mathbf{D}_{5}$.

If $\mathbf{C}$ is an infinite BI--chain, then its $0$ is meet--irreducible, hence $\mathbf{C}$ is an antiortholattice, and there exists a $u\in C$ such that $u<u^{\prime }$ and the filter $(u]$ is infinite, thus $(u]$ is an infinite bounded chain and $\mathbf{C}=(u]\oplus [u,u^{\prime }]\oplus [u^{\prime })=(u]\oplus [u,u^{\prime }]\oplus (u]^d$, where $[u,u^{\prime }]$ is clearly a bounded lattice. Then, by Lemma \ref{infchains} and Theorem \ref{mainthaol}.(\ref{mainthaol1}).(\ref{mainthaol2}), $\Delta _{(u]}$ is meet--reducible in ${\rm Con}((u])$, as well as in ${\rm Con}_0((u])$, hence $\Delta _C$ is meet--reducible in ${\rm Con}_{\BI }(\mathbf{C})\cong {\rm Con}((u])\times {\rm Con}_{\BI }([u,u^{\prime }])$, as well as in ${\rm Con}_{\BZ }(\mathbf{C})\cong ({\rm Con}_0((u])\times {\rm Con}_{\BI }([u,u^{\prime }]))\oplus {\bf D}_2$, and hence $\mathbf{C}$ is subdirectly reducible both in $\BI$ and in $\PBZs$.

Note that the argument above can be adapted for the subdirect irreducibility of any infinite involution chain $\mathbf{C}$, not necessarily bounded, in the class $\I $ of involution lattices, that is self--dual lattices $\mathbf{L}$ endowed with a unary operation $^{\prime }$ given by a dual lattice automorphism of $\mathbf{L}$, because, for any lattice $\mathbf{M}$ with a $1$ and any BI--lattice $\mathbf{K}$, $\mathbf{L}=\mathbf{M}\oplus \mathbf{K}\oplus \mathbf{M}^d$ is an involution lattice with ${\rm Con}_{\I }(\mathbf{L})\cong {\rm Con}(\mathbf{M})\times {\rm Con}_{\BI }(\mathbf{K})$; see also \cite{eucardbi}.\end{proof}

\begin{corollary} Let ${\bf C}$ be a bounded chain, ${\bf K}$ a BI--lattice, and $\mathbf{L}=\mathbf{C}\oplus\mathbf{K}\oplus\mathbf{C}^d$.

\begin{itemize}
\item \label{alsocoraol1} If $|C|\geq3$, then the BI--lattice $\mathbf{L}$ is subdirectly reducible.
\item \label{alsocoraol2} If $|C|\geq 4$ and $\mathbf{K}$ is a pseudo-Kleene algebra, then the antiortholattice $\mathbf{L}$
is subdirectly reducible.\end{itemize}\label{alsocoraol}\end{corollary}

\begin{proof} If $C$ is finite, then the statements follow by Corollary \ref{intermcoraol} and the fact that, if $n\geq 3$, then $n-1\geq 2$, so that the BI--lattice ${\bf D}_{n-1}\oplus {\bf K}\oplus {\bf D}_{n-1}$ is nontrivial, and we have $\mathbf{D}_{n}\oplus\mathbf{K}\oplus\mathbf{D}_{n}\cong {\bf D}_2\oplus {\bf D}_{n-1}\oplus {\bf K}\oplus {\bf D}_{n-1}\oplus {\bf D}_2$ and $\mathbf{D}_{n+1}\oplus\mathbf{K}\oplus\mathbf{D}_{n+1}\cong {\bf D}_3\oplus {\bf D}_{n-1}\oplus {\bf K}\oplus {\bf D}_{n-1}\oplus {\bf D}_3$.

If $C$ is infinite, then the subdirect reducibility of $\mathbf{L}$ follows by the argument at the end of the proof of Corollary \ref{2ndcoraol}.(\ref{2ndcoraol2}) in which we replace $(u]$ by $\mathbf{C}$ and $[u,u^{\prime }]$ by $\mathbf{K}$.\end{proof}

\section{Horizontal Sums of PBZ$^{\ast}$--Lattices\label{guru}}

There is a well-developped theory of{ \emph{horizontal sums}} in the context of{ orthomodular lattices: see e.g. \cite{Beran, Chajdor, Greechie}. In the present section, we follow in the footsteps of \cite{PBZ2} and try to broaden our scope to the context of PBZ*-lattices.

Of the next two results, the former is straightforward and the latter is implicit in {\cite[Ex. 5.3]{GLP}:}

\begin{lemma}
\begin{enumerate}
\item \label{pstar1} For any nontrivial BI-lattices $\mathbf{A}$ and
$\mathbf{B}$: $\mathbf{A}\boxplus\mathbf{B}$ is paraorthomodular iff
$\mathbf{A}$ and $\mathbf{B}$ are paraorthomodular.

\item \label{pstar0} For any nontrivial BI-lattices $\mathbf{A}$ and
$\mathbf{B}$: $\mathbf{A}\boxplus\mathbf{B}$ is an ortholattice, respectively
an orthomodular lattice, iff $\mathbf{A}$ and $\mathbf{B}$ are ortholattices,
respectively orthomodular lattices.

\item \label{pstar2} For any nontrivial BZ--lattices $\mathbf{A}$ and $\mathbf{B}$:
$\mathbf{A}\boxplus\mathbf{B}$ satisfies condition $(\ast)$ iff $\mathbf{A}$
and $\mathbf{B}$ satisfy condition $(\ast)$.

\item \label{pstar3} For any nontrivial BZ--lattices $\mathbf{A}$ and $\mathbf{B}$ such
that $\mathbf{A}\boxplus\mathbf{B}$ is a BZ--lattice: $\mathbf{A}%
\boxplus\mathbf{B}$ is a PBZ$^{\ast}$ --lattice iff $\mathbf{A}$ and
$\mathbf{B}$ are PBZ$^{\ast}$ --lattices.
\end{enumerate}

\label{pstar}
\end{lemma}

\begin{proposition}
\begin{enumerate}
\item \label{hsumkl1} If $\mathbf{A}$ and $\mathbf{B}$ are nontrivial
pseudo--Kleene algebras, then: $\mathbf{A}\boxplus\mathbf{B}$ is a
pseudo--Kleene algebra iff at least one of $\mathbf{A}$ and $\mathbf{B}$ is an ortholattice.

\item \label{hsumkl0} If $\mathbf{A}$ and $\mathbf{B}$ are nontrivial
BZ--lattices, then: $\mathbf{A}\boxplus\mathbf{B}$ is a BZ--lattice iff at
least one of $\mathbf{A}_{bi}$ and $\mathbf{B}_{bi}$ is an ortholattice.

\item \label{hsumkl2} If $\mathbf{A}$ and $\mathbf{B}$ are nontrivial
PBZ$^{\ast}$--lattices, then: $\mathbf{A}\boxplus\mathbf{B}$ is a PBZ$^{\ast}%
$-lattice iff at least one of $\mathbf{A}$ and $\mathbf{B}$ is an orthomodular lattice.
\end{enumerate}

\label{hsumkl}
\end{proposition}

The following corollaries ensue. The lesson we learn from the latter is that
classes of the form $\mathbb{V}\boxplus\mathbb{W}$, for $\mathbb{V}%
,\mathbb{W}$ subvarieties of $\mathbb{PBZL}^{\ast}$, are \emph{sometimes}
varieties in their own right, and in particular, well-known subvarieties of
$\mathbb{PBZL}^{\ast}$.

\begin{corollary}
\begin{itemize}
\item If $n\in{\mathbb{N}}\setminus\{0,1\}$ and $\mathbf{A}_{1},\ldots
,\mathbf{A}_{n}$ are nontrivial pseudo--Kleene algebras, then:
$\boxplus_{i=1}^{n}\mathbf{A}_{i}$ is a pseudo--Kleene algebra iff, for some
$k\in[1,n]$ and every $i\in[1,n]\setminus\{k\}$,
$\mathbf{A}_{i}$ is an ortholattice.

\item If $n\in{\mathbb{N}}\setminus\{0,1\}$ and $\mathbf{A}_{1},\ldots
,\mathbf{A}_{n}$ are nontrivial BZ--lattices, then: $\boxplus_{i=1}%
^{n}\mathbf{A}_{i}$ is a BZ--lattice iff, for some $k\in[1,n]$ and
every $i\in[1,n]\setminus\{k\}$, $(\mathbf{A}_{i})_{bi}$ is an ortholattice.

\item If $n\in{\mathbb{N}}\setminus\{0,1\}$ and $\mathbf{A}_{1},\ldots
,\mathbf{A}_{n}$ are nontrivial PBZ$^{\ast}$--lattices, then: $\boxplus
_{i=1}^{n}\mathbf{A}_{i}$ is a PBZ$^{\ast}$--lattice iff, for some
$k\in[1,n]$ and every $i\in[1,n]\setminus\{k\}$,
$\mathbf{A}_{i}$ is an orthomodular lattice.
\end{itemize}
\end{corollary}

\begin{corollary}
\begin{enumerate}
\item \label{cors1} $\mathbb{OL}\boxplus\mathbb{PKA}=\mathbb{PKA}$ and
$\mathbb{OML}\boxplus\mathbb{PBZL}^{\ast}=\mathbb{PBZL}^{\ast}$.

\item \label{olhsum} $\mathbb{OL}\boxplus\mathbb{OL}=\mathbb{OML}%
\boxplus\mathbb{OL}=\mathbb{OL}$ and $\mathbb{OML}\boxplus\mathbb{OML}%
=\mathbb{OML}$.

\item \label{particdistrib0} For any classes ${\mathbb{C}}$ and ${\mathbb{D}}$
of BZ--lattices such that ${\mathbb{C}}\boxplus{\mathbb{D}}\subseteq
\mathbb{BZL}$, $({\mathbb{C}}\boxplus{\mathbb{D}})\cap\mathbb{PBZL}^{\ast
}=({\mathbb{C}}\cap\mathbb{PBZL}^{\ast})\boxplus({\mathbb{D}}\cap
\mathbb{PBZL}^{\ast})$.
\end{enumerate}

\label{cors}
\end{corollary}

\begin{proof}
The right-to-left inclusions follow from the fact that $\mathbf{D}_{2}%
\in\mathbb{OL}\subseteq\mathbb{OML}$. The left-to-right inclusions are
consequences of Lemma \ref{pstar} and Proposition \ref{hsumkl}.
\end{proof}

\begin{lemma}
\label{hsumsubquo}If ${\mathbb{V}}$ is the variety of bounded lattices or one
of the varieties $\mathbb{BI}$ and $\mathbb{BZL}$ and $\mathbf{A}$ and
$\mathbf{B}$ are nontrivial members of ${\mathbb{V}}$ such that
$\mathbf{A}\boxplus\mathbf{B}\in{\mathbb{V}}$, then, for any subalgebra
$\mathbf{M}$ of $\mathbf{A}\boxplus\mathbf{B}$ and any $\theta\in
\mathrm{Con}_{{\mathbb{V}}}(\mathbf{A}\boxplus\mathbf{B})$, we have:

\begin{itemize}
\item $\mathbf{M}=(\mathbf{M}\cap\mathbf{A})\boxplus(\mathbf{M}\cap
\mathbf{B})$;

\item $(\mathbf{A}\boxplus\mathbf{B})/\theta=\mathbf{A}/\theta\boxplus
\mathbf{B}/\theta$.
\end{itemize}
\end{lemma}

\begin{proof}
Since $\mathbf{A}$ and $\mathbf{B}$ are subalgebras of $\mathbf{A}%
\boxplus\mathbf{B}$, it follows that $\mathbf{M}\cap\mathbf{A}$ and
$\mathbf{M}\cap\mathbf{B}$ are subalgebras of $\mathbf{M}$. We have
$M=M\cap(A\boxplus B)=M\cap(A\cup B)=(M\cap A)\cup(M\cap B)$. Since $A\cap
B=\{0,1\}$, it follows that $(M\cap A)\cap(M\cap B)=\{0,1\}$. Therefore
$\mathbf{M}=(\mathbf{M}\cap\mathbf{A})\boxplus(\mathbf{M}\cap\mathbf{B})$ by
the definition of a horizontal sum.

Since $\theta\in\mathrm{Con}_{{\mathbb{V}}}(\mathbf{A}\boxplus\mathbf{B})$ and
$\mathbf{A}$ and $\mathbf{B}$ are subalgebras of $\mathbf{A}\boxplus
\mathbf{B}$, it follows that $\theta\cap A^{2}\in\mathrm{Con}_{{\mathbb{V}}%
}(\mathbf{A})$, $\theta\cap B^{2}\in\mathrm{Con}_{{\mathbb{V}}}(\mathbf{B})$
and $\mathbf{A}/\theta=\mathbf{A}/(\theta\cap A^{2})$ and $\mathbf{B}%
/\theta=\mathbf{B}/(\theta\cap B^{2})$ are subalgebras of $(\mathbf{A}%
\boxplus\mathbf{B})/\theta$. $A\cap B=\{0,1\}$, hence $A/\theta\cap
B/\theta=(A\cap B)/\theta=\{0/\theta,1/\theta\}$. By the definition of the
horizontal sum, it follows that $(\mathbf{A}\boxplus\mathbf{B})/\theta
=\mathbf{A}/\theta\boxplus\mathbf{B}/\theta$.
\end{proof}

\begin{lemma}
For any PBZ$^{\ast}$ --lattice $\mathbf{L}$, any subalgebra $\mathbf{M}$ of
$\mathbf{L}$, any $\theta\in\mathrm{Con}_{\mathbb{BZL}}(\mathbf{L})$, any
nontrivial orthomodular lattice $\mathbf{A}$, any nontrivial PBZ$^{\ast}$
--lattice $\mathbf{B}$ and any non--empty family $(\mathbf{L}_{i})_{i\in I}$
of PBZ$^{\ast}$ --lattices, we have:

\begin{itemize}
\item $\mathbf{S}(\mathbf{M})=\mathbf{S}(\mathbf{L})\cap\mathbf{M}$;

\item $\mathbf{S}(\mathbf{L})/\theta=\mathbf{S}(\mathbf{L}/\theta)$;

\item $\mathbf{S}(\mathbf{A}\boxplus\mathbf{B})=\mathbf{A}\boxplus
\mathbf{S}(\mathbf{B})$;

\item $\mathbf{S}(\prod_{i\in I}\mathbf{L}_{i})=\prod_{i\in I}\mathbf{S}%
(\mathbf{L}_{i})$.
\end{itemize}

\label{ssubhsum}
\end{lemma}

\begin{proof}
$S(\mathbf{M})=\{x\in M:x^{\prime}=x^{\sim}\}=S(\mathbf{L})\cap M$, so
$\mathbf{S}(\mathbf{M})=\mathbf{S}(\mathbf{L})\cap\mathbf{M}$.

Clearly, $S(\mathbf{L})/\theta\subseteq S(\mathbf{L}/\theta)$. Now let $x\in
L$ be such that $x/\theta\in S(\mathbf{L}/\theta)$, that is $(x/\theta
)^{\prime}=(x/\theta)^{\sim}$. Then $x/\theta=x^{\prime\prime}/\theta
=(x^{\prime}/\theta)^{\prime}=(x^{\sim}/\theta)^{\prime}=x^{\sim\prime}%
/\theta=x^{\sim\sim}/\theta\in S(\mathbf{L})/\theta$. Hence $S(\mathbf{L}%
/\theta)\subseteq S(\mathbf{L})/\theta$. Therefore $S(\mathbf{L}%
)/\theta=S(\mathbf{L}/\theta)$, so $\mathbf{S}(\mathbf{L})/\theta
=\mathbf{S}(\mathbf{L}/\theta)$.

On the other hand, by Proposition \ref{hsumkl}.(\ref{hsumkl2}), Lemma \ref{hsumsubquo} and the definition of the subalgebra of sharp elements of a PBZ$^{\ast}$ --lattice:$$\mathbf{S}(\mathbf{A}\boxplus\mathbf{B})=(\mathbf{S}(\mathbf{A}\boxplus
\mathbf{B})\cap\mathbf{A})\boxplus(\mathbf{S}(\mathbf{A}\boxplus
\mathbf{B})\cap\mathbf{B})=\mathbf{S}(\mathbf{A})\boxplus\mathbf{S}(\mathbf{B})=\mathbf{A}\boxplus\mathbf{S}(\mathbf{B}).
$$\end{proof}

Next, let us give a direct proof of a result from \cite{GLP}, to the effect that the orthomodular lattice of sharp elements in a member of $V(\mathbb{AOL})$ is always Boolean.

\begin{proposition} If $\mathbf{L}\in V(\mathbb{AOL})$, then $\mathbf{S}(\mathbf{L})$ is a Boolean algebra.\label{svaolbool}\end{proposition}

\begin{proof} We will apply Lemma \ref{ssubhsum}. If $\mathbf{L}\in V(\mathbb{AOL})=\H \S \P (\mathbb{AOL})$, then there exists a non--empty family $(\mathbf{A}_{i})_{i\in I}\subseteq\mathbb{AOL}\setminus\{\mathbf{D}_{1}\}$, a subalgebra
$\mathbf{A}$ of $\prod_{i\in I}\mathbf{A}_{i}$ and a $\theta\in\mathrm{Con}_{\mathbb{BZL}}(\mathbf{A})$ such that $\mathbf{L}=\mathbf{A}/\theta$. Then,
for all $i\in I$, $S(\mathbf{A}_{i})=\{0,1\}$, so the orthomodular lattice
$\mathbf{S}(\mathbf{A}_{i})\cong\mathbf{D}_{2}$, thus $\mathbf{S}(\prod_{i\in
I}\mathbf{A}_{i})=\prod_{i\in I}\mathbf{S}(\mathbf{A}_{i})\cong\mathbf{D}_{2}^{I}$, which is a Boolean algebra, hence $\mathbf{S}(\mathbf{A})=\mathbf{S}(\prod_{i\in I}\mathbf{A}_{i})\cap\mathbf{A}$ is embedded in the
Boolean algebra $\mathbf{S}(\prod_{i\in I}\mathbf{A}_{i})$, therefore $\mathbf{S}(\mathbf{A})$ is a Boolean algebra, thus $\mathbf{S}(\mathbf{L})=\mathbf{S}(\mathbf{A}/\theta)=\mathbf{S}(\mathbf{A})/\theta=\mathbf{S}(\mathbf{A})/(\theta\cap(S(\mathbf{A}))^{2})$ is a Boolean algebra.\end{proof}

Note that, since $\mathbf{S}(\mathbf{L})$ is the largest orthomodular subalgebra in any $\mathbf{L}\in \PBZs $, Proposition \ref{svaolbool} shows that, for any $\mathbf{L}\in V(\AOL )$, any orthomodular subalgebra of $\mathbf{L}$ is Boolean.

\begin{corollary}
\begin{enumerate}
\item \label{somlhvaol1} For any $\mathbf{L}\in\mathbb{OML} \boxplus
V(\mathbb{AOL} )$, $\mathbf{S}(\mathbf{L})$ is a horizontal sum of an
orthomodular lattice with a Boolean algebra.

\item \label{somlhvaol2} $\{\mathbf{L}\in\mathbb{OML} \boxplus\mathbb{AOL}
:\mathbf{S}(\mathbf{L})\in\mathbb{BA} \}=\mathbb{BA} \boxplus\mathbb{AOL} $
and $\{\mathbf{L}\in\mathbb{OML} \boxplus V(\mathbb{AOL} ):\mathbf{S}%
(\mathbf{L})\in\mathbb{BA} \}=(\mathbb{BA} \boxplus\mathbb{AOL} )\cup
V(\mathbb{AOL} )$, but $\{\mathbf{L}\in\mathbb{PBZL}^{\ast} :\mathbf{S}%
(\mathbf{L})\in\mathbb{BA} \}\nsubseteq V(\mathbb{OML} \boxplus V(\mathbb{AOL}
))$.
\end{enumerate}

\label{somlhvaol}
\end{corollary}

\begin{proof}
(\ref{somlhvaol1}) Let $\mathbf{L}=\mathbf{A}\boxplus\mathbf{B}$, with
$\mathbf{A}\in\mathbb{OML} $ and $\mathbf{B}\in V(\mathbb{AOL} )$, thus, by
Lemma \ref{ssubhsum} and Proposition \ref{svaolbool}, $\mathbf{S}%
(\mathbf{L})=\mathbf{A}\boxplus\mathbf{S}(\mathbf{B})$, with $\mathbf{A}%
\in\mathbb{OML} $ and $\mathbf{S}(\mathbf{B})\in\mathbb{BA} $.

(\ref{somlhvaol2}) Clearly, for any bounded lattices $\mathbf{A}$ and
$\mathbf{B}$ with $|A|>2$ and $|B|>3$, $\mathbf{A}\boxplus\mathbf{B}$ has the
diamond or the pentagon (the latter if $\mathbf{A}$ or $\mathbf{B}$ has length
at least $4$) as a bounded sublattice, thus $\mathbf{A}\boxplus\mathbf{B}$ is
non--distributive. The horizontal sum of BI-lattices $\mathbf{D}_{3}\boxplus\mathbf{D}_{3}\ncong\mathbf{D}_{2}^{2}$, in fact $\mathbf{D}_{3}\boxplus\mathbf{D}_{3}\notin\mathbb{PKA}$. Hence, for any BI-lattices
$\mathbf{A}$ and $\mathbf{B}$ with $|A|>2$ and $|B|>2$, $\mathbf{A}\boxplus\mathbf{B}$ is not a Boolean algebra, more precisely, for any nontrivial BI-lattices $\mathbf{A}$ and $\mathbf{B}$, $\mathbf{A}\boxplus\mathbf{B}\in\mathbb{BA}$ iff $\mathbf{A}\cong\mathbf{D}_{2}$ and $\mathbf{B}\in\mathbb{BA}$ or vice--versa.

Now let $\mathbf{A}\in\mathbb{OML} \setminus\{\mathbf{D}_{1}\}$,
$\mathbf{B}\in\mathbb{PBZL}^{\ast} $ and $\mathbf{L}=\mathbf{A}\boxplus
\mathbf{B}$, so that $\mathbf{S}(\mathbf{L})=\mathbf{A}\boxplus\mathbf{S}%
(\mathbf{B})$ by Lemma \ref{ssubhsum}, hence, by the above, $\mathbf{S}%
(\mathbf{L})\in\mathbb{BA} $ iff $\mathbf{A}\cong\mathbf{D}_{2}$ and
$\mathbf{S}(\mathbf{B})\in\mathbb{BA} $ or $\mathbf{A}\in\mathbb{BA} $ and
$S(\mathbf{B})=\{0,1\}$. Now apply the fact that $\{\mathbf{B}\in
\mathbb{PBZL}^{\ast} :S(\mathbf{B})=\{0,1\}\}=\mathbb{AOL} $ and Proposition
\ref{svaolbool}.

See below the PBZ$^{\ast}$ --lattice $\mathbf{M}$ in Example \ref{exfail12}, which has $S(\mathbf{M})=\{0,a,a^{\prime},1\}$, so $\mathbf{D}_{2}^{2}\cong\mathbf{S}(\mathbf{M})\in\mathbb{BA}$, but $\mathbf{M}\notin V(\mathbb{OML} \boxplus V(\mathbb{AOL} ))$.\end{proof}

\begin{lemma}
If $\mathbf{A}$ is an ortholattice with $|A|>2$ and $\mathbf{B}$ is a non--trivial BI-lattice, then $\mathbf{A}\boxplus\mathbf{B}$, endowed with the trivial Brouwer complement, fails condition $(\ast)$.\label{hsumbi}\end{lemma}

\begin{proof}
Since $|A|>2$, there exists an $a\in A\setminus\{0,1\}$, so that $a^{\prime
}\in A\setminus\{0,1\}$, as well. Since $\mathbf{A}$ is an ortholattice, we
have $a\wedge a^{\prime}=0$. Thus, if $^{\sim}:A\boxplus B\rightarrow
A\boxplus B$ is the trivial Brouwer complement, then $(a\wedge a^{\prime
})^{\sim}=0^{\sim}=1\neq0=0\vee0=a^{\sim}\vee a^{\prime\sim}$.
\end{proof}

\begin{lemma}
Let $\mathbf{A}$ and $\mathbf{B}$ be PBZ$^{\ast}$ --lattices with $|A|>2$ and
$|B|>2$. Then:

\begin{enumerate}
\item \label{notomlaol1} $\mathbf{A}\boxplus\mathbf{B}$ is not an antiortholattice;

\item \label{notomlaol2} $\mathbf{A}\boxplus\mathbf{B}$ is an orthomodular
lattice iff $\mathbf{A}$ and $\mathbf{B}$ are orthomodular lattices.
\end{enumerate}

\label{notomlaol}
\end{lemma}

\begin{proof}
(\ref{notomlaol1}) By Proposition \ref{hsumkl}.(\ref{hsumkl2}), and Lemma
\ref{hsumbi}.

(\ref{notomlaol2}) By Lemma \ref{uniquesim} and the fact that
$\{0,1\}\subseteq S(\mathbf{A})\cap S(\mathbf{B})$, we have: $\mathbf{A}$ and
$\mathbf{B}$ are orthomodular lattices iff $S(\mathbf{A})=A$ and
$S(\mathbf{B})=B$ iff $S(\mathbf{A}\boxplus\mathbf{B})=A\boxplus B$ iff
$\mathbf{A}\boxplus\mathbf{B}$ is an orthomodular lattice.
\end{proof}

The following lemma clarifies the relationships between dense elements on the
one hand, and subalgebras, products and congruences on the other, in
PBZ$^{\ast}$ --lattices.

\begin{lemma}
\label{prodsubquo}For any PBZ$^{\ast}$--lattice $\mathbf{L}$, any subalgebra
$\mathbf{M}$ of $\mathbf{L}$, any $\theta\in$\linebreak$\mathrm{Con}%
_{\mathbb{BZL}}(\mathbf{L})$, any nontrivial orthomodular lattice
$\mathbf{A}$, any nontrivial PBZ$^{\ast}$ --lattice $\mathbf{B}$ and any
non--empty family $(\mathbf{L}_{i})_{i\in I}$ of PBZ$^{\ast}$ --lattices, we have:

\begin{itemize}
\item $T(\mathbf{M})=T(\mathbf{L})\cap M$ and $\langle T(\mathbf{M}%
)\rangle_{\mathbb{BZL} }=\langle T(\mathbf{L})\rangle_{\mathbb{BZL} }%
\cap\mathbf{M}$;

\item $T(\mathbf{L})/\theta\subseteq T(\mathbf{L}/\theta)$;

\item if $\theta\in\mathrm{Con}_{\mathbb{BZL}01}(\mathbf{L})$, then
$T(\mathbf{L})/\theta=T(\mathbf{L}/\theta)$;

\item $T(\mathbf{A}\boxplus\mathbf{B})=T(\mathbf{B})$, so $\mathbf{T}%
(\mathbf{A}\boxplus\mathbf{B})=\mathbf{T}(\mathbf{B})$ if $\mathbf{T}%
(\mathbf{B})$ is a subalgebra of $\mathbf{B}$;

\item $D(\prod_{i\in I}\mathbf{L}_{i})=\prod_{i\in I}(D(\mathbf{L}_{i}))$, so
$T(\prod_{i\in I}\mathbf{L}_{i})=\left\{  0^{\prod_{i\in I}\mathbf{L}_{i}%
}\right\}  \cup\prod_{i\in I}(D(\mathbf{L}_{i}))$;

\item $\langle T(\prod_{i\in I}\mathbf{L}_{i})\rangle_{\mathbb{BZL}}%
=\prod_{i\in I}\langle T(\mathbf{L}_{i})\rangle_{\mathbb{BZL}}$.
\end{itemize}
\end{lemma}

\begin{proof}
$T(\mathbf{M})=\{0\}\cup D(\mathbf{M})=T(\mathbf{L})\cap M$, hence $\langle
T(\mathbf{M})\rangle_{\mathbb{BZL}}=\langle T(\mathbf{L})\cap M\rangle
_{\mathbb{BZL}}=\langle T(\mathbf{L})\rangle_{\mathbb{BZL}}\cap\mathbf{M}$.

Clearly, $T(\mathbf{L})/\theta\subseteq T(\mathbf{L}/\theta)$ and, if
$\theta\in\mathrm{Con}_{\mathbb{BZL}01}(\mathbf{L})$, then, for any $x\in L$
such that $x/\theta\in T(\mathbf{L}/\theta)$, we have $x\in0/\theta=\{0\}$ or
$x^{\sim}\in0/\theta=\{0\}$, so $x\in T(\mathbf{L})$, thus $T(\mathbf{L}%
/\theta)\subseteq T(\mathbf{L})/\theta$.

By Proposition \ref{hsumkl}.(\ref{hsumkl2}), $\mathbf{A}\boxplus\mathbf{B}$ is
a PBZ$^{\ast}$--lattice. Since $\mathbf{A},\mathbf{B}\in S\left(
\mathbf{A}\boxplus\mathbf{B}\right)  $, we have, by Lemma \ref{olsaolt},
$T(\mathbf{A}\boxplus\mathbf{B})=T(\mathbf{A})\cup T(\mathbf{B})=\{0,1\}\cup
T(\mathbf{B})=T(\mathbf{B})$, which is a subuniverse of $\mathbf{A}\boxplus\mathbf{B}$ if it is a subuniverse of $\mathbf{B}$.

Clearly, $D(\prod_{i\in I}\mathbf{L}_{i})=\prod_{i\in I}(D(\mathbf{L}_{i}))$,
whence the rest of the statement follows. Therefore

\begin{center}\begin{tabular}{ll}
$\langle T(\prod_{i\in I}\mathbf{L}_{i})\rangle_{\mathbb{BZL}}$  &  $=\langle
D(\prod_{i\in I}\mathbf{L}_{i})\rangle_{\mathbb{BZL}}=\langle\prod_{i\in I}D(\mathbf{L}_{i})\rangle_{\mathbb{BZL}}$\\
&  $=\prod_{i\in I}\langle D(\mathbf{L}_{i})\rangle_{\mathbb{BZL}}=\prod_{i\in I}\langle T(\mathbf{L}_{i})\rangle_{\mathbb{BZL}}.$\end{tabular}\end{center}\end{proof}

Let us strenghthen the property mentioned at the end of Subsection \ref{pbzl} which characterizes antiortholattices with SDM:

\begin{proposition}
Let $(\mathbf{L}_{i})_{i\in I}$ be a non--empty family of nontrivial
PBZ$^{\ast}$ --lattices, $\mathbf{L}=\prod_{i\in I}\mathbf{L}_{i}$ and
$\mathbf{A}$ be a subalgebra of $\mathbf{L}$ such that $\mathbf{A}$ is an
antiortholattice. Then:

\begin{itemize}
\item if $a_{i}\in L_{i}$ for all $i\in I$ such that $a=(a_{i})_{i\in I}\in
A$, then: $a=0$ or $a_{i}\neq0$ for all $i\in I$, and, dually: $a=1$ or
$a_{i}\neq1$ for all $i\in I$;

\item if, for every $i\in I$ and all $x_{i},y_{i}\in T(\mathbf{L}%
_{i})\setminus\{0\}$, we have $x_{i}\wedge y_{i}\neq0$, in particular if
$T(\mathbf{L}_{i})\setminus\{0\}$ is closed w.r.t. the meet or $0$ is
meet--irreducible in $\mathbf{L}_{i}$ for every $i\in I$, then $0$ is not a
finite meet of elements of $T(\mathbf{L})\setminus\{0\}$, in particular $0$ is
meet--irreducible in $\mathbf{A}$.
\end{itemize}

\label{aols0irr}
\end{proposition}

\begin{proof}
By Lemma \ref{prodsubquo}, which ensures us that $A=T(\mathbf{A})\subseteq
T(\mathbf{L})=\{0\}\cup\prod_{i\in I}(T(\mathbf{L}_{i})\setminus\{0\})$ and
thus $A\setminus\{0\}\subseteq T(\mathbf{L})\setminus\{0\}=\prod_{i\in
I}(T(\mathbf{L}_{i})\setminus\{0\})$.
\end{proof}

We now focus on the properties of the sets of sharp elements and of dense
elements in some particular horizontal sums. We show that in any horizontal
sum of an orthomodular lattice and of an antiortholattice, the former includes
all the sharp elements and the latter all the dense elements; moreover,
horizontal sums of orthomodular lattices and of antiortholattices are
\emph{exactly} the PBZ$^{\ast}$--lattices $\mathbf{L}$ such that
$S(\mathbf{L})\cup T(\mathbf{L})=L$.

\begin{lemma}
\label{hsumolaol}If $\mathbf{A}$ is an orthomodular lattice and $\mathbf{B}$ is an antiortholattice, then $\mathbf{S}(\mathbf{A}\boxplus\mathbf{B})=\mathbf{A}$ and $\mathbf{T}(\mathbf{A}\boxplus\mathbf{B})=\mathbf{B}$.\end{lemma}

\begin{proof}
$\mathbf{A}\boxplus\mathbf{B}\in\mathbb{PBZL}^{\ast}$ by Proposition
\ref{hsumkl}.(\ref{hsumkl2}). By Lemmas \ref{ssubhsum} and \ref{olsaolt},
$\mathbf{S}(\mathbf{A}\boxplus\mathbf{B})=\mathbf{A}\boxplus\mathbf{S}%
(\mathbf{B})=\mathbf{A}$. By Lemmas \ref{olsaolt} and \ref{prodsubquo},
$\mathbf{T}(\mathbf{A}\boxplus\mathbf{B})=\mathbf{B}$.
\end{proof}

\begin{proposition} Let ${\bf A}$ be a nontrivial orthomodular lattice, ${\bf B}$ a nontrivial \PBZ --lattice and ${\bf L}={\bf A}\boxplus {\bf B}$. Then: $A=S({\bf L})$ iff $B=T({\bf L})$ iff ${\bf B}\in \AOL $, and, if so, then $\mathbf{L}\in\mathbb{OML}\boxplus\mathbb{AOL}$.\label{allsallt}\end{proposition}

\begin{proof} By Lemmas \ref{hsumolaol} and \ref{olsaolt}, $T({\bf L})=T({\bf B})$, hence: $B=T({\bf L})$ iff $B=T({\bf B})$ iff ${\bf B}\in \AOL $, which in turn implies $\mathbf{L}\in\mathbb{OML}\boxplus\mathbb{AOL}$.

By Lemmas \ref{ssubhsum} and \ref{olsaolt}, ${\bf S}({\bf L})={\bf A}\boxplus {\bf S}({\bf B})$, so $S({\bf L})=A\cup S({\bf B})$, hence: $A=S({\bf L})$ iff $A=A\cup S({\bf B})$ iff $S({\bf B})\subseteq A$ iff $S({\bf B})\subseteq A\cap B=\{0,1\}$ iff $S({\bf B})=\{0,1\}$ iff ${\bf B}\in \AOL $.\end{proof}

\begin{theorem}
\label{charg}For any nontrivial PBZ$^{\ast}$ --lattice $\mathbf{L}$, the following are equivalent:

\begin{enumerate}
\item\label{charg1} $\mathbf{L}\in\mathbb{OML}\boxplus\mathbb{AOL}$;

\item\label{charg0} $T(\mathbf{L})$ is a subuniverse of $\mathbf{L}$ and $\mathbf{L}=\mathbf{S}(\mathbf{L})\boxplus\mathbf{T}(\mathbf{L})$;

\item\label{charg6} $L=S(\mathbf{L})\cup T(\mathbf{L})$;

\item\label{charg2} $T(\mathbf{L})$ is a subuniverse of $\mathbf{L}$;

\item\label{charg3} $T(\mathbf{L})$ is closed w.r.t. the Kleene complement;

\item\label{charg4} $T(\mathbf{L})^{\prime}$ is closed w.r.t. the Kleene complement;

\item\label{charg5} $T(\mathbf{L})=T(\mathbf{L})^{\prime}$;

\item\label{charg7} $T(\mathbf{L})^{\prime}$ is closed w.r.t. the Brouwer complement;

\item\label{charg8} $T(\mathbf{L})\cup T(\mathbf{L})^{\prime}$ is closed w.r.t. the Brouwer complement.
\end{enumerate}
\end{theorem}

\begin{proof} (\ref{charg1}) $\Rightarrow$ (\ref{charg0}). If $\mathbf{L}\in\mathbb{OML}\boxplus\mathbb{AOL}$,
then, by Lemma \ref{hsumolaol}, $T(\mathbf{L})$ is a subuniverse of
$\mathbf{L}$ and $\mathbf{L}=\mathbf{S}(\mathbf{L})\boxplus\mathbf{T}%
(\mathbf{L})$.

Clearly, (\ref{charg0}) implies (\ref{charg6}) and (\ref{charg2}).

(\ref{charg6}) $\Rightarrow$ (\ref{charg1}). If $L=S(\mathbf{L})\cup T(\mathbf{L})$, then, if
$a\in\{0,1\}$, then $a^{\prime}\in\{0,1\}\subseteq T(\mathbf{L})$; if $a\in
T(\mathbf{L})\setminus\{0,1\}=T(\mathbf{L})\setminus S(\mathbf{L})=(S(\mathbf{L})\cup T(\mathbf{L}))\setminus S(\mathbf{L})=L\setminus
S(\mathbf{L})$, then it follows that $a^{\prime}\in L\setminus S(\mathbf{L}%
)=T(\mathbf{L})\setminus\{0,1\}\subset T(\mathbf{L})$; therefore
$T(\mathbf{L})$ is closed w.r.t. the Kleene complement, hence $T(\mathbf{L})$
is the universe of a subalgebra of $\mathbf{L}$ by Lemma \ref{tfilter}. So
$\mathbf{L}=\mathbf{S}(\mathbf{L})\boxplus\mathbf{T}(\mathbf{L})$, whence (\ref{charg1}) follows.

(\ref{charg2}) $\Leftrightarrow$ (\ref{charg3}) follows from Lemma \ref{tfilter}.

(\ref{charg3}) $\Leftrightarrow$ (\ref{charg4}) $\Leftrightarrow$ (\ref{charg5}) are clear.

(\ref{charg5}) $\Leftrightarrow$ (\ref{charg6}). If $T(\mathbf{L})=T(\mathbf{L})^{\prime}$, then,
for all $x\in L$, we have $x\geq x\wedge x^{\prime}\in T(\mathbf{L})^{\prime
}=T(\mathbf{L})=\{0\}\cup D(\mathbf{L})$, hence $x\wedge x^{\prime}=0$ or
$x\in T(\mathbf{L})\setminus\{0\}$, so $x\in S(\mathbf{L})$ or $x\in
T(\mathbf{L})$, thus $L=S(\mathbf{L})\cup T(\mathbf{L})$.

(\ref{charg5}) $\Leftrightarrow$ (\ref{charg7}). $T(\mathbf{L})^{\prime}$ is closed w.r.t. the
Brouwer complement iff $T(\mathbf{L})^{\prime\sim}\subseteq T(\mathbf{L}%
)^{\prime}$, which is equivalent to $T(\mathbf{L})^{\prime\sim}\subseteq
S(\mathbf{L})\cap T(\mathbf{L})^{\prime}=\{0,1\}$ since $T(\mathbf{L}%
)^{\prime\sim}\subseteq S(\mathbf{L})$. But, since $\{x\in L:x^{\sim}%
\in\{0,1\}\}=T(\mathbf{L})$, the inclusion $T(\mathbf{L})^{\prime\sim
}\subseteq\{0,1\}$ is equivalent to $T(\mathbf{L})^{\prime}\subseteq
T(\mathbf{L})$ and then to $T(\mathbf{L})^{\prime}=T(\mathbf{L})$.

(\ref{charg7}) $\Leftrightarrow$ (\ref{charg8}) is also obvious.
\end{proof}

\begin{proposition}
Let $\mathbf{L}$ be a nontrivial \PBZ --lattice. Then the following are equivalent:

\begin{enumerate}
\item \label{pbzirr1} all elements of $L\setminus\{0,1\}$ are
join--irreducible in $\mathbf{L}_{l}$;

\item \label{pbzirr2} all elements of $L\setminus\{0,1\}$ are
meet--irreducible in $\mathbf{L}_{l}$;

\item \label{pbzirr3} $\mathbf{L}={\bf MO}_{\kappa }\boxplus {\bf A}$ for a cardinal number $\kappa $ and an antiortholattice chain ${\bf A}$.\end{enumerate}\label{pbzirr}\end{proposition}

\begin{proof} Trivially, (\ref{aolirrch3}) implies (\ref{aolirrch1}), which is equivalent to
(\ref{aolirrch2}). To prove that (\ref{aolirrch1}) implies (\ref{aolirrch3}), assume that all elements of $L\setminus\{0,1\}$ are join--irreducible in $\mathbf{L}_{l}$.

Then, by Corollary \ref{sharpirred}, $S(\mathbf{L})\setminus\{0,1\}\subseteq {\rm At}(\mathbf{L}_l)\cap {\rm CoAt}(\mathbf{L}_l)\subseteq {\rm At}(\mathbf{L}_l)\cup {\rm CoAt}(\mathbf{L}_l)\subseteq S(\mathbf{L})\setminus\{0,1\}$, so that ${\rm At}(\mathbf{L}_l)={\rm CoAt}(\mathbf{L}_l)=S(\mathbf{L})\setminus\{0,1\}$, thus $\mathbf{S}(\mathbf{L})_l$ has length $3$ if $S(\mathbf{L})\setminus\{0,1\}\neq \emptyset $ (and, of course, $2$ otherwise), hence $\mathbf{S}(\mathbf{L})=\boxplus _{u\in S(\mathbf{L})\setminus\{0,1\}}\{0,u,u^{\prime },1\}\cong {\bf MO}_{\kappa }$ for a cardinal number $\kappa $ (such that $|S(\mathbf{L})\setminus\{0,1\}|=2\kappa $).

Now let $x,y\in L\setminus S(\mathbf{L})\subseteq L\setminus\{0,1\}$, so that $x^{\prime},y^{\prime}\in L\setminus S(\mathbf{L})$, as well. Assume by absurdum that $x$
and $y$ are incomparable, so that $x^{\prime}$ and $y^{\prime}$ are also incomparable, thus $x\vee y\neq 0\neq x^{\prime}\vee y^{\prime}$ are join--reducible, hence $x\vee y=1=x^{\prime}\vee y^{\prime}$. If $x$ and $x^{\prime}$ would be incomparable, then $0\neq x\vee x^{\prime}$ would be join--reducible, so that $x\vee x^{\prime}=1$, which would contradict the fact that $x\notin S(\mathbf{L})$. So $x$ and $x^{\prime}$ are comparable, thus $x\wedge x^{\prime}\in\{x,x^{\prime}\}$. Analogously, $y$ and $y^{\prime}$ are comparable, so $y\vee y^{\prime}\in\{y,y^{\prime}\}$. Since $\mathbf{L}$ is a pseudo--Kleene algebra, it follows that $x\wedge x^{\prime}\leq y\vee y^{\prime}$. But $x\nleq y$, hence either $x\wedge x^{\prime}\neq x$ or $y\vee y^{\prime}\neq y$, so that either
$x^{\prime}\geq x\leq y^{\prime}\geq y$ or $x\geq x^{\prime}\leq y\geq y^{\prime}$. In the first of these two situations, we obtain $1=x\vee y\leq x\vee y^{\prime}=y^{\prime}$, which contradicts the fact that $y^{\prime}\neq 1$, while, in the second situation, we obtain $1=x^{\prime}\vee y^{\prime
}\leq x^{\prime}\vee y=y$, which contradicts the fact that $y\neq 1$. Therefore $x$ and $y$ are comparable, hence $L\setminus S(\mathbf{L})$ is linearly ordered, so, by the structure of $\mathbf{S}(\mathbf{L})$, $x$ is comparable with at most one element of $S(\mathbf{L})\setminus \{0,1\}$ and, if $x$ and $x^{\sim }\in S(\mathbf{L})$ would be incomparable, then $x\vee x^{\sim }=1$, which would contradict the fact that $x\notin S(\mathbf{L})$. Hence $x$ and $x^{\sim }$ are comparable, thus $0=x\wedge x^{\sim }\in \{x,x^{\sim }\}$, thus $x^{\sim }=0$ since $x\neq 0$. Hence $L\setminus S(\mathbf{L})\subset T(\mathbf{L})$, thus $L=S(\mathbf{L})\cup T(\mathbf{L})$, so, by Theorem \ref{charg}, $T(\mathbf{L})$ is a subuniverse of $\mathbf{L}$, thus $\mathbf{T}(\mathbf{L})$ is an antiortholattice chain, and $\mathbf{L}=\mathbf{S}(\mathbf{L})\boxplus \mathbf{T}(\mathbf{L})={\bf MO}_{\kappa }\boxplus \mathbf{T}(\mathbf{L})$.\end{proof}

\begin{proposition}
Let $\mathbf{L}$ be a nontrivial antiortholattice. Then the following are equivalent:

\begin{enumerate}
\item \label{aolirrch1} all elements of $L\setminus\{0,1\}$ are
join--irreducible in $\mathbf{L}_{l}$;

\item \label{aolirrch2} all elements of $L\setminus\{0,1\}$ are
meet--irreducible in $\mathbf{L}_{l}$;

\item \label{aolirrch3} $\mathbf{L}_{l}$ is a chain.
\end{enumerate}\label{aolirrch}\end{proposition}

\begin{proof} By Proposition \ref{pbzirr} and the fact that $\mathbf{S}(\mathbf{L})\cong {\bf D}_2$.\end{proof}

\begin{corollary}
Let $\mathbf{L}$ be a nontrivial PBZ$^{\ast}$ --lattice that satisfies $J1$.
Then: all elements of $T(\mathbf{L})\setminus\{0,1\}$ are join--irreducible in
$\mathbf{L}_{l}$ iff $\mathbf{L}$ is a horizontal sum of an orthomodular
lattice with an antiortholattice chain.\label{j1irraolch}
\end{corollary}

\begin{proof}
The converse is trivial. For the direct implication, let $a\in L\setminus
S(\mathbf{L})$, so that $a\leq a\vee a^{\prime}\in T(\mathbf{L})\setminus
\{0,1\}$, hence $a\in T(\mathbf{L})$ by Proposition \ref{jirred}, so
$L=S(\mathbf{L})\cup T(\mathbf{L})$, therefore $\mathbf{L}\in\mathbb{OML}
\boxplus\mathbb{AOL} $ and $\mathbf{T}(\mathbf{L})$ is a subalgebra of
$\mathbf{L}$ and thus an antiortholattice by Theorem \ref{charg}. Furthermore, Proposition \ref{aolirrch} ensures us that $\mathbf{T}(\mathbf{L})$ is an
antiortholattice chain.
\end{proof}

\begin{corollary}
\label{irraolch}Let $\mathbf{L}$ be a nontrivial PBZ$^{\ast}$ --lattice.

\begin{enumerate}
\item \label{irraolch1} If all elements of $T(\mathbf{L})\setminus\{0,1\}$ are
join--irreducible in $\mathbf{L}_{l}$, then $L=S(\mathbf{L})\cup
T(\mathbf{L})\cup T(\mathbf{L})^{\prime}$ and $T(\mathbf{L})\setminus
T(\mathbf{L})^{\prime}\subseteq\{x\in L:x\geq x^{\prime}\}$. The converse is not true, even if $\mathbf{L}$ fulfills $J1$.

\item \label{irraolch2} All elements of $D(\mathbf{L})$ are join--irreducible
in $\mathbf{L}_{l}$ iff $\mathbf{L}$ is an antiortholattice chain.
\end{enumerate}
\end{corollary}

\begin{proof} (\ref{irraolch1}) For the direct implication, consider $a\in L\setminus
S(\mathbf{L})$, so that $a\vee a^{\prime}\in T(\mathbf{L})\setminus\{0,1\}$,
hence $a=a\vee a^{\prime}\in T(\mathbf{L})\setminus\{0,1\}$ or $a^{\prime
}=a\vee a^{\prime}\in T(\mathbf{L})\setminus\{0,1\}$, so that $a\in
T(\mathbf{L})^{\prime}\setminus\{0,1\}$ and, if $a\notin T(\mathbf{L}%
)^{\prime}$, then $a=a\vee a^{\prime}$, that is $a\geq a^{\prime}$. To disprove the converse, see below the PBZ$^{\ast}$ --lattice $\mathbf{K}$ in Example \ref{exfail12}, which fulfills $K=S(\mathbf{K})\cup T(\mathbf{K})\cup
T(\mathbf{K})^{\prime}$, but in which the element $t^{\prime}\in
T(\mathbf{K})\setminus\{0,1\}$ is join--reducible. Moreover, $T(\mathbf{K})\setminus T(\mathbf{K})^{\prime}=\{t^{\prime}\}$ and $t^{\prime}\geq t$. Also, $\mathbf{K}\in\mathbb{OML}\boxplus V(\mathbb{AOL})$, so $\mathbf{K}\vDash J1$ by Corollary \ref{axvarhsum}.

\noindent (\ref{irraolch2}) If all
elements of $D(\mathbf{L})$ are join--irreducible in $\mathbf{L}_{l}$, then
$1$ is join--irreducible in $\mathbf{L}_{l}$, hence $\mathbf{L}$ is an
antiortholattice, thus $\mathbf{L}$ is an antiortholattice chain by Proposition \ref{aolirrch}. The converse is trivial.
\end{proof}

\begin{proposition}
\label{vaolgent} $V(\mathbb{AOL})\subsetneq\{\mathbf{L}\in\mathbb{PBZL}^{\ast}:\mathbf{L}=\langle T(\mathbf{L})\rangle_{\mathbb{BZL}}\}$.
\end{proposition}

\begin{proof} By Lemma \ref{olsaolt}, $\mathbf{L}=\mathbf{T}(\mathbf{L})=\langle T(\mathbf{L})\rangle_{\mathbb{BZL}}$ for any $\mathbf{L}\in\mathbb{AOL}$.

By Lemma \ref{prodsubquo}, it follows that, for any non--empty family
$(\mathbf{L}_{i})_{i\in I}\subseteq\mathbb{AOL}$, $\prod_{i\in I}%
\mathbf{L}_{i}=\prod_{i\in I}\langle T(\mathbf{L}_{i})\rangle_{\mathbb{BZL}%
}=\langle T(\prod_{i\in I}\mathbf{L}_{i})\rangle_{\mathbb{BZL}}$, hence
$\mathbf{L}=\langle T(\mathbf{L})\rangle_{\mathbb{BZL}}$ for any
$\mathbf{L}\in \P (\mathbb{AOL})$.

Again by Lemma \ref{prodsubquo}, we obtain that, for any $\mathbf{A}\in
\P (\mathbb{AOL})$ and any subalgebra $\mathbf{B}$ of $\mathbf{A}$, $\langle
T(\mathbf{B})\rangle_{\mathbb{BZL}}=\langle T(\mathbf{A})\rangle
_{\mathbb{BZL}}\cap\mathbf{B}=\mathbf{A}\cap\mathbf{B}=\mathbf{B}$, hence
$\mathbf{L}=\langle T(\mathbf{L})\rangle_{\mathbb{BZL}}$ for any
$\mathbf{L}\in \S \P (\mathbb{AOL})$.

We apply Lemma \ref{prodsubquo} once again and obtain that, for any
$\mathbf{A}\in \S \P (\mathbb{AOL})$ and any $\theta\in\mathrm{Con}_{\mathbb{BZL}%
}(\mathbf{A})$, $\mathbf{A}/\theta=\langle T(\mathbf{A})\rangle_{\mathbb{BZL}%
}/\theta=\langle T(\mathbf{A})/\theta\rangle_{\mathbb{BZL}}\subseteq\langle
T(\mathbf{A}/\theta)\rangle_{\mathbb{BZL}}\subseteq\mathbf{A}/\theta$, hence
$\mathbf{A}/\theta=\langle T(\mathbf{A}/\theta)\rangle_{\mathbb{BZL}}$,
therefore $\mathbf{L}=\langle T(\mathbf{L})\rangle_{\mathbb{BZL}}$ for any
$\mathbf{L}\in$\linebreak$\H \S \P (\mathbb{AOL})=V(\mathbb{AOL})$.

Hence $V(\mathbb{AOL})\subseteq\{\mathbf{L}\in\mathbb{PBZL}^{\ast}:\mathbf{L}=\langle T(\mathbf{L})\rangle_{\mathbb{BZL}}\}$. The \PBZ --lattice ${\bf M}$ in Example \ref{exfail12} below disproves the converse inclusion.\end{proof}

\begin{corollary}
\label{aolvaol}Let $\mathbf{L}\in V(\mathbb{AOL})$. Then the following are equivalent:

\begin{enumerate}
\item \label{aolvaol1} $T(\mathbf{L})$ is a subuniverse of $\mathbf{L}$;

\item \label{aolvaol2} $T(\mathbf{L})$ is closed w.r.t. the Kleene complement;

\item \label{aolvaol3} $\mathbf{L}$ is an antiortholattice.
\end{enumerate}
\end{corollary}

\begin{proof}
(i) $\Leftrightarrow$ (ii). By Lemma \ref{tfilter}.

(i) $\Leftrightarrow$ (iii). For any BZ--lattice $\mathbf{L}$, $T(\mathbf{L})$
is the universe of a subalgebra of $\mathbf{L}$ iff $\langle T(\mathbf{L}%
)\rangle_{\mathbb{BZL}}=\mathbf{T}(\mathbf{L})$, hence, by Proposition
\ref{vaolgent} and Lemma \ref{olsaolt}, if $\mathbf{L}\in V(\mathbb{AOL})$,
then: $T(\mathbf{L})$ is a subuniverse of $\mathbf{L}$ iff $\mathbf{L}%
=\mathbf{T}(\mathbf{L})$ iff $\mathbf{L}\in\mathbb{AOL}$.
\end{proof}

\begin{corollary}
If $\mathbf{A}\in\mathbb{OML}\setminus \{{\bf D}_1\}$, $\mathbf{B}\in V(\mathbb{AOL})\setminus \{{\bf D}_1\}$ and $\mathbf{L}=\mathbf{A}\boxplus\mathbf{B}$, then:

\begin{itemize}
\item $A=(L\setminus\langle T(\mathbf{L})\rangle_{\mathbb{BZL} }%
)\cup\{0,1\}\subseteq S(\mathbf{L})$ and $\mathbf{B}=\langle T(\mathbf{L}%
)\rangle_{\mathbb{BZL} }$;

\item $L=S(\mathbf{L})\cup B=S(\mathbf{L})\cup\langle T(\mathbf{L}%
)\rangle_{\mathbb{BZL}}$;

\item $\mathbf{L}\in\mathbb{OML} \boxplus\mathbb{AOL} $ iff $A=S(\mathbf{L})$ iff $B=T(\mathbf{L})$ iff ${\bf B}\in \AOL $.
\end{itemize}

$\mathbb{OML} \boxplus V(\mathbb{AOL} )\subsetneq\{\mathbf{L}\in
\mathbb{PBZL}^{\ast} :L=S(\mathbf{L})\cup\langle T(\mathbf{L})\rangle
_{\mathbb{BZL} }\}\nsubseteq V(\mathbb{OML} \boxplus V(\mathbb{AOL} ))$.
Moreover, $\{\mathbf{L}\in\mathbb{PBZL}^{\ast} :\mathbf{L}=\langle
T(\mathbf{L})\rangle_{\mathbb{BZL} }\}\nsubseteq V(\mathbb{OML} \boxplus
V(\mathbb{AOL} ))$.\label{oml+vaol}
\end{corollary}

\begin{proof}
By Lemma \ref{ssubhsum}, $S(\mathbf{L})=A\cup S(\mathbf{B})\supseteq A$. By
Lemma \ref{prodsubquo} and Proposition \ref{vaolgent}, $T(\mathbf{L}%
)=T(\mathbf{B})$, thus $\mathbf{B}=\langle T(\mathbf{B})\rangle_{\mathbb{BZL}
}=\langle T(\mathbf{L})\rangle_{\mathbb{BZL} }$, hence $\mathbf{L}%
=\mathbf{A}\boxplus\langle T(\mathbf{L})\rangle_{\mathbb{BZL} }$ and
$(L\setminus\langle T(\mathbf{L})\rangle_{\mathbb{BZL} })\cup
\{0,1\}=(L\setminus B)\cup\{0,1\}=A\subseteq S(\mathbf{L})$. Also, $L=A\cup
B\subseteq S(\mathbf{L})\cup B\subseteq A\cup B$, hence $L=S(\mathbf{L})\cup
B=S(\mathbf{L})\cup\langle T(\mathbf{L})\rangle_{\mathbb{BZL} }$.

By Proposition \ref{allsallt}, $A=S(\mathbf{L})$ iff $B=T(\mathbf{L})$ iff ${\bf B}\in \AOL $, which implies $\mathbf{L}\in\mathbb{OML}\boxplus\mathbb{AOL}$.

Now assume that $\mathbf{L}\in\mathbb{OML}\boxplus\mathbb{AOL}$, so that, by Theorem \ref{charg}, $T(\mathbf{L})=T(\mathbf{B})\subseteq B$ is a subuniverse of $\mathbf{L}$ and thus a subuniverse of $\mathbf{B}$ since $\mathbf{B}$ is a subalgebra of $\mathbf{L}$, hence $\mathbf{B}\in\mathbb{AOL}$ by Corollary \ref{aolvaol}.

The \PBZ --lattice ${\bf M}$ in Example \ref{exfail12} below shows the non--inclusions and, along with Proposition \ref{vaolgent}, also the strict inclusion.\end{proof}

\section{Direct Irreducibility in Certain Varieties of \PBZ --lattices}

Recall from \cite{PBZ2} that antiortholattices are directly irreducible, and from \cite{rgcmfp} that, moreover, the class of the directly irreducible members of $V(\AOL )$ is $\AOL $. Now let us see that even the lattice reducts of antiortholattices are directly irreducible. In relation to this property, let us investigate pseudo--Kleene algebras with directly reducible lattice reducts, as well as bounded lattice complements in lattice reducts of antiortholattices.

\begin{proposition}\label{prodbi} Let $\mathbf{A}$ and $\mathbf{B}$ be bounded lattices. Then:\begin{enumerate}
\item\label{prodbi1} if $\mathbf{A},\mathbf{B}\in \BI $ and they are non--trivial, then the direct product of BI--lattices $\mathbf{A}\times \mathbf{B}$, endowed with the trivial
Brouwer complement, fails condition $(\ast)$;
\item\label{prodbi4} if $\mathbf{L}\in \KL $ is such that $\mathbf{L}_l=\mathbf{A}\times \mathbf{B}$, then $(0^{\bf A},1^{\bf B})^{\prime \mathbf{L}}=(1^{\bf A},0^{\bf B})$.\end{enumerate}\end{proposition}

\begin{proof} In the following, for brevity, we will drop the superscripts.

\noindent (\ref{prodbi1}) If $\cdot ^{\sim }:A\times B\rightarrow A\times B$ is the trivial Brouwer complement, then, in $\mathbf{A}\times \mathbf{B}$, we have: $(0,1)^{\prime }=(0^{\prime },1^{\prime })=(1,0)$ and $(0,1)\neq (1,1)\neq (1,0)$, hence: $(1,1)=(0,0)^{\sim }=((0,1)\wedge (1,0))^{\sim }=((0,1)\wedge (0,1)^{\prime })^{\sim }$, but $(0,1)^{\sim }\vee (0,1)^{\prime \sim }=(0,1)^{\sim }\vee (1,0)^{\sim }=(0,0)\vee (0,0)=(0,0)\neq (1,1)$.

\noindent (\ref{prodbi4}) Let $(0,1)^{\prime }=(a,b)\in L=A\times B$ and $(1,0)^{\prime }=(c,d)\in L=A\times B$. Since $\mathbf{L}\in \KL $, we have $(0,b)=(0,1)\wedge (a,b)\leq (1,0)\vee (c,d)=(1,d)$ and $(a,1)=(0,1)\vee (a,b)\geq (1,0)\wedge (c,d)=(c,0)$, so that $b\leq d$ in $\mathbf{B}$ and $a\geq c$ in $\mathbf{A}$. Hence $(a,d)=(a,b)\vee (c,d)=(0,1)^{\prime }\vee (1,0)^{\prime }=((0,1)\wedge (1,0))^{\prime }=(0,0)^{\prime }=(1,1)$ and $(c,b)=(a,b)\wedge (c,d)=(0,1)^{\prime }\wedge (1,0)^{\prime }=((0,1)\vee (1,0))^{\prime }=(1,1)^{\prime }=(0,0)$, thus $c=0$ and $a=1$ in $\mathbf{A}$, while $b=0$ and $d=1$ in $\mathbf{B}$. Therefore $(0,1)^{\prime }=(a,b)=(1,0)$.\end{proof}

Note that a BI--lattice $\mathbf{L}$ can be directly irreducible while $\mathbf{L}_l$ is directly reducible; indeed, the BI--lattice ${\bf D}_3\boxplus {\bf D}_3$, in which the incomparable elements equal their involutions, is directly irreducible, but its lattice reduct is isomorphic to ${\bf D}_2^2$.

Proposition \ref{prodbi}.(\ref{prodbi1}) shows that the BI--lattice reduct of any antiortholattice is directly irreducible. Moreover, we have:

\begin{proposition} The lattice reduct of any antiortholattice is directly irreducible.\label{aoldirirred}
\end{proposition}

\begin{proof} Let $\mathbf{L}\in \AOL $ and assume ex absurdo that $\mathbf{L}_{l}=\mathbf{A}\times\mathbf{B}$ for some non--trivial bounded lattices $\mathbf{A}$ and $\mathbf{B}$. Then $(0,1)^{\prime }=(1,0)$ by Proposition \ref{prodbi}.(\ref{prodbi4}), hence $(0,1)\in S(\mathbf{L})$, which contradicts the fact that $\mathbf{L}$ is an antiortholattice, since $(0,1)\notin \{(0,0),(1,1)\}$.\end{proof}

\begin{proposition} The only complemented elements of the lattice reduct of a distributive antiortholattice are $0$ and $1$.\end{proposition}

\begin{proof} Let $\mathbf{L}$ be a distributive antiortholattice and assume by absurdum that, for some $a,b\in L\setminus \{0,1\}$, $a\vee b=1$ and $a\wedge b=0$, so that $a^{\prime }\wedge b^{\prime }=(a\vee b)^{\prime }=1^{\prime }=0$. Since $\mathbf{L}_{bi}\in \KL $, we have $b\wedge b^{\prime }\leq a\vee a^{\prime }$, hence $b\wedge b^{\prime }=(a\vee a^{\prime })\wedge b\wedge b^{\prime }=(a\wedge b\wedge b^{\prime })\vee (a^{\prime }\wedge b\wedge b^{\prime })=0\vee 0=0$, thus $b\in S(\mathbf{L})$, which contradicts the fact that $\mathbf{L}$ is an antiortholattice.\end{proof}

\begin{example} Here is a non--modular antiortholattice with other complemented elements beside $0$ and $1$, namely, in the following Hasse diagram, $a$ and $a^{\prime }$ are bounded lattice complements of both $b$ and $b^{\prime }$:
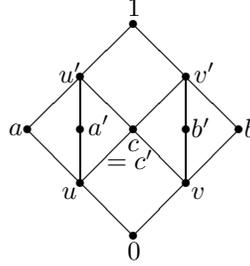
\begin{center}\begin{picture}(40,83)(0,0)
\put(20,0){\circle*{3}}
\put(20,80){\circle*{3}}
\put(20,40){\circle*{3}}
\put(0,20){\circle*{3}}
\put(-20,40){\circle*{3}}
\put(0,60){\circle*{3}}
\put(60,40){\circle*{3}}
\put(0,40){\circle*{3}}
\put(40,40){\circle*{3}}
\put(40,20){\circle*{3}}
\put(40,60){\circle*{3}}
\put(18,-9){$0$}
\put(18,83){$1$}
\put(18,32){$c$}
\put(10,25){$=c^{\prime }$}
\put(-27,38){$a$}
\put(-8,58){$u^{\prime }$}
\put(62,37){$b$}
\put(43,58){$v^{\prime}$}
\put(42,14){$v$}
\put(-7,14){$u$}
\put(3,38){$a^{\prime}$}
\put(42,37){$b^{\prime}$}
\put(20,0){\line(-1,1){40}}
\put(20,0){\line(1,1){40}}
\put(20,80){\line(-1,-1){40}}
\put(20,80){\line(1,-1){40}}
\put(0,20){\line(0,1){40}}
\put(40,20){\line(0,1){40}}
\put(0,20){\line(1,1){40}}
\put(40,20){\line(-1,1){40}}
\end{picture}\end{center}\end{example}

\begin{lemma} If $\mathbf{L}$, $\mathbf{A}$ and $\mathbf{B}$ are bounded lattices such that $\mathbf{L}=\mathbf{A}\boxplus\mathbf{B}$, $|A|>2$, $|B|>2$ and $|L|\geq 5$, then $\mathbf{L}$ is directly irreducible.\label{latdirirred}\end{lemma}

\begin{proof}
Let $\mathbf{L}=(L,\wedge ,\vee ,0,1)$, and assume ex absurdo that $\mathbf{L}=\mathbf{K}\times\mathbf{M}$ for some nontrivial bounded lattices $\mathbf{K}$ and $\mathbf{M}$. Since $|L|>4$, we may assume, w.l.g., that there exists a $u\in K\setminus\{0^{\mathbf{K}},1^{\mathbf{K}}\}$, so that
$(u,1^{\mathbf{M}})\notin\{0=(0^{\mathbf{K}},0^{\mathbf{M}}),(0^{\mathbf{K}},1^{\mathbf{M}}),1=(1^{\mathbf{K}},1^{\mathbf{M}})\}$ and $(u,0^{\mathbf{M}})\notin\{0=(0^{\mathbf{K}},0^{\mathbf{M}}),(1^{\mathbf{K}},0^{\mathbf{M}}),1=(1^{\mathbf{K}},1^{\mathbf{M}})\}$.

Since $\mathbf{L}=\mathbf{A}\boxplus\mathbf{B}$, we have, for every $a\in
A\setminus\{0,1\}$ and every $b\in B\setminus\{0,1\}$: $a\vee b=1$ and
$a\wedge b=0$. We can assume that $(u,1^{\mathbf{M}})\in A\setminus\{0,1\}$.
Since $(u,1^{\mathbf{M}})\vee(u,0^{\mathbf{M}})=(u,1^{\mathbf{M}}%
)\neq(1^{\mathbf{K}},1^{\mathbf{M}})=1$, it follows that $(u,0^{\mathbf{M}%
})\notin B\setminus\{0,1\}$, hence $(u,0^{\mathbf{M}})\in A\setminus\{0,1\}$.
Now let $(v,w)\in B\setminus\{0,1\}$. Then $(u\wedge v,w)=(u,1^{\mathbf{M}%
})\wedge(v,w)=0=(0^{\mathbf{K}},0^{\mathbf{M}})$ and $(u\vee
v,w)=(u,0^{\mathbf{M}})\vee(v,w)=1=(1^{\mathbf{K}},1^{\mathbf{M}})$, thus
$0^{\mathbf{M}}=w=1^{\mathbf{M}}$, which contradicts the fact that
$\mathbf{M}$ is nontrivial. Hence $\mathbf{L}$ is directly irreducible.
\end{proof}

Let $\mathbf{L}_{1}\in\mathbb{OML}\setminus\left\{  \mathbf{D}_{1}\right\}  $
and $\mathbf{L}_{2}\in\mathbb{AOL}\setminus\left\{  \mathbf{D}_{1}\right\}
$,\ be such that $\mathbf{L}=\mathbf{L}_{1}\boxplus\mathbf{L}_{2}%
\notin\mathbb{OML}$, whence $\mathbf{L}_{2}\neq\mathbf{D}_{2}$. If
$\mathbf{L}_{1}=\mathbf{D}_{2}$, then, by Proposition \ref{aoldirirred}, $\mathbf{L}_{l}$ is directly irreducible. If $\mathbf{L}_{1}\neq\mathbf{D}_{2}$, then by Lemma \ref{latdirirred} $\mathbf{L}_{l}$ is likewise
directly irreducible. So, we obtain that:

\begin{proposition}
\begin{enumerate}
\item \label{gdirirred1} If $\mathbf{L}\in(\mathbb{OML}\boxplus\mathbb{AOL})\setminus\mathbb{OML}$, then $\mathbf{L}_{l}$ is directly irreducible, thus $\mathbf{L}$ is directly irreducible.

\item \label{gdirirred0} If $\mathbf{L}\in(\mathbb{OML}\boxplus V(\mathbb{AOL}))\setminus(\mathbb{OML}\cup V(\mathbb{AOL}))$, then $\mathbf{L}_{l}$ is directly irreducible, thus $\mathbf{L}$ is directly irreducible.\end{enumerate}\label{gdirirred}\end{proposition}

Also, recall that any ortholattice $\mathbf{L}$ with more than $2$ elements is such that $0$ is meet--reducible and $1$ is join--reducible in $\mathbf{L}_l$. Thus:

\begin{corollary} If $\mathbf{A}$ is a finite ortholattice and $\mathbf{B}$ is a finite pseudo--Kleene algebra with $|A|>2$ and $|B|>2$, then $\mathbf{A}_l\boxplus\mathbf{B}_l$ has at least three distinct atoms (thus at least three distinct co--atoms).\end{corollary}

\begin{proof} Since $\mathbf{A}$ is not an antiortholattice, $\mathbf{A}_l$ has at least two distinct atoms. Since $\mathbf{B}$ is finite and has $|B|>2$, it follows that $\mathbf{B}_l$ has at least one atom, which is not equal to $1$. Our conclusion follows.\end{proof}

\section{Singleton--Generated Subalgebras of \PBZ --la\-t\-ti\-ces}

\begin{lemma} Let $\V $ be a variety and $\C ,\D $ subclasses of $\V $ such that, for all ${\bf M}\in \C $ and all $x\in M$, we have $\langle x\rangle _{\V ,{\bf M}}\in \D $. Let ${\bf A}\in \V $ and $a\in A$. Then:\begin{enumerate}
\item\label{subsglt1} if ${\bf A}\in \P _{\V }(\C )$, then $\langle a\rangle _{\V ,{\bf A}}\in \P _{\V }(\D )$;
\item\label{subsglt2} if ${\bf A}\in \S _{\V }(\C )$, then $\langle a\rangle _{\V ,{\bf A}}\in \S _{\V }(\D )$;
\item\label{subsglt3} if ${\bf A}\in \H _{\V }(\C )$, then $\langle a\rangle _{\V ,{\bf A}}\in \H _{\V }(\D )$;
\item\label{subsglt4} if ${\bf A}\in V_{\V }(\C )$, then $\langle a\rangle _{\V ,{\bf A}}\in V_{\V }(\D )$.\end{enumerate}\label{subsglt}\end{lemma}

\begin{proof} (\ref{subsglt1}) For some non--empty family $({\bf A}_i)_{i\in I}\subseteq \C $, we have ${\bf A}=\prod _{i\in I}{\bf A}_i$, so that $a=(a_i)_{i\in I}$, with $a_i\in A_i$ for all $i\in I$. Then $\langle a_i\rangle _{\V ,{\bf A}_i}\in \D $ for all $i\in I$, hence $\langle a\rangle _{\V ,{\bf A}}=\langle (a_i)_{i\in I}\rangle _{\V ,\prod _{i\in I}{\bf A}_i}=\prod _{i\in I}\langle a_i\rangle _{\V ,{\bf A}_i}\in \P _{\V }(\D )$.

\noindent (\ref{subsglt2}) For some ${\bf B}\in \C $ with $A\subseteq B$, we have ${\bf A}\in \S _{\V }({\bf B})$, therefore $\langle a\rangle _{\V ,{\bf A}}=\langle a\rangle _{\V ,{\bf B}}\cap {\bf A}\in \S _{\V }(\langle a\rangle _{\bf B})$, thus $\langle a\rangle _{\V ,{\bf A}}\in \S _{\V }(\D )$.

\noindent (\ref{subsglt3}) For some ${\bf B}\in \C $ and some $\theta \in {\rm Con}_{\V }({\bf B})$, we have ${\bf A}={\bf B}/\theta $, so $a=b/\theta $ for some $b\in B$. Since $b/\theta \in \langle b\rangle _{\V ,{\bf B}}/\theta \in \S _{\V }({\bf B}/\theta )$, it follows that $\langle b/\theta \rangle _{\V ,{\bf B}/\theta }\subseteq \langle b\rangle _{\V ,{\bf B}}/\theta $. If $u\in \langle b\rangle _{\V ,{\bf B}}$, then $u=t^{\bf B}(b)$ for some term $t$ over the type of $\V $, thus $u/\theta =t^{{\bf B}/\theta }(b/\theta )\in \langle b/\theta \rangle _{\V ,{\bf B}/\theta }$, therefore $\langle b\rangle _{\V ,{\bf B}}/\theta \subseteq \langle b/\theta \rangle _{\V ,{\bf B}/\theta }$. Hence $\langle a\rangle _{\V ,{\bf A}}=\langle b/\theta \rangle _{\V ,{\bf B}/\theta }=\langle b\rangle _{\V ,{\bf B}}/\theta \in \H _{\PBZs }(\langle b\rangle _{\V ,{\bf B}})$, therefore $\langle a\rangle _{\V ,{\bf A}}\in \H _{\V }(\D )$.

\noindent (\ref{subsglt4}) By (\ref{subsglt1}), (\ref{subsglt2}) and (\ref{subsglt3}), if ${\bf A}\in V_{\V }(\C )=\H _{\V }\S _{\V }\P _{\V }(\C )$, then $\langle a\rangle _{\V ,{\bf A}}\in \H _{\V }\S _{\V }\P _{\V }(\D )=V_{\V }(\D )$.\end{proof}

Note that, in any orthomodular lattice ${\bf L}$, for all $x\in L$, $\langle x\rangle _{\BZ ,{\bf L}}=\{0,x,x^{\prime },1\}$\linebreak $\in \{{\bf D}_1,{\bf D}_2,{\bf D}_2^2\}\subset \BA =V({\bf D}_2)\subset V({\bf D}_3)=V({\bf D}_4)$, where the last equality follows from the easy to notice facts that ${\bf D}_3\in \H _{\BZ }({\bf D}_4)$ and ${\bf D}_4\in \S _{\BZ }({\bf D}_2\times {\bf D}_3)$. If ${\bf M}$ is an antiortholattice and $x\in M$, then, clearly, $\langle x\rangle _{\BZ ,{\bf M}}=\{0,x\wedge x^{\prime },x,x^{\prime },x\vee x^{\prime },1\}\in \{{\bf D}_1,{\bf D}_2,{\bf D}_4,{\bf D}_2\oplus {\bf D}_2^2\oplus {\bf D}_2\}\subset \AOL \cap V({\bf D}_3)$ since ${\bf D}_2\oplus {\bf D}_2^2\oplus {\bf D}_2\in \S _{\BZ }({\bf D}_4\times {\bf D}_4)$, more precisely: if ${\bf M}\cong {\bf D}_1$, then $\langle x\rangle _{\BZ ,{\bf M}}={\bf M}\cong {\bf D}_1$, while, if ${\bf M}$ is non--trivial:\begin{itemize}
\item if $x\in \{0,1\}$, then $\langle x\rangle _{\BZ ,{\bf M}}=\{0,1\}\cong {\bf D}_2$;
\item if $x\notin \{0,1\}$, but $x$ and $x^{\prime }$ are comparable, then $\langle x\rangle _{\BZ ,{\bf M}}=\{0,x,x^{\prime },1\}\cong {\bf D}_4$;
\item if $x\notin \{0,1\}$ and $x||x^{\prime }$, then $\langle x\rangle _{\BZ ,{\bf M}}=\{0,x\wedge x^{\prime },x,x^{\prime },x\vee x^{\prime },1\}\cong {\bf D}_2\oplus {\bf D}_2^2\oplus {\bf D}_2$.\end{itemize}

Clearly, for any non--trivial orthomodular lattice ${\bf L}$, any non--trivial \PBZ --lattice ${\bf M}$ and any $x\in L\boxplus M$, we have:\begin{itemize}
\item if $x\in \{0,1\}$, then $\langle x\rangle _{\BZ ,{\bf L}\boxplus {\bf M}}=\{0,1\}\cong {\bf D}_2$;
\item if $x\in L\setminus \{0,1\}$, then $\langle x\rangle _{\BZ ,{\bf L}\boxplus {\bf M}}=\langle x\rangle _{\BZ ,{\bf L}}=\{0,x,x^{\prime },1\}\cong {\bf D}_2^2$;
\item if $x\in M\setminus \{0,1\}$, then $\langle x\rangle _{\BZ ,{\bf L}\boxplus {\bf M}}=\langle x\rangle _{\BZ ,{\bf M}}$.\end{itemize}

From the above, we obtain:

\begin{proposition} Let ${\bf A}\in \PBZs $ and $a\in A$. Then:\begin{itemize}
\item if ${\bf A}\in \OML $, then $\langle a\rangle _{\BZ ,{\bf A}}\in \{{\bf D}_1,{\bf D}_2,{\bf D}_2^2\}\subset \BA $;
\item if ${\bf A}\in \AOL $, then $\langle a\rangle _{\BZ ,{\bf A}}\in \{{\bf D}_1,{\bf D}_2,{\bf D}_4,{\bf D}_2\oplus {\bf D}_2^2\oplus {\bf D}_2\}\subset \AOL \cap V({\bf D}_3)$;
\item if ${\bf A}\in \OML \boxplus \AOL $, then $\langle a\rangle _{\BZ ,{\bf A}}\in \{{\bf D}_1,{\bf D}_2,{\bf D}_2^2,{\bf D}_4,{\bf D}_2\oplus {\bf D}_2^2\oplus {\bf D}_2\}\subset V({\bf D}_3)$.\end{itemize}\label{sghsum}\end{proposition}

\begin{lemma} Let $\C ,\D $ be subclasses of $\PBZs $ such that $\C $ contains non--trivial algebras and, for all ${\bf M}\in \C $ and all $a\in M$, we have $\langle a\rangle _{\BZ ,{\bf M}}\in \D $. Then, for all ${\bf A}\in V(\OML \boxplus V(\C ))$ and all $a\in A$, we have $\langle a\rangle _{\BZ ,{\bf A}}\in V(\D )$.\label{sgomlhsum}\end{lemma}

\begin{proof} In any non--trivial \PBZ --lattice ${\bf M}$, $\langle 0\rangle _{\BZ ,{\bf M}}=\{0,1\}\cong {\bf D}_2$, hence ${\bf D}_2\in \D $, therefore $\BA \subseteq V(\D )$. Now let ${\bf A}\in \PBZs $ and $a\in A$.

By Lemma \ref{subsglt}.(\ref{subsglt4}), if ${\bf A}\in V(\C )$, then $\langle a\rangle _{\BZ ,{\bf A}}\in V(\D )\supseteq \BA $, therefore, by the above, if ${\bf A}\in \OML \boxplus V(\C )$, then $\langle a\rangle _{\BZ ,{\bf A}}\in \BA \cup V(\D )=V(\D )$. Again by Lemma \ref{subsglt}.(\ref{subsglt4}), it follows that, if ${\bf A}\in V(\OML \boxplus V(\C ))$, then $\langle a\rangle _{\BZ ,{\bf A}}\in V(\D )$.\end{proof}

\begin{theorem} For any ${\bf A}\in V(\OML \boxplus V(\AOL ))$ and all $a\in A$, we have $\langle a\rangle _{\BZ ,{\bf A}}\in V({\bf D}_3)$.\end{theorem}

\begin{proof} By Proposition \ref{sghsum} and Lemma \ref{sgomlhsum}.\end{proof}

\section{Congruences of Horizontal Sums}

In order to better understand the properties of horizontal sums of \PBZ --lattices, it is crucial to investigate the structure of their congruence lattices --- in particular, to find convenient descriptions of simple and subdirectly irreducible horizontal sums. We now set about accomplishing this task.

Let ${\mathbb{V}}$ be the variety of bounded lattices or one of the varieties $\mathbb{BI},\mathbb{BZL}$, and let $\mathbf{L}$ and $\mathbf{M}$ be nontrivial members of ${\mathbb{V}}$. Since $\mathbf{L}$ and $\mathbf{M}$ are subalgebras of $\mathbf{L}\boxplus\mathbf{M}$, for any $\theta
\in\mathrm{Con}_{{\mathbb{V}}}(\mathbf{L}\boxplus\mathbf{M})$, we have $\theta\cap L^{2}\in\mathrm{Con}_{{\mathbb{V}}}(\mathbf{L})$ and $\theta\cap
M^{2}\in\mathrm{Con}_{{\mathbb{V}}}(\mathbf{M})$; additionally, if $\theta
\neq\nabla_{L\boxplus M}$, then $\theta\cap L^{2}\neq\nabla_{L}$, $\theta\cap
M^{2}\neq\nabla_{M}$ and $\theta=(\theta\cap L^{2})\boxplus(\theta\cap M^{2}%
)$.

\begin{lemma}{\rm \cite{eunoucard}} For any bounded lattices $\mathbf{L}$ and $\mathbf{M}$ with $|L|>2$ and
$|M|>2$,\begin{align*}
\mathrm{Con}_{01}(\mathbf{L}\boxplus\mathbf{M})  &  =\{\alpha\boxplus
\beta:\alpha\in\mathrm{Con}_{01}(\mathbf{L}),\beta\in\mathrm{Con}%
_{01}(\mathbf{M})\}\\
&  \cong\mathrm{Con}_{01}(\mathbf{L})\times\mathrm{Con}_{01}(\mathbf{M});
\end{align*}
and%
\begin{align*}
&  (\mathrm{Con}_{01}(\mathbf{L})\times\mathrm{Con}_{01}(\mathbf{M}%
))\oplus\mathbf{D}_{2}\\
&  \cong\mathrm{Con}_{01}(\mathbf{L}\boxplus\mathbf{M})\cup\{\nabla_{L\boxplus
M}\}\subseteq\mathrm{Con}(\mathbf{L}\boxplus\mathbf{M})\\
&  \subseteq\mathrm{Con}_{01}(\mathbf{L}\boxplus\mathbf{M})\cup\{eq(L\setminus
\{0\},M\setminus\{1\}),eq(L\setminus\{1\},M\setminus\{0\}),\nabla_{L\boxplus
M}\}.
\end{align*}

\label{cghsumlat}
\end{lemma}

\begin{lemma}
Let ${\mathbb{V}}$ be one of the varieties $\mathbb{BI},\mathbb{BZL}$, and let
$\mathbf{A}$ and $\mathbf{B}$ be nontrivial members of ${\mathbb{V}}$. Then,
for any $\alpha\in\mathrm{Con}_{{\mathbb{V}}}(\mathbf{A})\setminus\{\nabla
_{A}\}$ and any $\beta\in\mathrm{Con}_{{\mathbb{V}}}(\mathbf{B})\setminus
\{\nabla_{B}\}$, we have: $\alpha\boxplus\beta\in\mathrm{Con}_{{\mathbb{V}}%
}(\mathbf{A}\boxplus\mathbf{B})$ iff $\alpha\boxplus\beta\in\mathrm{Con}%
(\mathbf{A}\boxplus\mathbf{B})$.\label{hsumfullcg}
\end{lemma}

\begin{proof}
If $\alpha\in\mathrm{Con}_{{\mathbb{V}}}(\mathbf{A})\setminus\{\nabla_{A}\}$
and $\beta\in\mathrm{Con}_{{\mathbb{V}}}(\mathbf{B})\setminus\{\nabla_{B}\}$,
then $\alpha$ preserves the involution of $\mathbf{A}$, $\beta$ preserves the
involution of $\mathbf{B}$, $0/\alpha\neq0/\beta$ and $1/\alpha\neq1/\beta$,
thus, clearly, $\alpha\boxplus\beta$ preserves the involution of
$\mathbf{A}\boxplus\mathbf{B}$, and the same holds for the Brouwer complement
in the case when ${\mathbb{V}}=\mathbb{BZL}$; hence, whenever $\alpha
\boxplus\beta$ is a lattice congruence of $\mathbf{A}\boxplus\mathbf{B}$, it
is a full congruence.
\end{proof}

For ${\mathbb{V}}=\mathbb{BI}$, this result was proven in \cite{eucardbi},
but, for the sake of completeness, we provide a new proof for it here.

\begin{proposition}
\label{cghsum}Let ${\mathbb{V}}$ be one of the varieties $\mathbb{BI}$ and
$\mathbb{BZL}$ and $\mathbf{A}$ and $\mathbf{B}$ be members of ${\mathbb{V}}$
with $|A|>2$ and $|B|>2$ such that $\mathbf{A}\boxplus\mathbf{B}\in
{\mathbb{V}}$. Then:

\begin{itemize}
\item $\mathrm{Con}_{{\mathbb{V}}01}(\mathbf{A}\boxplus\mathbf{B}%
)=\{\alpha\boxplus\beta:\alpha\in\mathrm{Con}_{{\mathbb{V}}01}(\mathbf{A}%
),\beta\in\mathrm{Con}_{{\mathbb{V}}01}(\mathbf{B})\}\cong\mathrm{Con}%
_{{\mathbb{V}}01}(\mathbf{A})\times\mathrm{Con}_{{\mathbb{V}}01}(\mathbf{B})$;

\item $\mathrm{Con}_{{\mathbb{V}}}(\mathbf{A}\boxplus\mathbf{B})=\mathrm{Con}%
_{{\mathbb{V}}01}(\mathbf{A}\boxplus\mathbf{B})\cup\{\nabla_{A\boxplus
B}\}\cong(\mathrm{Con}_{{\mathbb{V}}01}(\mathbf{A})\times\mathrm{Con}%
_{{\mathbb{V}}01}(\mathbf{B}))\oplus\mathbf{D}_{2}$.
\end{itemize}
\end{proposition}

\begin{proof}
Every proper congruence $\theta$ of $\mathbf{A}\boxplus\mathbf{B}$ satisfies
$\theta=(\theta\cap A^{2})\boxplus(\theta\cap B^{2})$, with $\theta\cap
A^{2}\in\mathrm{Con}_{{\mathbb{V}}}(\mathbf{A})\backslash\left\{  \nabla
_{A}\right\}  $ and $\theta\cap B^{2}\in\mathrm{Con}_{{\mathbb{V}}}%
(\mathbf{B})\backslash\left\{  \nabla_{B}\right\}  $. Let $D=\{\alpha
\boxplus\beta:\alpha\in\mathrm{Con}_{{\mathbb{V}}01}(\mathbf{A}),\beta
\in\mathrm{Con}_{{\mathbb{V}}01}(\mathbf{B})\}$. By Lemmas \ref{hsumfullcg}
and \ref{cghsumlat}, it follows that:%
\begin{align*}
\{\nabla_{A\boxplus B}\}\cup D  &  \subseteq\mathrm{Con}_{{\mathbb{V}}%
}(\mathbf{A}\boxplus\mathbf{B})\\
&  \subseteq\{\nabla_{A\boxplus B},eq(A\setminus\{0\},B\setminus
\{1\}),eq(A\setminus\{1\},B\setminus\{0\})\}\cup D.
\end{align*}

Since $|A|>2$, there exists an $a\in A\setminus\{0,1\}$, so that $a^{\prime
}\in A\setminus\{0,1\}$. Note that $eq(A\setminus\{0\},B\setminus\{1\})\cap
A^{2}$ contains $(a,0)$ but not $(a^{\prime},1)$, while $eq(A\setminus
\{1\},B\setminus\{0\})\cap A^{2}$ contains $(a,1)$ but not $(a^{\prime},0)$. Hence $eq(A\setminus\{0\},B\setminus\{1\})$ $\notin\mathrm{Con}_{{\mathbb{V}}}(\mathbf{A}\boxplus\mathbf{B})$ and $eq(A\setminus
\{1\},B\setminus\{0\})\notin\mathrm{Con}_{{\mathbb{V}}}(\mathbf{A}\boxplus\mathbf{B})$.

Therefore $\mathrm{Con}_{{\mathbb{V}}}(\mathbf{A}\boxplus\mathbf{B}%
)=\{\nabla_{A\boxplus B}\}\cup\{\alpha\boxplus\beta:\alpha\in\mathrm{Con}%
_{{\mathbb{V}}01}(\mathbf{A}),\beta\in\mathrm{Con}_{{\mathbb{V}}01}%
(\mathbf{B})\}$, so, clearly,\begin{align*}
\mathrm{Con}_{{\mathbb{V}}01}(\mathbf{A}\boxplus\mathbf{B})  &  =\{\alpha
\boxplus\beta:\alpha\in\mathrm{Con}_{{\mathbb{V}}01}(\mathbf{A}),\beta
\in\mathrm{Con}_{{\mathbb{V}}01}(\mathbf{B})\}\\
&  \cong\mathrm{Con}_{{\mathbb{V}}01}(\mathbf{A})\times\mathrm{Con}%
_{{\mathbb{V}}01}(\mathbf{B}).
\end{align*}
Hence
\begin{align*}
\mathrm{Con}_{{\mathbb{V}}}(\mathbf{A}\boxplus\mathbf{B})  &  =\{\nabla
_{A\boxplus B}\}\cup\mathrm{Con}_{{\mathbb{V}}01}(\mathbf{A}\boxplus
\mathbf{B})\\
&  \cong\mathrm{Con}_{{\mathbb{V}}01}(\mathbf{A}\boxplus\mathbf{B}%
)\oplus\mathbf{D}_{2}\\
&  \cong(\mathrm{Con}_{{\mathbb{V}}01}(\mathbf{A})\times\mathrm{Con}%
_{{\mathbb{V}}01}(\mathbf{B}))\oplus\mathbf{D}_{2}.
\end{align*}

\end{proof}

\begin{corollary}
Let ${\mathbb{V}}$ be one of the varieties $\mathbb{BI}$ and $\mathbb{BZL}$,
$n\in{\mathbb{N}}\setminus\{0,1\}$ and $\mathbf{A}_{1},\ldots,\mathbf{A}_{n}$
be members of ${\mathbb{V}}$ with $|A_{i}|>2$ for all $i\in[1,n]$. Then:

\begin{itemize}
\item $\mathrm{Con}_{{\mathbb{V}} 01}(\boxplus_{i=1}^{n}\mathbf{A}%
_{i})=\{\boxplus_{i=1}^{n}\alpha_{i}:(\forall\, i\in[1,n])\,
(\alpha_{i}\in\mathrm{Con}_{{\mathbb{V}} 01}(\mathbf{A}_{i}))\}\cong%
$\linebreak$\prod_{i=1}^{n}\mathrm{Con}_{{\mathbb{V}} 01}(\mathbf{A}_{i})$;

\item $\mathrm{Con}_{{\mathbb{V}} }(\boxplus_{i=1}^{n}\mathbf{A}%
_{i})=\mathrm{Con}_{{\mathbb{V}} 01}(\boxplus_{i=1}^{n}\mathbf{A}_{i}%
)\cup\{\nabla_{A\boxplus B}\}\cong(\prod_{i=1}^{n}\mathrm{Con}_{{\mathbb{V}}
01}(\mathbf{A}_{i}))\oplus\mathbf{D}_{2}$;

\item $\boxplus_{i=1}^{n}\mathbf{A}_{i}$ is subdirectly irreducible as a
member of ${\mathbb{V}}$ iff, for some $k\in[1,n]$, either
$\mathrm{Con}_{{\mathbb{V}}01}(\mathbf{A}_{k})=\{\Delta_{A_{k}}\}$ or
$\mathrm{Con}_{{\mathbb{V}}01}(\mathbf{A}_{k})$ has a single atom, and, for
every $i\in[1,n]\setminus\{k\}$, $\mathrm{Con}_{{\mathbb{V}}%
01}(\mathbf{A}_{i})=\{\Delta_{A_{i}}\}$.
\end{itemize}
\end{corollary}

Recall from \textup{\cite[Prop. 4.3]{BH} that if }$\mathbf{L}$ is an orthomodular lattice, then $\mathbf{L}$ is congruence--regular and $\mathrm{Con}_{\mathbb{BZL}}(\mathbf{L})=\mathrm{Con}_{\mathbb{BI}}(\mathbf{L})=\mathrm{Con}(\mathbf{L})$.

\begin{proposition}\label{cghsumpbz}Let $\mathbf{A}$ be an orthomodular lattice and $\mathbf{B}$ a BZ--lattice with $|A|>2$ and $|B|>2$. Then:

\begin{enumerate}
\item \label{cghsumpbz1} $\mathrm{Con}_{\mathbb{BZL}01}(\mathbf{A}\boxplus\mathbf{B})=\{\Delta_{A}\boxplus\beta:\beta\in\mathrm{Con}_{\mathbb{BZL}01}(\mathbf{B})\}\cong\mathrm{Con}_{\mathbb{BZL}01}(\mathbf{B})$
and $\mathrm{Con}_{\mathbb{BZL}}(\mathbf{A}\boxplus\mathbf{B})=\mathrm{Con}_{\mathbb{BZL}01}(\mathbf{A}\boxplus\mathbf{B})\cup\{\nabla_{A\boxplus B}\}\cong\mathrm{Con}_{\mathbb{BZL}01}(\mathbf{B})\oplus\mathbf{D}_{2}$;

\item \label{cghsumpbz2} if $\mathbf{B}$ is an antiortholattice, then $\mathrm{Con}_{\mathbb{BZL}}(\mathbf{A}\boxplus\mathbf{B})=\{\Delta _{A}\boxplus\beta:\beta\in\mathrm{Con}_{\mathbb{BZL}}(\mathbf{B})\setminus\{\nabla_{B}\}\}\cup\{\nabla_{A\boxplus B}\}\cong\mathrm{Con}_{\mathbb{BZL}}(\mathbf{B})$.
\end{enumerate}\end{proposition}

\begin{proof} (\ref{cghsumpbz1}) Since $\mathrm{Con}_{\mathbb{BZL}01}(\mathbf{A})=\{\Delta_{A}\}\cong\mathbf{D}_{1}$, the result follows from Proposition
\ref{cghsum}.

\noindent (\ref{cghsumpbz2}) By (\ref{cghsumpbz1}) and Proposition \ref{bibzl}.(\ref{bibzl2}),
$\mathrm{Con}_{\mathbb{BZL}01}(\mathbf{B})=\mathrm{Con}_{\mathbb{BZL}%
}(\mathbf{B})\setminus\{\nabla_{B}\}$ and $\mathrm{Con}_{\mathbb{BZL}%
}(\mathbf{B})\cong\mathrm{Con}_{\mathbb{BZL}01}(\mathbf{B})\oplus
\mathbf{D}_{2}\cong\mathrm{Con}_{\mathbb{BZL}}(\mathbf{A}\boxplus\mathbf{B})$.\end{proof}

We now list a few corollaries of the results obtained so far.

\begin{corollary}
Let $\mathbf{A}$ be an orthomodular lattice and $\mathbf{B}$ a BZ--lattice with $|A|>2$ and $|B|>2$. Then:

\begin{enumerate}
\item \label{sih0} $\mathbf{A}\boxplus\mathbf{B}$ is simple iff $\mathrm{Con}_{\mathbb{BZL}01}(\mathbf{B})=\{\Delta_{B}\}$;

\item \label{sih1} $\mathbf{A}\boxplus\mathbf{B}$ is subdirectly irreducible
iff either $\mathrm{Con}_{\mathbb{BZL}01}(\mathbf{B})=\{\Delta_{B}\}$
or\linebreak$\mathrm{Con}_{\mathbb{BZL}01}(\mathbf{B})$ has a single atom;

\item \label{sih2} if $\mathbf{B}$ is an antiortholattice, then:
$\mathbf{A}\boxplus\mathbf{B}$ is subdirectly irreducible iff $\mathbf{B}$ is
subdirectly irreducible;

\item \label{sih3} if $\mathbf{B}$ is an antiortholattice chain, then: $\mathbf{A}\boxplus\mathbf{B}$ is subdirectly irreducible iff $|B|\leq 5$.\end{enumerate}\label{sih}\end{corollary}

\begin{proof} (\ref{sih0}),(\ref{sih1}) By Proposition \ref{cghsumpbz}.(\ref{cghsumpbz1}).

\noindent (\ref{sih2}) By Proposition \ref{cghsumpbz}.(\ref{cghsumpbz2}).

\noindent (\ref{sih3}) By (\ref{sih2}) and Corollary \ref{2ndcoraol}.(\ref{2ndcoraol2}).\end{proof}

By Proposition \ref{pbzirr}, the \PBZ --lattices $\mathbf{L}$ with all elements in $L\setminus\{0,1\}$ join--irreducible belong to the subvariety $\mathbb{HPBZL}^{\ast}$ of $\PBZs $ generated by the horizontal sums of antiortholattice chains with arbitrary horizontal sums of Boolean algebras, which is generated by its finite members according to \cite[Corollary 4.1]{PBZ2}, so the subvariety generated by these \PBZ --lattices is generated by its finite subdirectly irreducible members, hence Corollary \ref{sih}.(\ref{sih3}) gives us:

\begin{corollary} $V(\{{\bf L}\in \PBZs :$ all elements of $L\setminus \{0,1\}$ are join--irreducible in ${\bf L}_l\})=V(\{{\bf MO}_k\boxplus {\bf D}_n:k\in \N ^*,n\in [2,5]\})$.\end{corollary}

\begin{corollary}
Let $\mathbf{A}\in\mathbb{OML}\setminus\{\mathbf{D}_{1},\mathbf{D}_{2}\}$, and
let $(\mathbf{B}_{i})_{i\in I}$ be a nonempty family of nontrivial
antiortholattices such that $\mathbf{B}=\prod_{i\in I}\mathbf{B}_{i}%
\ncong\mathbf{D}_{2}$. Then:

\begin{enumerate}
\item \label{genfavex1} if $\mathbf{A}\boxplus\mathbf{B}$ is simple, then
$\mathbf{B}_{i}$ is simple for each $i\in I$;

\item \label{genfavex2} if $\prod_{i\in I}\mathbf{B}_{i}$ has no skew
congruences, in particular if $I$ is finite, then: $\mathbf{A}\boxplus
\mathbf{B}$ is simple iff $\mathbf{B}_{i}$ is simple for each $i\in I$;

\item \label{genfavex3} if $\prod_{i\in I}\mathbf{B}_{i}$ has no skew
congruences, in particular if $I$ is finite, then: $\mathbf{A}\boxplus
\mathbf{B}$ is subdirectly irreducible iff $\mathbf{B}_{i}$ is simple for all
$i\in I$ or, for some $j\in I$, $\mathbf{B}_{j}$ is subdirectly irreducible,
but not simple, and $\mathrm{Con}_{\mathbb{BZL}01}(\mathbf{B}_{i})$ has no
atoms for any $i\in I\setminus\{j\}$;

\item \label{genfavex4} if $\prod_{i\in I}\mathbf{B}_{i}$ has no skew
congruences, in particular if $I$ is finite, and\linebreak$\mathrm{Con}%
_{\mathbb{BZL}}(\mathbf{B}_{i})$ is finite for all $i\in I$, then:
$\mathbf{A}\boxplus\mathbf{B}$ is subdirectly irreducible iff, for some $j\in
I$, $\mathbf{B}_{j}$ is subdirectly irreducible and $\mathbf{B}_{i}$ is simple
for all $i\in I\setminus\{j\}$.
\end{enumerate}

\label{genfavex}
\end{corollary}

\begin{proof}
For all $i\in I$, $\mathbf{B}_{i}\in\mathbb{AOL} \setminus\{\mathbf{D}_{1}\}$,
hence $\mathrm{Con}_{\mathbb{BZL} }(\mathbf{B}_{i})=\mathrm{Con}_{\mathbb{BZL}
01}(\mathbf{B}_{i})\cup\{\nabla_{B_{i}}\}\cong\mathrm{Con}_{\mathbb{BZL}
01}(\mathbf{B}_{i})\oplus\mathbf{D}_{2}$, so that $\mathbf{B}_{i}$ is simple
iff $\mathrm{Con}_{\mathbb{BZL} 01}(\mathbf{B}_{i})=\{\Delta_{B_{i}}\}$, which
has no atoms, and $\mathbf{B}_{i}$ is subdirectly irreducible iff either
$\mathrm{Con}_{\mathbb{BZL} 01}(\mathbf{B}_{i})=\{\Delta_{B_{i}}\}$ or
$|\mathrm{At}(\mathrm{Con}_{\mathbb{BZL} 01}(\mathbf{B}_{i}))|=1$.

(\ref{genfavex1}) Clearly, $\mathrm{Con}_{\mathbb{BZL}01}%
(\mathbf{B})\supseteq\{\prod_{i\in I}\beta_{i}:(\forall\,i\in I)\,(\beta
_{i}\in\mathrm{Con}_{\mathbb{BZL}01}(\mathbf{B}_{i}))\}$, so, if
$\mathbf{B}_{k}$ is not simple for some $k\in I$, then $\mathrm{Con}%
_{\mathbb{BZL}01}(\mathbf{B})\supsetneq\{\Delta_{B}\}$, so by Corollary
\ref{sih}.(\ref{sih0}) $\mathbf{A}\boxplus\mathbf{B}$ is not simple.

(\ref{genfavex2}) If $\prod_{i\in I}\mathbf{B}_{i}$ has no skew
congruences, then $\mathrm{Con}_{\mathbb{BZL}01}(\mathbf{B})=\{\prod_{i\in
I}\beta_{i}:(\forall\,i\in I)\,(\beta_{i}\in\mathrm{Con}_{\mathbb{BZL}%
01}(\mathbf{B}_{i}))\}$, so by Corollary \ref{sih}.(\ref{sih0}) $\mathbf{A}%
\boxplus\mathbf{B}$ is simple iff $\mathrm{Con}_{\mathbb{BZL}01}%
(\mathbf{B})=\{\Delta_{B}\}$ iff $\mathrm{Con}_{\mathbb{BZL}01}(\mathbf{B}%
_{i})=\{\Delta_{B_{i}}\}$ for each $i\in I$ iff $\mathbf{B}_{i}$ is simple for
each $i\in I$.

(\ref{genfavex3}) If $\prod_{i\in I}\mathbf{B}_{i}$ has no skew
congruences, then $\mathrm{Con}_{\mathbb{BZL}01}(\mathbf{B})=\{\prod_{i\in
I}\beta_{i}:(\forall\,i\in I)\,(\beta_{i}\in\mathrm{Con}_{\mathbb{BZL}%
01}(\mathbf{B}_{i}))\}$, from which it is easy to derive that
\[
\mathrm{At}(\mathrm{Con}_{\mathbb{BZL}01}(\mathbf{B}))=\bigcup_{j\in
I}\{\alpha_{j}\times\prod_{i\in I\setminus\{j\}}\Delta_{B_{i}}:\alpha_{j}%
\in\mathrm{At}(\mathrm{Con}_{\mathbb{BZL}01}(\mathbf{B}_{j}))\}.
\]
Set $\kappa=|\mathrm{At}(\mathrm{Con}_{\mathbb{BZL}01}(\mathbf{B}))|$, and
$\kappa_{i}=|\mathrm{At}(\mathrm{Con}_{\mathbb{BZL}01}(\mathbf{B}_{i}))|$, for
all $i\in I$. Thus $\kappa=\sum_{i\in I}\kappa_{i}$, and hence, by
(\ref{genfavex2}):
\[%
\begin{tabular}
[c]{lll}%
$\mathbf{A}\boxplus\mathbf{B}$ is s.i. & iff & $\mathbf{A}\boxplus\mathbf{B}$
is simple or $\kappa=1$\\
& iff & $\mathbf{B}_{i}$ is simple for all $i\in I$ or, for some $j\in I$,
$\kappa_{j}=1$\\
&  & and $\kappa_{i}=0$ for any $i\in I\setminus\{j\}$\\
& iff & $\mathbf{B}_{i}$ is simple for all $i\in I$ or, for some $j\in I$,
$\mathbf{B}_{j}$ is s.i.,\\
&  & but not simple, and $\kappa_{i}=0$ for any $i\in I\setminus\{j\}$.
\end{tabular}
\
\]

(\ref{genfavex4}) By (\ref{genfavex3}) and the fact that, if, for
some $j\in I$, $\mathrm{Con}_{\mathbb{BZL}}(\mathbf{B}_{j})$ is finite, then:
$\mathbf{B}_{j}$ is simple iff $\mathrm{Con}_{\mathbb{BZL}01}(\mathbf{B}%
_{j})=\{\Delta_{B_{j}}\}$ iff $\kappa_{j}=0$.
\end{proof}

\begin{corollary}
Let $\mathbf{A}\in\mathbb{OML}\setminus\{\mathbf{D}_{1},\mathbf{D}_{2}\}$, $I$
be a non--empty set and $(\mathbf{K}_{i})_{i\in I}\subseteq\mathbb{PKA}$. For
all $i\in I$, we consider the antiortholattice $\mathbf{B}_{i}=\mathbf{D}%
_{2}\oplus\mathbf{K}_{i}\oplus\mathbf{D}_{2}$, and we let $\mathbf{B}%
=\prod_{i\in I}\mathbf{B}_{i}$. Then:

\begin{enumerate}
\item \label{alsogenfavex1} if $\mathbf{A}\boxplus\mathbf{B}$ is simple, then
$\mathbf{K}_{i}\cong\mathbf{D}_{1}$ for each $i\in I$;

\item \label{alsogenfavex2} if $\prod_{i\in I}\mathbf{B}_{i}$ has no skew
congruences, in particular if $I$ is finite, then: $\mathbf{A}\boxplus
\mathbf{B}$ is simple iff $\mathbf{K}_{i}\cong\mathbf{D}_{1}$ for each $i\in
I$;

\item \label{alsogenfavex3} if $\prod_{i\in I}\mathbf{B}_{i}$ has no skew
congruences, in particular if $I$ is finite, then: $\mathbf{A}\boxplus
\mathbf{B}$ is subdirectly irreducible iff $\mathbf{K}_{i}\cong\mathbf{D}_{1}$ (that is $\mathbf{B}_{i}\cong\mathbf{D}_{3}$)
for all $i\in I$ or, for some $j\in I$, the BI-lattice $\mathbf{K}_{j}$ is
nontrivial and subdirectly irreducible and $\mathrm{Con}_{\mathbb{BI}%
}(\mathbf{K}_{i})$ has no atoms for any $i\in I\setminus\{j\}$;

\item \label{alsogenfavex4} if $\prod_{i\in I}\mathbf{B}_{i}$ has no skew congruences, in particular if $I$ is finite, and\linebreak$\mathrm{Con}_{\mathbb{BI}}(\mathbf{K}_{i})$ is finite for all $i\in I$, then: $\mathbf{A}\boxplus\mathbf{B}$ is subdirectly irreducible iff, for some $j\in I$, the BI-lattice $\mathbf{K}_{j}$ is subdirectly irreducible and $\mathbf{K}_{i}\cong\mathbf{D}_{1}$ for all $i\in I\setminus\{j\}$.\end{enumerate}\label{alsogenfavex}\end{corollary}

\begin{proof} By Corollary \ref{coraol}.(\ref{coraol2}) and Corollary \ref{genfavex}.\end{proof}

\section{Varieties of \PBZ --Lattices Generated by Ho\-ri\-zon\-tal Sums\label{notevole}}

The aim of this final section is to investigate the structures of varieties of PBZ$^{\ast}$-lattices generated by horizontal sums and to provide axiomatic
bases for some of them. In particular, we will give a basis for $V\left({\mathbb{OML}}\boxplus{\mathbb{AOL}}\right)$ relative to $\mathbb{PBZL}^{\mathbb{\ast}}$, while the problem of finding a basis for $V\left( {\mathbb{OML}}\boxplus V\left(  {\mathbb{AOL}}\right)  \right)  $ is left
open. In the process, we give a different proof to the axiomatization of the varietal join ${\mathbb{OML}}\vee V\left(  {\mathbb{AOL}}\right)  $ relative to $\PBZs $, established in \cite{rgcmfp}.

\begin{proposition}
\label{closedhsum}Let ${\mathbb{V}}$ be the variety of bounded lattices or one
of the varieties $\mathbb{BI}$ and $\mathbb{BZL}$ and ${\mathbb{C}}$ and
${\mathbb{D}}$ be subclasses of ${\mathbb{V}}$. Then:

\begin{enumerate}
\item \label{closedhsum1} if ${\mathbb{C}} $ and ${\mathbb{D}} $ are closed
under subalgebras, then ${\mathbb{C}} \boxplus{\mathbb{D}} $ is closed under subalgebras;

\item \label{closedhsum2} if ${\mathbb{C}}$ and ${\mathbb{D}}$ are closed
under quotients, then ${\mathbb{C}}\boxplus{\mathbb{D}}$ is closed under quotients.
\end{enumerate}
\end{proposition}

\begin{proof}
By Lemma \ref{hsumsubquo}.
\end{proof}

\begin{corollary}
\label{closedsubquo}$\mathbb{OML}\boxplus\mathbb{AOL}$ and $\mathbb{OML}%
\boxplus V(\mathbb{AOL})$ are closed w.r.t. subalgebras and quotients.
\end{corollary}

Observe that for any $\mathbf{L}\in\mathbb{OML}\boxplus\mathbb{AOL}$,
$\mathrm{Con}_{\mathbb{BZL}}(\mathbf{L})=\mathrm{Con}_{\mathbb{BZL}%
01}(\mathbf{L})\cup\left\{  \nabla\right\}  $, so, by Lemma \ref{prodsubquo},
$\mathbf{T}(\mathbf{L}/\theta)=\mathbf{T}(\mathbf{L})/\theta$ for any
$\theta\in\mathrm{Con}_{\mathbb{BZL}}(\mathbf{L})$.

The next batch of results is about the De Morgan laws $SDM$ and $WSDM$. In
particular, we show that $WSDM$ is satisfied in $\mathbb{OML}\boxplus
\mathbb{AOL}$ only in limit cases.

\begin{proposition}
If $\mathbf{A}\in\mathbb{PBZL}^{\mathbb{\ast}}\setminus\mathbb{AOL}$ and
$\mathbf{B}\in\mathbb{PBZL}^{\mathbb{\ast}}\setminus\mathbb{OML}$, then the
algebra $\mathbf{A}\boxplus\mathbf{B}$ fails WSDM.\label{hsumfail1}
\end{proposition}

\begin{proof}
It follows from our assumptions that neither algebra is $\mathbf{D}_{1}$ or
$\mathbf{D}_{2}$, hence $|A|>2$ and $|B|>2$. Thus there exist an $x\in
B\setminus\{0,1\}$ and a $y\in A\setminus\{0,1\}$ and, moreover, we can choose
$y$ such $y^{\sim}\neq0$, because $\mathbf{A}$ is not an antiortholattice. But
then $y^{\sim}\neq1\neq\Diamond y$, so we have $\{y,y^{\sim},\Diamond
y\}\cap\{0,1\}=\emptyset$. We obtain: $(x\wedge y^{\sim})^{\sim}=0^{\sim
}=1\neq\Diamond y=0\vee\Diamond y=x^{\sim}\vee\Diamond y$, so $\mathbf{A}%
\boxplus\mathbf{B}$ fails WSDM.
\end{proof}

Some corollaries follow.

\begin{corollary}
\begin{itemize}
\item If $\mathbf{A}$ is an orthomodular lattice with $|A|>2$ and
$\mathbf{B}\in\mathbb{PBZL}^{\ast}\setminus\mathbb{OML}$, then the PBZ$^{\ast
}$ --lattice $\mathbf{A}\boxplus\mathbf{B}$ fails WSDM.

\item If $\mathbf{A}$ is an orthomodular lattice and $\mathbf{B}$ is an
antiortholattice such that $|A|>2$ and $|B|>2$, then the PBZ$^{\ast}$
--lattice $\mathbf{A}\boxplus\mathbf{B}$ fails WSDM.
\end{itemize}

\label{hfail1}
\end{corollary}

\begin{corollary}
\label{when1ok}\label{notinolvaol}

\begin{itemize}
\item Let $\mathbf{L}\in\mathbb{OML}\boxplus\mathbb{AOL}$. Then:
$\mathbf{L}\vDash WSDM$ iff $\mathbf{L}\in\mathbb{OML}\cup\mathbb{AOL}$.

\item For all $\mathbf{A}\in\mathbb{OML}\setminus\{\mathbf{D}_{1}%
,\mathbf{D}_{2}\}$ and all $\mathbf{B}\in\mathbb{PBZL}^{\ast}\setminus
\mathbb{OML}$, we have $\mathbf{A}\boxplus\mathbf{B}\notin\mathbb{OML}\vee
V(\mathbb{AOL})$.

\item For all $\mathbf{A}\in\mathbb{OML}\setminus\{\mathbf{D}_{1}%
,\mathbf{D}_{2}\}$ and all $\mathbf{B}\in\mathbb{AOL}\setminus\{\mathbf{D}%
_{1},\mathbf{D}_{2}\}$, we have $\mathbf{A}\boxplus\mathbf{B}\notin%
\mathbb{OML}\vee V(\mathbb{AOL})$.
\end{itemize}
\end{corollary}

\begin{proof}
By Corollary \ref{hfail1} and the fact that $\mathbb{OML}\vee V(\mathbb{AOL}%
)\vDash WSDM$.
\end{proof}

\begin{corollary}\begin{itemize}
\item $(\mathbb{OML}\boxplus\mathbb{AOL})\cap V(\mathbb{AOL})=\mathbb{AOL}$.
\item $(\mathbb{OML}\boxplus\mathbb{AOL})\cap (\mathbb{OML}\vee V(\mathbb{AOL}))=\mathbb{OML}\cup \mathbb{AOL}$.\end{itemize}\label{omlaolvaol}
\end{corollary}

\begin{corollary}
\label{twithsdm}If a PBZ$^{\ast}$ --lattice $\mathbf{L}$ satisfies the SDM, then:

\begin{enumerate}
\item \label{twithsdm1} $\langle T(\mathbf{L})\rangle_{\mathbb{BI}}=T(\mathbf{L})\cup T(\mathbf{L})^{\prime}$;

\item \label{twithsdm2} $\langle T(\mathbf{L})\rangle_{\mathbb{BZL}%
}=T(\mathbf{L})\cup T(\mathbf{L})^{\prime}$ iff either $\mathbf{L}%
\in\mathbb{OML}$ or $\mathbf{L}\in\mathbb{AOL}$ and $0$ is meet--irreducible
in $\mathbf{L}$ iff $\langle T(\mathbf{L})\rangle_{\mathbb{BZL}}%
=T(\mathbf{L})$ iff $T(\mathbf{L})$ is a subuniverse of $\mathbf{L}$.
\end{enumerate}
\end{corollary}

\begin{proof}
(\ref{twithsdm1}) By Lemma \ref{tfilter}.

(\ref{twithsdm2}) By (\ref{twithsdm1}), Lemma \ref{tfilter}, Theorem \ref{charg}, Corollary \ref{when1ok} and Lemma \ref{olsaolt}.
\end{proof}

We now examine condition $J2$. It turns out that this weakened form of orthomodularity characterizes horizontal sums of an orthomodular lattice and of an antiortholattice among all horizontal sums of PBZ*-lattices.

\begin{proposition}
$\mathbb{OML}\boxplus\mathbb{AOL}\vDash J2$.\label{j2olaol}
\end{proposition}

\begin{proof}
We know that $\mathbb{OML}\vDash J2$ and $\mathbb{AOL}\vDash J2$. Now let
$\mathbf{A}\in\mathbb{OML}$ and $\mathbf{B}\in\mathbb{AOL}$ with $|A|>2$ and
$|B|>2$, and let $\mathbf{L}=\mathbf{A}\boxplus\mathbf{B}$. Then by Lemma
\ref{hsumolaol}, $\mathbf{S}(\mathbf{L})=\mathbf{A}$, from which it easily
follows that $\mathbf{L}\vDash_{L,A}J2$. Since $\mathbf{B}$ is an
antiortholattice, we have $S(\mathbf{B})=\{0,1\}$, so, for every $y\in
L\setminus A=B\setminus\{0,1\}=B\setminus S(\mathbf{B})$, we have $y\wedge
y^{\prime}\neq0$, thus $(y\wedge y^{\prime})^{\sim}=0$, so $\Diamond(y\wedge
y^{\prime})=1$, from which it easily follows that $\mathbf{L}\vDash
_{L,L\setminus A}J2$. Therefore $\mathbf{L}\vDash J2$.
\end{proof}

\begin{theorem}
\label{axhsum}Let $\mathbf{A}\in\mathbb{OML}\setminus\{\mathbf{D}_{1}\}$ and
$\mathbf{B}\in\mathbb{PBZL}^{\ast}\setminus\{\mathbf{D}_{1}\}$. Then:

\begin{itemize}
\item $\mathbf{A}\boxplus\mathbf{B}\vDash S1$ iff $\mathbf{B}\vDash S1$;

\item $\mathbf{A}\boxplus\mathbf{B}\vDash S2$ iff $\mathbf{B}\vDash S2$;

\item $\mathbf{A}\boxplus\mathbf{B}\vDash S3$ iff $\mathbf{B}\vDash S3$;

\item $\mathbf{A}\boxplus\mathbf{B}\vDash J1$ iff $\mathbf{B}\vDash J1$;

\item $\mathbf{A}\boxplus\mathbf{B}\vDash J2$ iff $\mathbf{B}\in
\mathbb{OML}\boxplus\mathbb{AOL}$.
\end{itemize}
\end{theorem}

\begin{proof}
For the first four equivalences, the left-to-right implications are trivial,
recalling that $\mathbb{OML}\vDash\{J1,S1,S2,S3\}$.

Denote by $\mathbf{L}=\mathbf{A}\boxplus\mathbf{B}$, which is a PBZ$^{\ast}$
--lattice by Proposition \ref{hsumkl}.(\ref{hsumkl2}), and note that
$L\setminus A=B\setminus\{0,1\}$ and $L\setminus B=A\setminus\{0,1\}$. By the
above, to prove the right-to-left implications in the first four equivalences,
it suffices to show that $\mathbf{L}\vDash_{A\setminus\{0,1\},B\setminus
\{0,1\}}\{J1,S1,S2,S3\}$ and $\mathbf{L}\vDash_{B\setminus\{0,1\},A\setminus
\{0,1\}}\{J1,S1,S2,S3\}$.

For all $a\in A\setminus\{0,1\}$ and all $B\setminus\{0,1\}$, we have $a\wedge
b=0$, thus $(a\wedge b)^{\sim}=1$, so $\Diamond(a\wedge b)=0$. Since
$A=S(\mathbf{A})$, $a\wedge a^{\prime}=0$, hence $(a\wedge a^{\prime})^{\sim
}=1$, thus $\Diamond(a\wedge a^{\prime})=0$, and, if $b\in S(\mathbf{B})$,
then $b\wedge b^{\prime}=0$, hence $(b\wedge b^{\prime})^{\sim}=1$, thus
$\Diamond(b\wedge b^{\prime})=0$. It immediately follows that $\mathbf{L}%
\vDash_{A\setminus\{0,1\},B\setminus\{0,1\}}\{J1,S1\}$, $\mathbf{L}%
\vDash_{B\setminus\{0,1\},A\setminus\{0,1\}}\{J1,S1\}$, $\mathbf{L}%
\vDash_{B\setminus\{0,1\},A\setminus\{0,1\}}\{J2,S2,S3\}$ and $\mathbf{L}%
\vDash_{A\setminus\{0,1\},S(\mathbf{B})\setminus\{0,1\}}\{J2,S2,S3\}$.

Now let $a\in A\setminus\{0,1\}$ and $b\in B\setminus S(\mathbf{B})\subseteq
B\setminus\{0,1\}$, so that $b\wedge b^{\prime}\notin\{0,1\}$, thus $(b\wedge
b^{\prime})^{\sim}\neq1$, hence $a\wedge(b\wedge b^{\prime})^{\sim}=0$, so
$(a\wedge(b\wedge b^{\prime})^{\sim})^{\sim}=1$. Since $A=S(\mathbf{A})$ and
$a\notin\{0,1\}$, it follows that $a^{\sim}=a^{\prime}\notin\{0,1\}$. Since
$0\neq b\wedge b^{\prime}\leq\Diamond(b\wedge b^{\prime})$, we have
$\Diamond(b\wedge b^{\prime})\neq0$, so $a^{\sim}\vee\Diamond(b\wedge
b^{\prime})=1$, thus $\mathbf{L}\vDash_{A\setminus\{0,1\},B\setminus
S(\mathbf{B})}S2$, therefore $\mathbf{L}\vDash_{A\setminus\{0,1\},B\setminus
\{0,1\}}S2$. Also, if $\Diamond(b\wedge b^{\prime})\neq1$, then $(b\wedge
b^{\prime})^{\sim}\neq0$ and $a\wedge\Diamond(b\wedge b^{\prime})=0$, so that
$(a\wedge\Diamond(b\wedge b^{\prime}))^{\sim}=1=a^{\sim}\vee(b\wedge
b^{\prime})^{\sim}$ since $a^{\sim}=a^{\prime}\neq0$, while $\Diamond(b\wedge
b^{\prime})=1$ implies $(b\wedge b^{\prime})^{\sim}=(\Diamond(b\wedge
b^{\prime}))^{\sim}=0$, so that $(a\wedge\Diamond(b\wedge b^{\prime}))^{\sim
}=a^{\sim}=a^{\sim}\vee(b\wedge b^{\prime})^{\sim}$, hence $\mathbf{L}%
\vDash_{A\setminus\{0,1\},B\setminus S(\mathbf{B})}S3$, therefore
$\mathbf{L}\vDash_{A\setminus\{0,1\},B\setminus\{0,1\}}S3$.

By Corollary \ref{cors}.(\ref{olhsum}), Proposition \ref{j2olaol} and the
commutativity and associativity of horizontal sums, $\mathbb{OML}%
\boxplus\mathbb{OML}\boxplus\mathbb{AOL}=\mathbb{OML}\boxplus\mathbb{AOL}%
\vDash J2$, which proves the right-to-left implication in the last
equivalence. Now assume that $\mathbf{L}\vDash J2$, so that $\mathbf{L}%
\vDash_{A\setminus\{0,1\},B\setminus S(\mathbf{B})}J2$, thus, for all $a\in
A\setminus\{0,1\}$ and all $b\in B\setminus S(\mathbf{B})$, $a=a\wedge
\Diamond(b\wedge b^{\prime})$, hence $a\leq\Diamond(b\wedge b^{\prime})$, thus
$\Diamond(b\wedge b^{\prime})=1$, so $(b\wedge b^{\prime})^{\sim}=0$, whence
$b^{\sim}\leq b^{\sim}\vee\square b=0$. Therefore $B\setminus S(\mathbf{B}%
)\subseteq T(\mathbf{B})$, hence $B=S(\mathbf{B})\cup T(\mathbf{B})$, thus
$\mathbf{B}\in\mathbb{OML}\boxplus\mathbb{AOL}$ by Theorem \ref{charg}.
\end{proof}

\begin{corollary}
\begin{itemize}
\item $\mathbb{OML}\boxplus\mathbb{AOL}\nvDash WSDM$;

\item $\mathbb{OML}\boxplus V(\mathbb{AOL})\vDash\{S1,S2,S3,J1\}$;

\item $\mathbb{OML}\boxplus\mathbb{AOL}\vDash J2$;

\item $\mathbb{OML}\boxplus V(\mathbb{AOL})\nvDash J2$.
\end{itemize}\label{axvarhsum}\end{corollary}

\begin{proof}
By Corollary \ref{hfail1}, $\mathbb{OML}\boxplus\mathbb{AOL}\nvDash WSDM$.

By Theorem \ref{axhsum} and the fact that $\mathbb{AOL}$, and thus
$V(\mathbb{AOL})$, satisfies $J1$, $J2$, $S1$, $S2$ and $S3$, we have:
$\mathbb{OML}\boxplus\mathbb{AOL}\vDash J2$ and $\mathbb{OML}\boxplus
V(\mathbb{AOL})\vDash\{S1,S2,S3,J1\}$.

The \PBZ --lattice $\mathbf{K}\in\mathbb{OML}\boxplus V(\mathbb{AOL})$ in Example \ref{exfail12} below fails $J2$, thus $\mathbb{OML}\boxplus V(\mathbb{AOL})\nvDash J2$.\end{proof}

\begin{corollary}
\begin{itemize}
\item $\mathbb{OML}\vee V(\mathbb{AOL})\subsetneq V(\mathbb{OML}%
\boxplus\mathbb{AOL})\subsetneq V(\mathbb{OML}\boxplus V(\mathbb{AOL}%
))$\linebreak$\subsetneq\mathbb{PBZL}^{\ast}$;

\item $\{\mathbf{L}\in\mathbb{OML}\boxplus V(\mathbb{AOL}):\mathbf{L}\vDash
J2\}=\mathbb{OML}\boxplus\mathbb{AOL}$.
\end{itemize}
\end{corollary}

\begin{proof}
We will use Corollary \ref{axvarhsum}.

$\mathbb{OML}\cup\mathbb{AOL}\subseteq\mathbb{OML}\boxplus\mathbb{AOL}%
\subseteq\mathbb{OML}\boxplus V(\mathbb{AOL})\subseteq\mathbb{PBZL}^{\ast}$ by
Proposition \ref{hsumkl}.(\ref{hsumkl2}), thus $\mathbb{OML}\vee
V(\mathbb{AOL})\subseteq V(\mathbb{OML}\boxplus\mathbb{AOL})\subseteq
V(\mathbb{OML}\boxplus V(\mathbb{AOL}))\subseteq\mathbb{PBZL}^{\ast}$.

All these inclusions are proper. Indeed, $\mathbb{OML}\boxplus\mathbb{AOL}$, and thus $V(\mathbb{OML}\boxplus\mathbb{AOL})$, fails WSDM, while $\mathbb{OML}\vee V(\mathbb{AOL})$ satisfies WSDM, therefore $\mathbb{OML}\vee V(\mathbb{AOL})\subsetneq V(\mathbb{OML}\boxplus\mathbb{AOL})$. $\mathbb{OML}\boxplus\mathbb{AOL}$, and thus $V(\mathbb{OML}\boxplus \mathbb{AOL})$, satisfies $J2$, while $\mathbb{OML}\boxplus V(\mathbb{AOL})$,
and thus $V(\mathbb{OML}\boxplus V(\mathbb{AOL}))$, fails $J2$, hence $V(\mathbb{OML}\boxplus\mathbb{AOL})\subsetneq V(\mathbb{OML}\boxplus
V(\mathbb{AOL}))$.

The PBZ$^{\ast}$ --lattice $\mathbf{L}$ in Example \ref{exfail12} below fails $S2$, while $\mathbb{OML}\boxplus V(\mathbb{AOL})$ and thus $V(\mathbb{OML}\boxplus
V(\mathbb{AOL}))$, satisfies $S2$. So $\mathbf{L}\in\PBZs \setminus V(\mathbb{OML}\boxplus V(\mathbb{AOL}))$, hence $V(\mathbb{OML}\boxplus V(\mathbb{AOL}))\subsetneq\mathbb{PBZL}^{\ast}$.

For the last bullet, the right-to-left inclusion follows from Corollary \ref{axvarhsum}, and the other inclusion from Theorem \ref{axhsum} and Corollary \ref{cors}.(\ref{olhsum}).\end{proof}

\begin{example}\label{exfail12} Let us consider the PBZ$^{\ast}$--lattices $\mathbf{M}_{3}$,
$\mathbf{K}$, $\mathbf{L}$ and $\mathbf{M}$, with the lattice orderings, elements and Kleene complements given by the diagrams below and:
$S(\mathbf{K})=\{0,u,u^{\prime},s,s^{\prime},1\}$, $T(\mathbf{K}%
)=\{0,t^{\prime},1\}$ and $t^{\sim}=s$ in $\mathbf{K}$, $S(\mathbf{L}%
)=\{0,s,s^{\prime},1\}$, $T(\mathbf{L})=\{u,t^{\prime}\}$ and $t^{\sim}=s$ in
$\mathbf{L}$, $S(\mathbf{M})=\{0,a,a^{\prime},1\}$, $T(\mathbf{M}%
)=\{0,z,t,u^{\prime},v^{\prime},z^{\prime},1\}$, $u^{\sim}=a$ and $v^{\sim
}=a^{\prime}$ in $\mathbf{M}$. By Corollaries \ref{notinolvaol} and
\ref{hfail1}, $\mathbf{M}_{3}=\mathbf{D}_{2}^{2}\boxplus\mathbf{D}_{3}$ fails
WSDM, thus $\mathbf{M}_{3}\in(\mathbb{OML}\boxplus\mathbb{AOL})\setminus
(\mathbb{OML}\vee V(\mathbb{AOL}))$.

$\mathbf{K}=\mathbf{D}_{2}^{2}\boxplus(\mathbf{D}_{2}\times\mathbf{D}_{3}%
)\in(\mathbb{OML}\boxplus V(\mathbb{AOL}))\setminus V(\mathbb{OML}%
\boxplus\mathbb{AOL})$, because $\mathbb{OML}\boxplus\mathbb{AOL}$, and thus
$V(\mathbb{OML}\boxplus\mathbb{AOL})$, satisfies $J2$ by Proposition
\ref{j2olaol}, while $\mathbf{K}\nvDash J2$, because, in $\mathbf{K}$,
$(u\wedge(t\wedge t^{\prime})^{\sim})\vee(u\wedge\Diamond(t\wedge t^{\prime
}))=(u\wedge t^{\sim})\vee(u\wedge\Diamond t)=(u\wedge s)\vee(u\wedge s^{\prime})=0\vee0=0\neq u$. Note, also, that $\mathbf{K}\vDash \{S2,S3\}$, by Corollary \ref{axvarhsum} and the fact that $\mathbf{K}\in \mathbb{OML}\boxplus V(\mathbb{AOL})$.

In $\mathbf{L}$, $u\vee t=u\neq t=0\vee t=(u\wedge s)\vee(u\wedge s^{\prime
})=((u\vee t)\wedge t^{\sim})\vee((u\vee t)\wedge t^{\sim\sim})$, therefore
$\mathbf{L}$ fails $J1$, and $(u\wedge(t\wedge t^{\prime})^{\sim})^{\sim
}=(u\wedge t^{\sim})^{\sim}=(u\wedge s)^{\sim}=0^{\sim}=1\neq s^{\prime}=0\vee
s^{\sim}=u^{\sim}\vee t^{\sim\sim}=u^{\sim}\vee\Diamond(t\wedge t^{\prime})$, therefore $\mathbf{L}$ fails $S2$. Furthermore, easy verifications establish that $\mathbf{L}$ satisfies $S3$ and fails $J2$.

We notice that $\mathbf{M}$ satisfies $J1$. Notice, also, that $\mathbf{M}$
fails $S1$, because, in $\mathbf{M}$, $(z^{\prime}\wedge(z^{\prime}\wedge
a)^{\sim})^{\sim}=(z^{\prime}\wedge u^{\sim})^{\sim}=(z^{\prime}\wedge
a)^{\sim}=u^{\sim}=a\neq a^{\prime}=a^{\sim}=\Diamond u=0\vee\Diamond
(z^{\prime}\wedge a)=z^{\prime\sim}\vee\Diamond(z^{\prime}\wedge a)$. Note, also, that $\langle T(\mathbf{M})\rangle_{\mathbb{BZL}}=\mathbf{M}$, so
$M=\langle T(\mathbf{M})\rangle_{\mathbb{BZL}}=S(\mathbf{M})\cup\left\langle T(\mathbf{M})\right\rangle _{\mathbb{BZL}}$. Since $\mathbf{M}\nvDash S1$, while $\mathbb{OML}\boxplus V(\mathbb{AOL})$, and thus $V(\mathbb{OML}\boxplus V(\mathbb{AOL}))$, satisfies $S1$ by Corollary \ref{axvarhsum}, it follows that $\mathbf{M}\notin V(\mathbb{OML}\boxplus V(\mathbb{AOL}))$, in particular $\mathbf{M}\notin\mathbb{OML}\boxplus V(\mathbb{AOL})$, so $\mathbf{M}\notin
V(\mathbb{AOL})$. Easy verifications establish, furthermore, that $\mathbf{M}$ fails each of $J2$, $S2$ and $S3$.

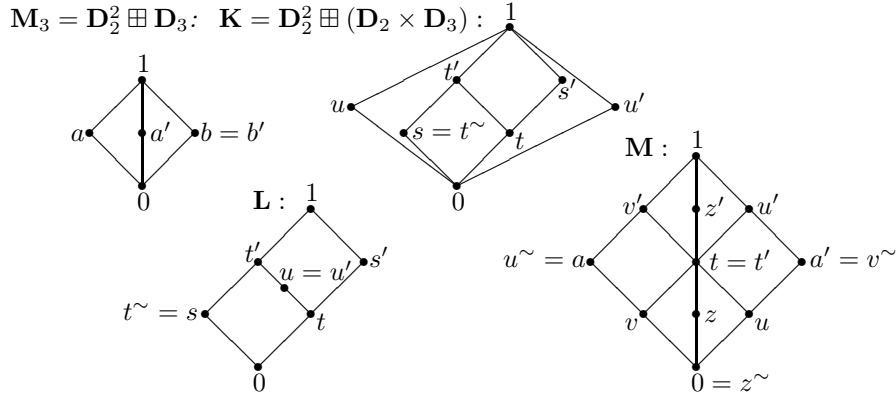
\begin{center}\begin{tabular}{cccc}
\hspace*{-30pt}\begin{picture}(40,65)(0,0)
\put(-30,60){${\bf M}_3={\bf D}_2^2\boxplus {\bf D}_3$:}
\put(20,0){\circle*{3}} \put(20,40){\circle*{3}}
\put(0,20){\circle*{3}} \put(20,20){\circle*{3}} \put(40,20){\circle*{3}}
\put(18,-9){$0$} \put(18,43){$1$}
\put(42,17){$b=b^{\prime }$} \put(-7,17){$a$} \put(23,17){$a^{\prime }$}
\put(20,0){\line(-1,1){20}} \put(20,40){\line(1,-1){20}}
\put(20,0){\line(1,1){20}} \put(0,20){\line(1,1){20}}
\put(20,0){\line(0,1){40}}\end{picture}
& \hspace*{55pt} &
\begin{picture}(40,65)(0,0)
\put(-70,60){${\bf K}={\bf D}_2^2\boxplus ({\bf D}_2\times {\bf D}_3):$} \put(20,0){\circle*{3}}
\put(20,40){\circle*{3}}
\put(0,20){\circle*{3}}
\put(60,40){\circle*{3}}
\put(40,20){\circle*{3}}
\put(40,60){\circle*{3}}
\put(-20,30){\circle*{3}}
\put(-28,28){$u$}
\put(80,30){\circle*{3}}
\put(83,28){$u^{\prime }$}
\put(18,-9){$0$}
\put(38,63){$1$}
\put(3,18){$s=t^{\sim }$}
\put(58,33){$s^{\prime }$}
\put(15,40){$t^{\prime }$}
\put(42,14){$t$}
\put(20,0){\line(-1,1){20}}
\put(20,40){\line(1,-1){20}}
\put(40,60){\line(1,-1){20}}
\put(20,0){\line(1,1){40}}
\put(0,20){\line(1,1){40}}
\put(20,0){\line(-4,3){40}}
\put(20,0){\line(2,1){60}}
\put(40,60){\line(4,-3){40}}
\put(40,60){\line(-2,-1){60}}
\end{picture} &\vspace*{-20pt}\\ 
&\hspace*{-35pt} \begin{picture}(40,85)(0,0)
\put(18,60){${\bf L}:$} \put(20,0){\circle*{3}} \put(20,40){\circle*{3}}
\put(0,20){\circle*{3}} \put(60,40){\circle*{3}} \put(30,30){\circle*{3}}
\put(40,20){\circle*{3}} \put(40,60){\circle*{3}} \put(18,-9){$0$} \put(38,63){$1$} \put(28,33){$u=u^{\prime }$}
\put(-31,18){$t^{\sim }=s$} \put(62,38){$s^{\prime }$} \put(15,40){$t^{\prime }$}
\put(42,14){$t$} \put(20,0){\line(-1,1){20}} \put(20,40){\line(1,-1){20}}
\put(40,60){\line(1,-1){20}} \put(20,0){\line(1,1){40}}
\put(0,20){\line(1,1){40}}\end{picture}
& &\hspace*{35pt}
\begin{picture}(40,85)(0,0) \put(-7,80){${\bf M}:$} \put(20,0){\circle*{3}}
\put(20,80){\circle*{3}} \put(20,40){\circle*{3}} \put(0,20){\circle*{3}}
\put(-20,40){\circle*{3}} \put(20,60){\circle*{3}} \put(20,20){\circle*{3}}
\put(0,60){\circle*{3}} \put(60,40){\circle*{3}} \put(40,20){\circle*{3}}
\put(40,60){\circle*{3}} \put(18,-9){$0=z^{\sim }$}
\put(18,83){$1$} \put(25,37){$t=t^{\prime }$}
\put(23,57){$z^{\prime }$} \put(23,17){$z$} \put(-53,38){$u^{\sim }=a$}
\put(-8,58){$v^{\prime }$} \put(63,38){$a^{\prime }=v^{\sim }$} \put(43,58){$u^{\prime
}$} \put(42,14){$u$} \put(-7,14){$v$} \put(20,0){\line(-1,1){40}}
\put(20,0){\line(1,1){40}} \put(20,80){\line(-1,-1){40}}
\put(20,80){\line(1,-1){40}} \put(20,0){\line(0,1){80}}
\put(0,20){\line(1,1){40}} \put(40,20){\line(-1,1){40}} \end{picture}\end{tabular}\end{center}\vspace*{5pt}\end{example}

Recall that, if ${\bf A}$ is an algebra from a double--pointed variety $\V $ with constants $0,1$ then an element $e\in A$ is {\em central} iff $Cg_{\V ,{\bf A}}(e,0)$ and $Cg_{\V ,{\bf A}}(e,1)$ are complementary factor congruences of ${\bf A}$. Let us denote by $C({\bf L})$ the set of the central elements of any \PBZ --lattice ${\bf L}$.

\begin{lemma}{\rm \cite{rgcmfp}} For any \PBZ --lattice ${\bf L}$, $C({\bf L})=\{a\in S({\bf L}):(\forall \, b\in L)\, ((a\wedge b)^{\sim }=a^{\sim }\vee b^{\sim },(a^{\prime }\wedge b)^{\sim }=a^{\prime \sim }\vee b^{\sim },b=(a\wedge b)\vee (a^{\prime }\wedge b))\}=\{a\in S({\bf L}):{\bf L}\vDash _{L,\{a^{\prime }\}}SDM,{\bf L}\vDash _{L,\{a\}}\{SDM,J0\}\}$.\label{central}\end{lemma}

\begin{lemma} Any PBZ$^{\ast}$ --lattice that satisfies $J2$, $S2$ and $S3$ and does not belong to $\mathbb{OML}\boxplus\mathbb{AOL}$ is directly reducible.\label{j2s2}\end{lemma}

\begin{proof} Part of this argument has been applied in a result in \cite{rgcmfp} in a slightly different context; for the sake of completeness, we provide a complete proof of the present lemma.

Let $\mathbf{L}$ be a directly irreducible PBZ$^{\ast}$ --lattice that satisfies $J2$, $S2$ and $S3$. Then the only central elements of $\mathbf{L}$ are $0$ and $1$. We want to show that $S(\mathbf{L})\cup T\left(  {\mathbf{L}}\right) =L$.

Let $x\in L$. Then the element $(x\wedge x^{\prime})^{\sim}$ is sharp and, by J2, we have that $\mathbf{L}\vDash_{L,\{(x\wedge x^{\prime})^{\sim}\}}u\approx(v\wedge u)\vee(v^{\prime}\wedge u)$. Furthermore, by S2, we have that $\mathbf{L}\vDash_{L,\{(x\wedge x^{\prime})^{\sim}\}}SDM$, while S3 gives us
$\mathbf{L}\vDash_{L,\{\Diamond(x\wedge x^{\prime})\}}SDM$. By Lemma \ref{central}, $(x\wedge x^{\prime})^{\sim}$ is a central element of $\mathbf{L}$, so $\Diamond\left(  x\wedge x^{\prime}\right)  $ is central and thus $\Diamond\left(  x\wedge x^{\prime}\right)  \in\left\{  0,1\right\}  $. If $\Diamond\left(  x\wedge x^{\prime}\right)  =0$, then $x\wedge x^{\prime}\leq\Diamond\left(  x\wedge x^{\prime}\right)  =0$, whence $x\in S(\mathbf{L})$. If $\Diamond\left(
x\wedge x^{\prime}\right)  =1$, then ${x{^{\sim}\leq}x{^{\sim}\vee\square x=}}\left(  x\wedge x^{\prime}\right)
{^{\sim}=0}$, so that $x\in T\left(  {\mathbf{L}}\right) $.

Our claim is therefore settled, and by Theorem \ref{charg} $\mathbf{L}$ belongs to $\mathbb{OML} \boxplus\mathbb{AOL}$.\end{proof}

\begin{proposition} All members of $V(\mathbb{OML}\boxplus\mathbb{AOL})\setminus (\OML \boxplus \AOL )$ are directly reducible. In particular, all subdirectly irreducible members of $V(\mathbb{OML} \boxplus\mathbb{AOL} )$ belong to $\mathbb{OML} \boxplus\mathbb{AOL} $.\label{sdirirred}\end{proposition}

\begin{proof} By Lemma \ref{j2s2} and Corollary \ref{axvarhsum}, which ensures us that $V(\mathbb{OML} \boxplus\mathbb{AOL} )\vDash\{J2,S2,S3\}$.\end{proof}

\begin{theorem} $\{{\bf L}\in V(\mathbb{OML}\boxplus\mathbb{AOL}):{\bf L}\vDash WSDM\}=\mathbb{OML}\vee V(\mathbb{AOL})$.\label{axwsdm}\end{theorem}

\begin{proof}
We have $\mathbb{OML} \vee V(\mathbb{AOL} )\subseteq\{\mathbf{L}\in
V(\mathbb{OML} \boxplus\mathbb{AOL} ):\mathbf{L}\vDash WSDM\}$. Now let $\mathbf{L}\in V(\mathbb{OML}\boxplus\mathbb{AOL})$ be such that
$\mathbf{L}$ satisfies WSDM and is subdirectly irreducible. Then, by
Proposition \ref{sdirirred} we have $\mathbf{L}\in\mathbb{OML}\boxplus
\mathbb{AOL}$, and by Corollary \ref{when1ok}, $\mathbf{L}\in\mathbb{OML}\cup\mathbb{AOL}\subseteq\mathbb{OML}\vee V(\mathbb{AOL})$, which completes the proof.
\end{proof}

\begin{theorem}
\label{cacchiocacchio} $\{{\bf L}\in \PBZs :{\bf L}\vDash \{J2,S2,S3\}\}=V(\mathbb{OML} \boxplus\mathbb{AOL} )$.
\end{theorem}

\begin{proof}
By Corollary \ref{axvarhsum}, all members of $\mathbb{OML} \boxplus \mathbb{AOL} $ satisfy the identities S2, S3 and J2, hence the right-to-left inclusion is established. Lemma \ref{j2s2} gives us the converse inclusion.\end{proof}

Note that, since $V(\OML \boxplus \AOL )\vDash \{J1,S1\}$ according to Corollary \ref{axvarhsum}, Theorem \ref{cacchiocacchio} shows that $\{J2,S2,S3\}\vDash \{J1,S1\}$. By Corollary \ref{axvarhsum}, $\OML \boxplus V
(\AOL )$ satisfies $J1$, $S1$, $S2$ and $S3$ and fails $J2$, so $\{J1,S1,S2,S3\}\nvDash J2$.

Theorem \ref{cacchiocacchio} and the fact that WSDM implies S2 and S3 give us a new proof for the following result from \cite{rgcmfp}:

\begin{corollary}
\label{brutus} $\{{\bf L}\in \PBZs :{\bf L}\vDash \{J2,WSDM\}\}=\mathbb{OML}{\vee V}\left(  \mathbb{AOL}\right)$.\end{corollary}

\section*{Acknowledgements}

We thank Davide Fazio and Antonio Ledda for insightful discussions on the topics of this paper.

This work was supported by the research grant \textquotedblleft Propriet\`{a}
d`Ordine Nella\linebreak Semantica Algebrica delle Logiche
Non--classiche\textquotedblright, Universit\`{a} degli Studi di Cagliari,
Regione Autonoma della Sardegna, L. R. 7/2007, n. 7, 2015, CUP:
F72F16002920002. Moreover, all authors gratefully acknowledge the European
Union's Horizon 2020 research and innovation programme under the Marie
Sklodowska-Curie grant agreement No 689176 (project \textquotedblleft Syntax
Meets Semantics: Methods, Interactions, and Connections in Substructural
logics\textquotedblright)

\end{document}